%% file: mainfile.tex
\documentclass[12pt]{article}
\input{preamble}

\begin{document}

\maketitle

\begin{abstract}
    Classical persistent homology provides a summary of the multiscale topology of a filtration of a space. These summaries are not naturally valued in a vector space, limiting their direct use in standard machine learning pipelines. To address this issue, Bubenik introduced persistence landscapes, vector-valued invariants of persistence modules, which operate by conveniently encoding the behavior of rank functions of modules along the parameter axis. When passing to multiparameter persistent homology, changes in the rank functions vary by different amounts in different directions. In this paper, we introduce the landscape interval, a subset of the parameter space of a multiparameter persistence module which records changes in the rank function in all directions simultaneously. This geometric perspective leads to two novel invariants which serve as notions of persistence landscapes for multiparameter modules.  We show that our invariants generalize existing landscape-type invariants: in particular, they recover Bubenik's landscapes in the one-parameter setting and allow us to obtain the multiparameter landscapes introduced by Vipond. For our main theoretical contributions, we establish two exact identities connecting landscapes with erosion distances,  prove that our invariants are stable under perturbations, and refine the interleaving-stability bounds for the existing landscape-type invariants. We also show that on a certain class of multiparameter persistence modules, the full invariants can be recovered from finite information about the module. On the practical side, we introduce a robust numerical framework for extracting our invariants from real data and for incorporating them into standard machine learning pipelines. The framework is illustrated on surface mesh and medical image data. 
\end{abstract}

\input{sections/introduction}

\input{sections/preliminaries_and_related_work}
\input{sections/generalized_persistence_landscapes}

\input{sections/distances_and_stability}

\input{sections/finite_characterization_of_landscape_intervals}

\input{sections/finite_direction_approximation_and_curve_structure}

\input{sections/applications_and_experiments}

\section*{Code availability}

Code for reproducing the experiments is available on
\href{https://github.com/pfanghaha/landscape-interval}{GitHub}. 

\clearpage
\appendix
\input{sections/appendix}

\clearpage

\bibliographystyle{plain}
\bibliography{references}

\end{document}

%% file: preamble.tex
\usepackage{amsmath,amssymb,amsthm,mathtools}
\usepackage[margin=1in]{geometry}
\usepackage{enumitem}
\usepackage{dsfont}
\usepackage{tikz-cd}
\usepackage{booktabs}
\usepackage{adjustbox}
\usepackage{caption}
\usepackage{graphicx}
\usepackage{xcolor}
\usepackage{needspace}
\usepackage{comment}
\usepackage[normalem]{ulem}

\usepackage{hyperref}
\hypersetup{
    colorlinks=true,
    linkcolor=black,
    citecolor=black,
    urlcolor=black,
    pdfborder={0 0 0}
}

\title{Landscape Intervals: A Geometric Framework for Multiparameter Persistence Landscapes}

\author{
Pan Fang\thanks{Email: \texttt{pfang@fsu.edu}}
\quad
Washington Mio
\quad
Tom Needham
\\[0.6em]
{\small Department of Mathematics}\\
{\small Florida State University}
}

\date{}

\theoremstyle{plain}
\newtheorem{theorem}{Theorem}[section]
\newtheorem{lemma}[theorem]{Lemma}
\newtheorem{corollary}[theorem]{Corollary}
\newtheorem{proposition}[theorem]{Proposition}

\theoremstyle{definition}
\newtheorem{definition}[theorem]{Definition}
\newtheorem{example}[theorem]{Example}

\theoremstyle{remark}

\newcommand{\R}{\mathbb{R}}
\newcommand{\N}{\mathbb{N}}
\newcommand{\Z}{\mathbb{Z}}
\newcommand{\F}{\mathbb{F}}

\newcommand{\Vect}{\mathbf{Vect}}
\newcommand{\rk}{\operatorname{rk}}

\newcommand{\I}{\mathcal{I}}

\newcommand{\LI}{\mathrm{LI}}
\newcommand{\UL}{\mathrm{UL}}
\newcommand{\DL}{\mathrm{DL}}
\newcommand{\IL}{\mathrm{IL}}
\newcommand{\WL}{\mathrm{WL}}

\usepackage{microtype}

%% file: sections/introduction.tex
\section{Introduction}
\label{sec:introduction}

Classical persistent homology studies how topological features, such as connected components and loops, evolve in an increasing family of spaces as a scale or threshold varies. This paper considers the multiparameter setting, where one is interested in feature evolution along multiple simultaneously changing scales. Our primary interest is in the extraction of vector-valued features from multiparameter persistent homology modules, which can then be incorporated into machine learning pipelines for downstream data analysis tasks. Such vectorizations have been studied extensively in both the single parameter~\cite{JMLR:v16:bubenik15a,adams2017persistence,ali2023survey} and multiparameter~\cite{JMLR:v21:Vipond,carriere2020multiparameter} settings, but the latter is more challenging due to the lack of the barcode representations which are present in the classical theory~\cite{CrawleyBoevey2015,botnan2023introductionmultiparameterpersistence}. As we describe below, our approach is inspired by the persistence landscape framework of Bubenik. Throughout this paper, we assume the reader's familiarity with standard concepts of persistent homology theory, e.g., persistence modules and interleaving distances. For introductions to these ideas, we refer to \cite{EdelsbrunnerLetscherZomorodian2002,ZomorodianCarlsson2005,OtterEtAl2017,botnan2023introductionmultiparameterpersistence}.

Pointwise finite-dimensional persistence modules indexed by $\R$ are determined up to isomorphism by their barcodes. Since barcodes do not naturally form a vector space, statistical and machine-learning methods based on linear operations and norms are not directly applicable to them. Representing barcodes by functions in a normed vector space provides access to these methods. Bubenik's persistence landscape provides such a functional representation \cite{JMLR:v16:bubenik15a}. For a one-parameter persistence module $M$ and $a\le b$, write $\rk_M(a,b)$ for the rank of the structure map $M(a)\to M(b)$. For $t\in\R$ and a positive integer $k$, the landscape value is defined by
\[
    \lambda_M(k,t)
    \coloneqq
    \sup\left\{
        h\ge 0:
        \rk_M(t-h,t+h)\ge k
    \right\},
\]
where the supremum of the empty set is taken to be $0$. The rank condition requires at least $k$ independent classes to persist across the interval $[t-h,t+h]$, whose half-width is $h$. Thus, $\lambda_M(k,t)$ records how far this persistence extends symmetrically about $t$.

Multiparameter persistence modules arise when data are studied simultaneously with respect to several scales, measurements, or filtering functions \cite{CarlssonZomorodian2009DCG,botnan2023introductionmultiparameterpersistence}. The parameter is then a vector in $\R^n$, which we endow with the product poset structure $\preceq$, that is, $a\preceq b$ if and only if $a_i\le b_i$ for every $i$. For $n\ge2$, these modules do not in general admit a complete description by barcodes \cite{CarlssonZomorodian2009DCG}. For an $n$-parameter persistence module $M$, a basic summary is the rank invariant, which assigns to each pair $a\preceq b$ the rank $\rk_M(a,b)$ of the structure map $M(a)\to M(b)$. This rank records how many independent homological classes persist from $a$ to $b$.

Vipond generalized Bubenik's construction to multiparameter persistence \cite[Definition~19]{JMLR:v21:Vipond}. We call this construction the \emph{uniform multiparameter persistence landscape}. For an $n$-parameter persistence module $M$, it can be written as
\[
    \UL_M(k,t)
    =
    \sup\left\{
        \varepsilon\ge 0:
        \rk_M(t-\varepsilon\mathbf{1}_n,t+\varepsilon\mathbf{1}_n)\ge k
    \right\},
\]
where $\mathbf{1}_n=(1,\ldots,1)$. Thus, the uniform landscape measures symmetric persistence when every parameter changes by the same amount.

To allow different relative changes in the parameters, Vipond introduced weighted landscapes \cite[Definition~25]{JMLR:v21:Vipond}. For a weight vector $w\in\R_{>0}^n$ with $\lVert w\rVert_\infty=1$, write $w^{-1}=(w_1^{-1},\ldots,w_n^{-1})$. The weighted landscape can be written as
\[
    \WL_M(w,k,t)
    =
    \sup\left\{
        \varepsilon\ge 0:
        \rk_M(t-\varepsilon w^{-1},t+\varepsilon w^{-1})\ge k
    \right\}.
\]
Here, $\varepsilon/w_i$ is the half-width in the $i$th parameter, so the weights specify the relative amounts by which the parameters vary. Taking $w=\mathbf{1}_n$ gives the uniform landscape. These formulas are equivalent to the original definitions recalled in Section~\ref{sec:preliminaries}.

Our starting point is to ask for the entire set of symmetric displacements over which a given number of classes persist. That is, we wish to define an object which encodes the answer to the question: for fixed $t\in\R^n$ and $k\in\N_{>0}$, which vectors $h\in\R_{\ge 0}^n$ allow at least $k$ independent homological classes to persist from $t-h$ to $t+h$?

Vipond's rescaled multiparameter rank function $r_M(t,h)=\rk_M(t-h,t+h)$ records the rank between these symmetric endpoints \cite[Definition~18]{JMLR:v21:Vipond}. We organize these same rank values into the set
\[
    \I_M^k(t)
    \coloneqq
    \left\{
        h\in\R_{\ge 0}^n:
        r_M(t,h)\ge k
    \right\},
\]
which we call the \emph{landscape interval} of $M$ at $t$ (Definition~\ref{def:landscape-interval}). Taking this set as the primary object allows us to study the geometry of local persistence, including how far it extends in different directions. Let us record some immediate observations about the landscape interval:
\begin{itemize}
    \item When $n=1$, $\I_M^k(t)$ is either empty or an interval beginning at $0$, and its supremum is the one-parameter persistence landscape value $\lambda_M(k,t)$. Thus landscape intervals can be viewed as a generalization of one parameter persistence landscapes.
    \item In every dimension, the complete family of landscape intervals determines the rank invariant. Indeed, for any $a\preceq b$, taking $t=(a+b)/2$ and $h=(b-a)/2$ gives $\rk_M(a,b)\ge k$ if and only if $h\in\I_M^k(t)$. This mirrors a result of Vipond: \cite[Proposition~36]{JMLR:v21:Vipond} shows that the complete family obtained by varying the weight vector  determines the rank invariant for finitely presented modules.
\end{itemize} 
These properties suggest that the collection of all landscape intervals is a powerful invariant of a multiparameter persistence module. However, it is rather unwieldy to deal with such a collection in practical applications. As such, we introduce two invariants of the landscape interval, which can themselves be viewed as multiparameter persistence landscapes.

The \emph{directional multiparameter persistence landscape} describes the geometry of a landscape interval by recording how far it extends in each direction (Definition~\ref{def:directional-multiparameter-persistence-landscape}). For a nonnegative direction vector $v\in\R_{\ge 0}^n$ normalized by $\lVert v\rVert_\infty=1$, define
\[
    \DL_M(v,k,t)
    \coloneqq
    \sup\left\{
        \varepsilon\ge 0:
        \varepsilon v\in\I_M^k(t)
    \right\}.
\]
When $n=1$, this gives $\DL_M(1,k,t)=\lambda_M(k,t)$, while in every dimension the diagonal direction gives $\DL_M(\mathbf{1}_n,k,t)=\UL_M(k,t)$. For strictly positive directions, these values agree, up to rescaling, with Vipond's weighted landscapes (Proposition~\ref{prop:directional-weighted-relation}).

The \emph{integrated multiparameter persistence landscape} generalizes the interpretation of a one-parameter landscape value as the length of its landscape interval (Definition~\ref{def:integrated-multiparameter-persistence-landscape}). Using $n$-dimensional Lebesgue measure $\mathcal{L}^n$, we define, for $r>0$,
\[
    \IL_M(r,k,t)
    \coloneqq
    \frac{
        \mathcal{L}^n\bigl([0,r]^n\cap\I_M^k(t)\bigr)
    }{
        r^{n-1}
    },
\]
and set $\IL_M(0,k,t)=0$. The numerator measures the portion of the landscape interval contained in $[0,r]^n$, and division by $r^{n-1}$ gives a quantity with units of length. When $n=1$, we have $\IL_M(r,k,t)=\min\{r,\lambda_M(k,t)\}$, so the one-parameter landscape is recovered as $r\to\infty$. In every dimension, the values across all scales also determine the uniform landscape value $\UL_M(k,t)$, as shown in Proposition~\ref{prop:uniform-recovered-from-integrated}.

\paragraph{Main contributions.}
We treat landscape intervals as the underlying geometric objects from which these landscape representations are obtained. The results of this paper are concerned with properties of landscape intervals, comparisons between our multiparameter landscapes and other multiparameter invariants in the literature, stability properties of our constructions, and numerical implementation. Our main contributions are summarized as follows.

\begin{enumerate}
    \item \emph{Landscape intervals and two geometric generalizations.} We introduce landscape intervals and define the directional and integrated multiparameter persistence landscapes using their radial extent and Lebesgue measure, respectively (Definitions~\ref{def:landscape-interval}, \ref{def:directional-multiparameter-persistence-landscape}, and~\ref{def:integrated-multiparameter-persistence-landscape}). We relate these constructions to Bubenik's one-parameter landscape and Vipond's uniform and weighted landscapes, and prove that the integrated landscape determines the uniform landscape pointwise (Propositions~\ref{prop:directional-weighted-relation} and~\ref{prop:uniform-recovered-from-integrated}). We also give an integral formula expressing the integrated landscape in terms of directional values (Lemma~\ref{lem:integrated-radial-representation}).

    \item \emph{Sharper metric and stability results.}
    We establish quantitative comparisons between landscape distances (Corollary~\ref{cor:directional-weighted-comparison}, Theorem~\ref{thm:integrated-directional-comparison}, and Corollary~\ref{cor:integrated-weighted-comparison}). We define a distance between landscape-interval families and prove that it equals the rectangle-restricted erosion distance (Theorem~\ref{thm:landscape-interval-rectangle-erosion}). We also prove that the supremum distance between uniform landscapes equals the cube-restricted erosion distance in every parameter dimension (Theorem~\ref{thm:uniform-erosion-distance}). The supremum distance between weighted-landscape families lies between the cube- and rectangle-restricted erosion distances. We obtain corresponding lower and upper bounds for directional landscapes when every direction coordinate is bounded below by a fixed positive constant, and an upper erosion bound for the integrated landscape (Theorems~\ref{thm:directional-erosion-bound}--\ref{thm:integrated-erosion-bound}). Because these erosion distances are no greater than the interleaving distance, our bounds recover Vipond's interleaving-stability results for the uniform and weighted landscapes \cite{JMLR:v21:Vipond} and yield sharper estimates (Corollary~\ref{cor:landscape-interleaving-stability}). Finally, we show that no uniform Lipschitz bound with respect to the interleaving distance holds in general for the unrestricted directional landscape (Example~\ref{ex:directional-interleaving-instability}).

    \item \emph{Finite reconstruction and curve structure.} For bounded-below, rectangular-constructible persistence modules, we prove that the closure of each nonempty landscape interval is determined by finitely many directional values (Theorem~\ref{thm:finite-maximal-direction-reconstruction}). Our finite-reconstruction question was partly inspired by finite-direction reconstruction results for the persistent homology transform \cite{CurryMukherjeeTurner2022}. The finite representation of landscape intervals gives an explicit piecewise description of the associated integrated landscape curves (Theorem~\ref{thm:piecewise-form-integrated-curves}). In two parameters, every nonzero such curve is nondecreasing until it reaches its maximum, may remain constant there over an interval, and is nonincreasing thereafter (Theorem~\ref{thm:two-parameter-curve-shape}).

    \item \emph{Computation and applications.} We develop a computational framework based on finitely many directional values. For nonempty bounded landscape intervals, it provides inner and outer approximations of their closures and lower and upper bounds for the associated integrated landscape curves (Proposition~\ref{prop:finite-direction-bounds} and Corollary~\ref{cor:finite-direction-curve-bounds}). When the sampled direction sets are nested and their union is dense, these curve bounds converge at every scale (Subsection~\ref{subsec:finite-direction-approximations}). The diagonal sample gives Vipond's uniform landscape, while strictly positive directional samples give the corresponding weighted values after rescaling. We illustrate the framework through synthetic point-cloud examples and evaluate uniform and integrated landscape representations in shape and medical-image classification (Section~\ref{sec:applications-and-experiments}).
    
\end{enumerate}

\paragraph{Organization of the paper.}
Section~\ref{sec:preliminaries} reviews persistence modules, rank invariants, and existing persistence landscapes. Section~\ref{sec:Generalized-landscapes} introduces landscape intervals and the directional and integrated landscapes. Section~\ref{sec:stability} establishes relationships among these landscapes and studies their metric and stability properties. Section~\ref{sec:finite-characterization} develops finite reconstruction under rectangular constructibility, and Section~\ref{sec:finite-directional-approximation} studies finite-direction approximation and the structure of integrated landscape curves. Section~\ref{sec:applications-and-experiments} presents synthetic examples and classification experiments.

\paragraph{Use of AI tools.}
The project began in August 2024. Following completion of the initial manuscript draft, ChatGPT was used to improve the language, clarity, and organization of the draft. ChatGPT was also used to refine existing research code. For the classification experiments, the first author provided the original code draft and detailed instructions, based on which ChatGPT generated the implementation. The authors take full responsibility for the final manuscript, code, and reported results.

%% file: sections/preliminaries_and_related_work.tex
\section{Preliminaries and Related Work}
\label{sec:preliminaries}

This section collects the background and related work needed for the developments that follow. We first recall intervals in posets, persistence modules indexed by subposets of $\R^n$, and their construction from filtered simplicial complexes. After briefly recalling the role of barcodes in $1$-parameter persistence, we review the rank invariant, its generalization to intervals, and representative rank-based summaries for multiparameter persistence modules. We conclude by introducing persistence landscapes and their uniform and weighted multiparameter versions.

\begin{definition}
    For a poset $(P, \preceq)$, a nonempty subset $I \subseteq P$ is an \emph{interval} in $P$ if it satisfies the following conditions:
    \begin{enumerate}
        \item If $s, t \in I$ and $m \in P$ satisfy
        $s \preceq m \preceq t$, then $m \in I$.
        
        \item If $s, t \in I$, then there exist points
        $s = m_0, \ldots, m_\ell = t$ in $I$ such that, for every
        $0 \le i < \ell$, either $m_i \preceq m_{i+1}$ or
        $m_i \succeq m_{i+1}$.
    \end{enumerate}
\end{definition}

A poset $(P, \preceq)$ can be viewed as a category whose objects are the elements of $P$ and in which there is a unique morphism $s \to t$ if and only if $s \preceq t$. Let $\Vect$ denote the category of vector spaces over a field $\F$ and $\F$-linear maps. We write $\Vect^P$ for the category whose objects are functors from $P$ to $\Vect$ and whose morphisms are natural transformations.

We equip $\R^n$ with the product order $\preceq$. Thus, for
$x, y \in \R^n$,
\[
    x \preceq y
    \quad\Longleftrightarrow\quad
    x_i \le y_i
    \text{ for every } i = 1, \ldots, n.
\]
Every subset $P \subseteq \R^n$ is regarded as a subposet of $(\R^n, \preceq)$ with the induced order.

The algebraic theory of multiparameter persistence modules was introduced in \cite{CarlssonZomorodian2009DCG}. We use the standard
functorial formulation reviewed in \cite{botnan2023introductionmultiparameterpersistence}.

\begin{definition}[Persistence module]
    Let $P$ be a nonempty poset. A \emph{persistence module indexed by $P$} is a functor
    \[
        M \colon P \longrightarrow \Vect.
    \]
    Equivalently, $M$ is an object of $\Vect^P$. When $P = P_1 \times \cdots \times P_n$, where each $P_i \subseteq \R$ is nonempty and $P$ is equipped with the product order, we call $M$ an \emph{$n$-parameter persistence module}. The module $M$ is \emph{pointwise finite-dimensional} if $\dim_{\F} M(p) < \infty$ for every $p \in P$.
\end{definition}

For $a, b \in P$ with $a \preceq b$, we write $M_a^b \colon M(a) \longrightarrow M(b)$ for the structure map assigned by $M$ to the unique morphism $a \to b$ in $P$. For an interval $I \subseteq P$, the \emph{interval module} $\mathds{1}_I \in \Vect^P$ is defined by
\[
    \mathds{1}_I(p)
    =
    \begin{cases}
        \F, & p \in I,\\
        0,  & p \notin I,
    \end{cases}
\]
with identity structure maps whenever both indices lie in $I$ and zero structure maps otherwise.

In applications, persistence modules indexed by subposets of $\R^n$ are often constructed from finite point clouds using the following standard constructions.

\begin{example}
\label{ex:multiparameter-module}
    Let $(X, d_X)$ be a finite metric space. For $r \ge 0$, the \emph{Vietoris--Rips complex} of $X$ at scale $r$, denoted by $\mathrm{VR}(X)_r$, is the simplicial complex whose simplices are the subsets $\sigma \subseteq X$ satisfying
    \[
        d_X(x, y) \le r
        \quad
        \text{for all } x, y \in \sigma.
    \]
    The inclusions $\mathrm{VR}(X)_r \hookrightarrow \mathrm{VR}(X)_{r'}$ for $r \le r'$ define a functor
    \[
        \mathrm{VR}(X)_{\bullet}
        \colon
        \R_{\ge 0}
        \longrightarrow
        \mathbf{Simp},
    \]
    where $\mathbf{Simp}$ denotes the category of simplicial complexes and simplicial maps. For $p \ge 0$, applying the $p$-th simplicial homology functor with coefficients in $\F$ gives a $1$-parameter persistence module
    \[
        H_p\bigl(\mathrm{VR}(X)_{\bullet}; \F\bigr)
        \colon
        \R_{\ge 0}
        \longrightarrow
        \Vect.
    \]

    Now let $\varphi \colon X \to \R$ be a function, and for each $a \in \R$, let $X_a = \varphi^{-1}((-\infty, a])$. The assignment
    \[
        (a, r)
        \longmapsto
        \mathrm{VR}(X_a)_r
    \]
    defines a functor
    \[
        \mathrm{VR}(X)^{\varphi}_{\bullet,\bullet}
        \colon
        \R \times \R_{\ge 0}
        \longrightarrow
        \mathbf{Simp}.
    \]
    Applying the $p$-th simplicial homology functor gives a $2$-parameter persistence module
    \[
        H_p\bigl(
            \mathrm{VR}(X)^{\varphi}_{\bullet,\bullet};
            \F
        \bigr)
        \colon
        \R \times \R_{\ge 0}
        \longrightarrow
        \Vect.
    \]
\end{example}

Applying homology to a filtration, as in Example~\ref{ex:multiparameter-module}, produces its \emph{persistent homology module}. When several filtration parameters are present, this construction is called \emph{multiparameter persistent homology}.

By the interval-decomposition theorem, every pointwise finite-dimensional persistence module indexed by a subposet of $\R$ decomposes as a direct sum of interval modules, uniquely up to isomorphism and reordering. The resulting multiset of intervals is called its \emph{barcode} and records the parameter ranges over which homological features persist \cite{CrawleyBoevey2015}. In contrast, for $n \ge 2$, pointwise finite-dimensional persistence modules need not be interval decomposable and therefore cannot in general be classified by barcodes \cite{CarlssonZomorodian2009DCG}. One commonly used invariant in this setting is the rank invariant.

\begin{definition}[Rank invariant]
    Let $M \in \Vect^P$, and let $\overline{\Z}_{\ge 0} \coloneqq \Z_{\ge 0} \cup \{\infty\}$. Here and throughout, the rank of a linear map with infinite-dimensional image is understood to be $\infty$. The \emph{rank invariant} of $M$ is the function
    \[
        \rk_M
        \colon
        \{(a, b) \in P \times P : a \preceq b\}
        \longrightarrow
        \overline{\Z}_{\ge 0}
    \]
    defined by
    \[
        \rk_M(a, b)
        \coloneqq
        \operatorname{rank}(M_a^b)
        =
        \dim_{\F}\operatorname{Im}(M_a^b).
    \]
\end{definition}

The rank invariant records the ranks of all structure maps between comparable indices. The generalized rank extends this construction from comparable pairs to intervals of the indexing poset \cite{KimMemoli2021JACT}. The categorical definitions of limits and colimits, together with the construction of the canonical map used below, are recalled in Appendix~\ref{app:categorical-background}.

For an interval $I \subseteq P$, let
\[
    \psi_{M|_I}
    \colon
    \varprojlim_{p \in I} M(p)
    \longrightarrow
    \varinjlim_{p \in I} M(p)
\]
denote the canonical limit-to-colimit map associated with the restriction $M|_I$. The \emph{generalized rank} of $M$ over $I$ is
\[
    \rk_M^{\mathrm{gen}}(I)
    \coloneqq
    \operatorname{rank}\bigl(\psi_{M|_I}\bigr).
\]
If $I$ has a minimum $a$ and a maximum $b$, then Lemma~\ref{lem:generalized-rank-extrema} gives
\[
    \rk_M^{\mathrm{gen}}(I)
    =
    \rk_M(a, b).
\]

For finite intervals in $\Z^2$, generalized ranks can be computed from suitable boundary-tracing zigzag restrictions \cite{DeyKimMemoli2024}. The connection between ranks and barcodes has motivated several summaries of multiparameter persistence modules. Representative examples include fibered barcodes and matching-distance methods based on one-dimensional affine restrictions \cite{CerriEtAl2013,LesnickWright2015}; generalized persistence diagrams obtained through M\"obius inversion \cite{Patel2018GPD,KimMemoli2021JACT}; signed barcodes arising from rank decompositions \cite{Botnan2024SignedBarcodes,BotnanOppermannOudotScoccola2022Bottleneck}; and GRIL, a vectorization for $2$-parameter persistence modules based on generalized ranks over families of worm-shaped intervals \cite{XinMukherjeeSamagaDey2023GRIL}. A differentiable extension of GRIL for end-to-end learning was subsequently developed in \cite{MukherjeeEtAl2026DGRIL}. Related data-analytic representations include multiparameter persistence kernels and images, stable representations based on module decompositions, and vectorizations of signed barcodes as measures \cite{CorbetEtAl2019Kernel,CarriereBlumberg2020,LoiseauxCarriereBlumberg2023,LoiseauxEtAl2023SignedMeasures}. Our construction instead retains the full set of nonnegative symmetric displacements satisfying a fixed rank threshold, which Section~\ref{sec:Generalized-landscapes} defines as a landscape interval.

For pointwise finite-dimensional $1$-parameter persistence modules indexed by $\R$, the rank invariant and the barcode determine each other. The persistence landscape was introduced in \cite{JMLR:v16:bubenik15a} as a functional summary constructed from the rank invariant, or equivalently from the barcode. A stable extension to multiparameter persistence modules was introduced in \cite{JMLR:v21:Vipond}, yielding a functional summary derived from the multiparameter rank invariant.

Relations between rank information and one-parameter persistence landscapes were further studied in \cite{BetthauserBubenikEdwards2022}. The multiparameter persistence landscape has been used in spatial biological data analysis \cite{VipondEtAl2021Tumors} and studied statistically through a functional central limit theorem and bootstrap confidence bands \cite{GarciaRedondoMonodWang2026}. Landscape constructions have also been developed for labeled persistent homology and for extended zigzag modules arising from spatiotemporal data \cite{FuEtAl2025GeneralizedLandscapes,FlammerHuper2026}. These works develop landscapes for labeled and spatiotemporal data, while our focus is the geometry of landscape intervals in $\R^n$.

Although persistence landscapes are commonly defined for pointwise finite-dimensional modules, their rank-based definitions extend to arbitrary persistence modules under the extended-rank convention adopted above. In the following definitions, we take the supremum of the empty set to be $0$.

\begin{definition}[One-parameter persistence landscape~\cite{JMLR:v16:bubenik15a}]
     Under the extended-rank convention adopted above, the \emph{persistence landscape} of a persistence module $M \in \Vect^{\R}$ is the function
    \[
        \lambda_M
        \colon
        \N_{>0} \times \R
        \longrightarrow
        [0,\infty]
    \]
    defined by
    \[
        \lambda_M(k, t)
        \coloneqq
        \sup
        \left\{
            h \ge 0 :
            \rk_M(t - h, t + h) \ge k
        \right\}.
    \]
\end{definition}

For persistence modules arising from finite filtrations, the function $t\mapsto\lambda_M(k,t)$ is piecewise linear for every fixed $k\in\N_{>0}$.

Vipond introduced the multiparameter persistence landscape in \cite{JMLR:v21:Vipond}. To distinguish this construction from the generalizations developed later in this paper, we call it the \emph{uniform multiparameter persistence landscape}.

\begin{definition}[Uniform multiparameter persistence landscape]
\label{def:uniform-multiparameter-persistence-landscape}
    Let $M \in \Vect^{\R^n}$ be a persistence module. The \emph{uniform multiparameter persistence landscape}, or simply \emph{uniform landscape}, of $M$ is the function
    \[
        \UL_M
        \colon
        \N_{>0} \times \R^n
        \longrightarrow
        [0,\infty]
    \]
    defined by
    \[
        \UL_M(k, t)
        \coloneqq
        \sup
        \left\{
            \epsilon \ge 0 :
            \rk_M(t - h, t + h) \ge k
            \text{ for every } h \in \R_{\ge 0}^n
            \text{ with } \lVert h \rVert_\infty \le \epsilon
        \right\}.
    \]
\end{definition}

When $n = 1$, $\UL_M$ coincides with the one-parameter persistence landscape. For $n \ge 2$, $\UL_M(k, t)$ records the supremal uniform scale over which the rank remains at least $k$ in every nonnegative displacement direction around $t$.

To account for differences in the relative scales of the parameters, Vipond introduced weighted multiparameter persistence landscapes \cite[Definition~25]{JMLR:v21:Vipond}. Let
\[
    \mathcal{W}_n
    \coloneqq
    \left\{
        w \in \R_{>0}^n :
        \lVert w \rVert_\infty = 1
    \right\}
\]
denote the space of admissible weight vectors. For $w \in \mathcal{W}_n$, define the $w$-weighted infinity norm by
\[
    \lVert h \rVert_{\infty,w}
    \coloneqq
    \left\lVert
        (w_i h_i)_{i=1}^n
    \right\rVert_\infty.
\]

Following Vipond's weighted landscape space \cite[Definition~35]{JMLR:v21:Vipond}, we regard the complete family over all admissible weights as a single function whose first argument is the weight vector.

\begin{definition}[Weighted multiparameter persistence landscape]
\label{def:weighted-multiparameter-persistence-landscape}
    Let $M \in \Vect^{\R^n}$ be a persistence module. The \emph{weighted multiparameter persistence landscape}, or simply \emph{weighted landscape}, of $M$ is the function
    \[
        \WL_M
        \colon
        \mathcal{W}_n \times \N_{>0} \times \R^n
        \longrightarrow
        [0,\infty]
    \]
    defined by
    \[
        \WL_M(w, k, t)
        \coloneqq
        \sup
        \left\{
            \epsilon \ge 0 :
            \rk_M(t - h, t + h) \ge k
            \text{ for every } h \in \R_{\ge 0}^n
            \text{ with }
            \lVert h \rVert_{\infty,w} \le \epsilon
        \right\}.
    \]
\end{definition}

Taking $w = \mathbf{1}_n \coloneqq (1,\ldots,1)$ recovers the uniform landscape:
\[
    \UL_M(k, t)
    =
    \WL_M(\mathbf{1}_n, k, t).
\]

Vipond proved that the uniform landscape and each fixed admissible weighted landscape are $1$-Lipschitz with respect to the interleaving distance \cite[Theorem~30 and Corollary~33]{JMLR:v21:Vipond}. He also proved that, for finitely presented modules, the complete family of weighted landscapes determines the rank invariant \cite[Proposition~36]{JMLR:v21:Vipond}.

%% file: sections/generalized_persistence_landscapes.tex
\section{Landscape Intervals and Multiparameter Persistence Landscapes}
\label{sec:Generalized-landscapes}

We define landscape intervals and establish their basic structure. We then use their geometry to define the directional and integrated multiparameter persistence landscapes. We conclude with basic analytic properties of the integrated landscape.

\subsection{Landscape intervals}
\label{subsec:landscape-intervals}

\begin{definition}[Landscape interval]
\label{def:landscape-interval}
    Let $M \in \Vect^{\R^n}$. For $t \in \R^n$ and $k \in \N_{>0}$, define
    \[
        \I_M^k(t)
        \coloneqq
        \left\{
            h \in \R_{\ge 0}^n :
            \rk_M(t - h, t + h) \ge k
        \right\}.
    \]
    We call $\I_M^k(t)$ the \emph{$k$th landscape interval of $M$ at $t$}, allowing it to be empty.
\end{definition}

The following proposition justifies the term \emph{landscape interval} by showing that every nonempty such set is an interval in the sense of Section~\ref{sec:preliminaries}.

\begin{proposition}
\label{prop:landscape-interval-structure}
    For any $M \in \Vect^{\R^n}$, $t \in \R^n$, and $k \in \N_{>0}$, the set $\I_M^k(t)$ is either empty or an interval in $\R^n$. Moreover,
    \[
        \I_M^k(t) \neq \varnothing
        \quad\Longleftrightarrow\quad
        \dim_{\F} M(t) \ge k.
    \]
\end{proposition}

\begin{proof}
    Let $\mathbf{0} = (0, \ldots, 0) \in \R_{\ge 0}^n$. Suppose that $h \in \I_M^k(t)$ and $\mathbf{0} \preceq h' \preceq h$. By functoriality,
    \[
        M_{t - h}^{t + h}
        =
        M_{t + h'}^{t + h}
        \circ
        M_{t - h'}^{t + h'}
        \circ
        M_{t - h}^{t - h'}.
    \]
    Hence, $\rk_M(t - h', t + h') \ge \rk_M(t - h, t + h) \ge k$, so $h' \in \I_M^k(t)$. Thus, $\I_M^k(t)$ is closed under decreasing coordinates. In particular, it is nonempty if and only if it contains $\mathbf{0}$. Since $\rk_M(t, t) = \dim_{\F} M(t)$, this is equivalent to $\dim_{\F} M(t) \ge k$.

    Now suppose that $\I_M^k(t)$ is nonempty. If $h_1, h_2 \in \I_M^k(t)$ and $h_1 \preceq h' \preceq h_2$, then $h' \in \I_M^k(t)$ by the preceding observation. This proves the first condition in the definition of an interval. Moreover, $\mathbf{0} \in \I_M^k(t)$, and any $h_1, h_2 \in \I_M^k(t)$ are connected by the sequence $h_1 \succeq \mathbf{0} \preceq h_2$. Thus, $\I_M^k(t)$ also satisfies the second condition and is therefore an interval.
\end{proof}

The proof also shows that landscape intervals are closed under decreasing coordinates. When $n=1$, the supremum of the landscape interval recovers the one-parameter persistence landscape:
\[
    \lambda_M(k,t)
    =
    \sup \I_M^k(t).
\]

\subsection{Directional multiparameter persistence landscapes}
\label{subsec:directional-landscapes}

The landscape interval records how far the rank condition persists in every nonnegative direction. We summarize these radial extents in the following landscape.

Let
\[
    S_{\infty,+}^{n-1}
    \coloneqq
    \left\{
        v \in \R_{\ge 0}^n :
        \lVert v \rVert_\infty = 1
    \right\}
\]
denote the nonnegative part of the $\ell^\infty$-unit sphere.

\begin{definition}[Directional multiparameter persistence landscape]
\label{def:directional-multiparameter-persistence-landscape}
    Let $M \in \Vect^{\R^n}$. The \emph{directional multiparameter persistence landscape}, or simply \emph{directional landscape}, of $M$ is the function
    \[
        \DL_M
        \colon
        S_{\infty,+}^{n-1}
        \times
        \N_{>0}
        \times
        \R^n
        \longrightarrow
        [0,\infty]
    \]
    defined by
    \[
        \DL_M(v, k, t)
        \coloneqq
        \sup
        \left\{
            \epsilon \ge 0 :
            \rk_M(t - \epsilon v, t + \epsilon v) \ge k
        \right\}.
    \]
\end{definition}

Equivalently,
\[
    \DL_M(v, k, t)
    =
    \sup
    \left\{
        \epsilon \ge 0 :
        \epsilon v \in \I_M^k(t)
    \right\}.
\]
Thus, $\DL_M(v,k,t)$ records the radial extent of the landscape interval $\I_M^k(t)$ in the direction $v$. For fixed $v$ and $t$, the function $\DL_M(v,\bullet,t)$ is nonincreasing, since $\I_M^{k+1}(t)\subseteq\I_M^k(t)$. When $n=1$, we have $S_{\infty,+}^{0}=\{1\}$ and $\DL_M(1,k,t)=\lambda_M(k,t)$. Relationships with Vipond's uniform and weighted landscapes, together with metric and stability properties, are established in Section~\ref{sec:stability}. Further geometric properties and reconstruction from directional values are developed in Section~\ref{sec:finite-characterization}.

\subsection{Integrated multiparameter persistence landscapes}
\label{subsec:integrated-landscapes}

We next summarize the collective extent of the landscape interval rather than its extent along a single direction. Let $\mathcal{L}^m$ denote the $m$-dimensional Lebesgue measure. When $n = 1$, the landscape interval is an initial interval whose length is the one-parameter persistence landscape value. Hence,
\[
    \mathcal{L}^1\bigl(\I_M^k(t)\bigr)
    =
    \lambda_M(k, t).
\]
This motivates using higher-dimensional Lebesgue measure to summarize the geometry of $\I_M^k(t)$. The following measurability result was established in \cite[Proposition~8]{JMLR:v21:Vipond}. For completeness, Appendix~\ref{app:measurability-of-intervals} gives an alternative proof based on the geometry of lower and upper sets.

\begin{proposition}
\label{prop:interval-measurability}
    Every interval in $\R^n$ is Lebesgue-measurable.
\end{proposition}

By Proposition~\ref{prop:landscape-interval-structure}, every landscape interval is either empty or an interval in $\R^n$. Hence, every landscape interval is Lebesgue-measurable.

For $r \ge 0$, let $Q_r \coloneqq [0, r]^n$. Rather than reduce a landscape interval to its total volume, we measure the portion contained in $Q_r$. As $r$ varies, these truncated volumes yield a measure-based profile of the landscape interval's scale-dependent geometry by recording how its volume accumulates across scales. Since $\mathcal{L}^n(Q_r) = r^n$, dividing by $r^{n-1}$ normalizes the truncated volume to a quantity with units of length and ensures that it is bounded above by $r$. This leads to the following definition.

\begin{definition}[Integrated multiparameter persistence landscape]
\label{def:integrated-multiparameter-persistence-landscape}
    Let $M \in \Vect^{\R^n}$. The \emph{integrated multiparameter persistence landscape}, or simply \emph{integrated landscape}, of $M$ is the function
    \[
        \IL_M
        \colon
        \R_{\ge 0}
        \times
        \N_{>0}
        \times
        \R^n
        \longrightarrow
        [0,\infty)
    \]
    defined, for $r > 0$, by
    \[
        \IL_M(r, k, t)
        \coloneqq
        \frac{
            \mathcal{L}^n\bigl(Q_r \cap \I_M^k(t)\bigr)
        }{
            r^{n-1}
        },
    \]
    and by $\IL_M(0, k, t) \coloneqq 0$.
\end{definition}

For $r > 0$, the inequality $\mathcal{L}^n\bigl(Q_r \cap \I_M^k(t)\bigr) \le r^n$ gives $0 \le \IL_M(r, k, t) \le r$. Thus, $\IL_M(0, k, t) = 0$ is the natural extension at $r = 0$. When $n = 1$,
\[
    \IL_M(r, k, t)
    =
    \min\left\{
        r,
        \lambda_M(k, t)
    \right\}.
\]
Consequently, the one-parameter persistence landscape can be recovered from the integrated landscape by
\[
    \lambda_M(k, t)
    =
    \lim_{r \to \infty}
    \IL_M(r, k, t).
\]
In every dimension, Proposition~\ref{prop:uniform-recovered-from-integrated} shows that the uniform landscape can be recovered pointwise from the integrated landscape:
\[
    \UL_M(k, t)
    =
    \sup
    \left\{
        r \ge 0 :
        \IL_M(r, k, t) = r
    \right\}.
\]
Stability bounds for the integrated landscape with respect to erosion and interleaving distances are established in Section~\ref{sec:stability}.

The following example illustrates how the integrated landscape reflects the changing geometry of $Q_r \cap \I_M^k(t)$ as the scale $r$ varies.

\begin{example}
\label{ex:triangular-landscape-interval}
    Let $a > 0$ and consider the interval
    \[
        J_a
        \coloneqq
        \left\{
            x = (x_1, x_2) \in \R^2 :
            x_1 + x_2 \le a
        \right\}.
    \]
    Let $M = \mathds{1}_{J_a} \in \Vect^{\R^2}$ be the interval module supported on $J_a$. For $h \in \R_{\ge 0}^2$, we have $-h \in J_a$, while $h \in J_a$ if and only if $h_1 + h_2 \le a$. Therefore, the structure map $M_{-h}^{h}$ has rank one if and only if $h_1 + h_2 \le a$, and hence
    \[
        \I_M^1(\mathbf{0})
        =
        \left\{
            h = (h_1, h_2) \in \R_{\ge 0}^2 :
            h_1 + h_2 \le a
        \right\}.
    \]
    Thus, $\I_M^1(\mathbf{0})$ is the triangle with vertices $(0, 0)$, $(a, 0)$, and $(0, a)$.

    The shape of $Q_r \cap \I_M^1(\mathbf{0})$ depends on $r$. If $0 < r \le a/2$, the intersection is the entire square $Q_r$. If $a/2 < r < a$, it is $Q_r$ with a triangle of side length $2r - a$ removed. If $r \ge a$, it is the entire triangle $\I_M^1(\mathbf{0})$. Consequently,
    \[
        \IL_M(r, 1, \mathbf{0})
        =
        \begin{cases}
            r,
            & 0 < r \le \dfrac{a}{2},\\[2mm]
            r - \dfrac{(2r - a)^2}{2r},
            & \dfrac{a}{2} < r < a,\\[3mm]
            \dfrac{a^2}{2r},
            & r \ge a.
        \end{cases}
    \]
    By definition, $\IL_M(0, 1, \mathbf{0}) = 0$. For $k \ge 2$, the landscape interval is empty, and hence $\IL_M(r, k, \mathbf{0}) = 0$ for every $r \ge 0$.
\end{example}

\subsection{Analytic properties of the integrated landscape}
\label{subsec:integrated-landscape-properties}

We conclude this section by establishing bounds, monotonicity, continuity, and Lipschitz properties of the integrated landscape.

\begin{theorem}
\label{thm:integrated-landscape-properties}
    Let $M \in \Vect^{\R^n}$. For every $r \in \R_{\ge 0}$, $k \in \N_{>0}$, and $t \in \R^n$, the integrated landscape satisfies:
    \begin{enumerate}
        \item $0 \le \IL_M(r, k, t) \le r$;
        \item $\IL_M(\bullet, k, t)$ is continuous;
        \item $\IL_M(r, \bullet, t)$ is nonincreasing;
        \item $\IL_M(r, k, \bullet)$ is $1$-Lipschitz with respect to the $\ell^1$-norm and hence $n$-Lipschitz with respect to the $\ell^\infty$-norm.
    \end{enumerate}
\end{theorem}

\begin{proof}
    For part~(1), if $r > 0$, then $0 \le \mathcal{L}^n\bigl(Q_r \cap \I_M^k(t)\bigr) \le r^n$. Dividing by $r^{n-1}$ gives
    $0 \le \IL_M(r, k, t) \le r$. The case $r = 0$ follows from the definition.

    For part~(2), set $V(r) \coloneqq
    \mathcal{L}^n\bigl(Q_r \cap \I_M^k(t)\bigr)$.
    If $0 \le r \le s$, then
    \[
        0
        \le
        V(s) - V(r)
        \le
        \mathcal{L}^n(Q_s \setminus Q_r)
        =
        s^n - r^n.
    \]
    Hence, $V$ is continuous. It follows that
    $\IL_M(r, k, t) = V(r)/r^{n-1}$ is continuous for $r > 0$. Continuity at $r = 0$ follows from
    $0 \le \IL_M(r, k, t) \le r$.

    Part~(3) follows from $\I_M^{k+1}(t) \subseteq \I_M^k(t)$.

    The proof of part~(4), which uses two translation lemmas for landscape intervals, is given in
    Appendix~\ref{app:integrated-landscape-lipschitz}.
\end{proof}

%% file: sections/distances_and_stability.tex
\section{Distances and Stability}
\label{sec:stability}

Throughout this section, all persistence modules are indexed by $\R^n$. We first recall the interleaving distance and erosion distances for restricted generalized rank invariants. We then establish relationships among the landscapes introduced above, define an erosion distance for families of landscape intervals, and compare the resulting landscape distances with erosion distances. Finally, we derive stability bounds with respect to the interleaving distance.

\subsection{Interleaving and erosion distances}
\label{subsec:interleaving-erosion}

Interleaving and erosion distances provide two complementary frameworks for formulating stability in persistence theory. Interleaving distances have become central tools for comparing persistence modules and related constructions \cite{BauerLesnick2015,BotnanLesnick2018,Bjerkevik2021Stability,BlumbergLesnick2024Stability,BlumbergLesnick2023Universality,mcfaddin2026interleaving}. Erosion distances and their variants likewise provide stability frameworks for generalized persistence diagrams, rank invariants, and other persistence constructions \cite{KimMemoli2021Spatiotemporal,MemoliStefanouZhou2024CupProducts,XinMukherjeeSamagaDey2023GRIL,ClauseDeyMemoliWang2023MetaDiagrams}. Comparing these two frameworks therefore provides a natural setting for studying the stability of the landscape constructions introduced above.

We begin with the interleaving distance. Interleavings were introduced for $1$-parameter persistence modules in \cite{ChazalEtAl2009Proximity}. Their categorical formulation was developed in \cite{BubenikScott2014}, their multidimensional theory in \cite{lesnick2015interleaving}, and interleaving metrics for persistence modules indexed by arbitrary preordered sets in \cite{BubenikDeSilvaScott2015}.

For $v \in \R_{\ge 0}^n$, let $\tau_v \colon \R^n \longrightarrow \R^n$ be the translation $\tau_v(a) = a + v$.

\begin{definition}[Shift and transition morphism]
    Let $v \in \R_{\ge 0}^n$. The \emph{$v$-shift} is the endofunctor
    \[
        S_v
        \colon
        \Vect^{\R^n}
        \longrightarrow
        \Vect^{\R^n},
        \qquad
        S_v M
        \coloneqq
        M \circ \tau_v.
    \]
    For a natural transformation $\alpha \colon M \Rightarrow N$, its $v$-shift is defined by $\bigl(S_v\alpha\bigr)_a \coloneqq \alpha_{a + v}$. The \emph{$v$-transition morphism} of $M$ is the natural transformation
    \[
        \varphi_M^v
        \colon
        M
        \Rightarrow
        S_v M,
        \qquad
        \bigl(\varphi_M^v\bigr)_a
        \coloneqq
        M_a^{a + v}.
    \]
\end{definition}

\begin{definition}[$v$-interleaving]
    Let $M, N \in \Vect^{\R^n}$ and $v \in \R_{\ge 0}^n$. The persistence modules $M$ and $N$ are \emph{$v$-interleaved} if there exist natural transformations
    \[
        f \colon M \Rightarrow S_v N
        \qquad\text{and}\qquad
        g \colon N \Rightarrow S_v M
    \]
    satisfying
    \[
        S_v(g) \circ f
        =
        \varphi_M^{2v}
        \qquad\text{and}\qquad
        S_v(f) \circ g
        =
        \varphi_N^{2v}.
    \]
    Equivalently, for every $a \in \R^n$,
    \[
        g_{a + v} \circ f_a
        =
        M_a^{a + 2v}
        \qquad\text{and}\qquad
        f_{a + v} \circ g_a
        =
        N_a^{a + 2v}.
    \]
\end{definition}


\begin{definition}[Interleaving distance]
    Let $M, N \in \Vect^{\R^n}$. The \emph{interleaving distance} between $M$ and $N$ is
    \[
        d_I(M, N)
        \coloneqq
        \inf
        \left\{
            \varepsilon \ge 0 :
            M \text{ and } N
            \text{ are }
            \varepsilon \mathbf{1}_n\text{-interleaved}
        \right\}.
    \]
    We set $d_I(M, N) = \infty$ if no such interleaving exists.
\end{definition}

We next recall the erosion distance for restricted generalized rank invariants. An earlier related stable pseudodistance for multidimensional persistent homology with torsion was introduced in \cite{Frosini2013StableComparison}. The erosion distance was introduced for generalized persistence diagrams in \cite{Patel2018GPD} and extended to rank invariants of generalized persistence modules in \cite{Puuska2020Erosion}; the latter work also shows that the earlier pseudodistance is a special case of the erosion framework. For integer-valued rank functions, the erosion distance can also be formulated as an interleaving distance between poset maps \cite{KimMemoliStefanou2024}. The formulation below for restrictions of the generalized rank invariant to thickening-closed families of intervals follows \cite{ClauseKimMemoli2024GRI}.

Let $\operatorname{Int}(\R^n)$ denote the collection of intervals in $\R^n$. For $I \in \operatorname{Int}(\R^n)$ and $\varepsilon \ge 0$, define the \emph{$\varepsilon$-thickening} of $I$ by
\[
    I^\varepsilon
    \coloneqq
    \left\{
        p \in \R^n :
        \text{there exists } q \in I
        \text{ such that }
        \lVert p - q\rVert_\infty \le \varepsilon
    \right\}.
\]
A family $\mathcal{A} \subseteq \operatorname{Int}(\R^n)$ is \emph{closed under thickenings} if $I^\varepsilon \in \mathcal{A}$ for every $I \in \mathcal{A}$ and $\varepsilon \ge 0$.

For $M \in \Vect^{\R^n}$ and a family $\mathcal{A} \subseteq \operatorname{Int}(\R^n)$, let
\[
    \rk_M^{\mathcal{A}}
    \colon
    \mathcal{A}
    \longrightarrow
    \overline{\Z}_{\ge 0},
    \qquad
    \rk_M^{\mathcal{A}}(I)
    \coloneqq
    \rk_M^{\mathrm{gen}}(I),
\]
denote the restriction of the generalized rank invariant to $\mathcal{A}$.

\begin{definition}[Erosion distance]
    Let $M, N \in \Vect^{\R^n}$, and let $\mathcal{A} \subseteq \operatorname{Int}(\R^n)$ be closed under thickenings. We say that there is an \emph{$\varepsilon$-erosion} between $\rk_M^{\mathcal{A}}$ and $\rk_N^{\mathcal{A}}$ if
    \[
        \rk_M^{\mathcal{A}}(I^\varepsilon)
        \le
        \rk_N^{\mathcal{A}}(I)
        \qquad\text{and}\qquad
        \rk_N^{\mathcal{A}}(I^\varepsilon)
        \le
        \rk_M^{\mathcal{A}}(I)
    \]
    for every $I \in \mathcal{A}$. The \emph{erosion distance} between $\rk_M^{\mathcal{A}}$ and $\rk_N^{\mathcal{A}}$ is
    \[
        d_E\left(
            \rk_M^{\mathcal{A}},
            \rk_N^{\mathcal{A}}
        \right)
        \coloneqq
        \inf
        \left\{
            \varepsilon \ge 0 :
            \text{there is an $\varepsilon$-erosion between }
            \rk_M^{\mathcal{A}}
            \text{ and }
            \rk_N^{\mathcal{A}}
        \right\}.
    \]
    We set $d_E\left(\rk_M^{\mathcal{A}}, \rk_N^{\mathcal{A}}\right) = \infty$ if no such erosion exists.
\end{definition}

For brevity, we write
\[
    d_E^{\mathcal{A}}(M, N)
    \coloneqq
    d_E\left(
        \rk_M^{\mathcal{A}},
        \rk_N^{\mathcal{A}}
    \right).
\]

Stability of the erosion distance with respect to the interleaving distance was established for rank invariants in \cite[Theorem~3.11]{Puuska2020Erosion}. For restricted generalized rank invariants, the stability theorem \cite[Theorem~H]{ClauseKimMemoli2024GRI} gives the following result. Although that theorem is stated for pointwise finite-dimensional modules, its proof factors one canonical limit-to-colimit map through another and therefore also applies under the extended-rank convention adopted here.

\begin{theorem}[Erosion--interleaving stability]
\label{thm:erosion-interleaving-stability}
    Let $M, N \in \Vect^{\R^n}$, and let $\mathcal{A} \subseteq \operatorname{Int}(\R^n)$ be closed under thickenings. Then
    \[
        d_E^{\mathcal{A}}(M, N)
        \le
        d_I(M, N).
    \]
\end{theorem}

We apply the erosion distance to the families of closed axis-aligned rectangles and cubes. For $t \in \R^n$ and $h \in \R_{\ge 0}^n$, define
\[
    R(t, h)
    \coloneqq
    \prod_{i = 1}^n
    [t_i - h_i, t_i + h_i],
\]
and let
\[
    \mathcal{R}_n
    \coloneqq
    \left\{
        R(t, h) :
        t \in \R^n,\ 
        h \in \R_{\ge 0}^n
    \right\}.
\]
We allow zero coordinates in $h$, so $\mathcal{R}_n$ contains degenerate rectangles. For $t \in \R^n$ and $r \ge 0$, define
\[
    C(t, r)
    \coloneqq
    R(t, r \mathbf{1}_n)
    =
    \left\{
        x \in \R^n :
        \lVert x - t \rVert_\infty \le r
    \right\},
\]
and let
\[
    \mathcal{C}_n
    \coloneqq
    \left\{
        C(t, r) :
        t \in \R^n,\ 
        r \ge 0
    \right\}.
\]
Thus, $\mathcal{C}_n \subseteq \mathcal{R}_n$. Both families are closed under thickenings because
\[
    R(t, h)^\varepsilon
    =
    R\bigl(t, h + \varepsilon \mathbf{1}_n\bigr),
    \qquad
    C(t, r)^\varepsilon
    =
    C(t, r + \varepsilon).
\]

\subsection{Relations among the landscapes}
\label{subsec:relations-among-landscapes}

We first establish the precise relationship between the weighted and directional landscapes. For $w \in \mathcal{W}_n$, write $w^{-1} \coloneqq (w_1^{-1}, \ldots, w_n^{-1})$.

\begin{proposition}[Directional--weighted correspondence]
\label{prop:directional-weighted-relation}
    For every $M \in \Vect^{\R^n}$, $w \in \mathcal{W}_n$, $k \in \N_{>0}$, and $t \in \R^n$,
    \[
        \DL_M
        \left(
            \frac{w^{-1}}
            {\lVert w^{-1} \rVert_\infty},
            k,
            t
        \right)
        =
        \lVert w^{-1} \rVert_\infty
        \WL_M(w, k, t).
    \]
\end{proposition}

\begin{proof}
    For $h \in \R_{\ge 0}^n$, recall that
    $\lVert h \rVert_{\infty,w} = \max_i w_i h_i$. Hence, for every
    $\varepsilon \ge 0$,
    \[
        \left\{
            h \in \R_{\ge 0}^n :
            \lVert h \rVert_{\infty,w}
            \le
            \varepsilon
        \right\}
        =
        \left\{
            h \in \R_{\ge 0}^n :
            h \preceq \varepsilon w^{-1}
        \right\}.
    \]
    Since $\I_M^k(t)$ is closed under decreasing coordinates, the rank condition in the definition of $\WL_M(w, k, t)$ holds for every such $h$ if and only if $\varepsilon w^{-1} \in \I_M^k(t)$. Therefore,
    \[
        \WL_M(w, k, t)
        =
        \sup
        \left\{
            \varepsilon \ge 0 :
            \varepsilon w^{-1}
            \in
            \I_M^k(t)
        \right\}.
    \]
    Moreover,
    \[
    \begin{aligned}
        \DL_M
        \left(
            \frac{w^{-1}}
            {\lVert w^{-1} \rVert_\infty},
            k,
            t
        \right)
        &=
        \sup
        \left\{
            a \ge 0 :
            \frac{a}{\lVert w^{-1} \rVert_\infty}
            w^{-1}
            \in
            \I_M^k(t)
        \right\}\\
        &=
        \lVert w^{-1} \rVert_\infty
        \WL_M(w, k, t).
    \end{aligned}
    \]
\end{proof}

\begin{proposition}[Recovery of the uniform landscape]
\label{prop:uniform-recovered-from-integrated}
    For every $M \in \Vect^{\R^n}$, $k \in \N_{>0}$, and $t \in \R^n$,
    \[
        \UL_M(k, t)
        =
        \WL_M(\mathbf{1}_n, k, t)
        =
        \DL_M(\mathbf{1}_n, k, t)
        =
        \sup
        \left\{
            r \ge 0 :
            \IL_M(r, k, t) = r
        \right\}.
    \]
\end{proposition}

\begin{proof}
    The equality $\UL_M(k, t) = \WL_M(\mathbf{1}_n, k, t)$ follows directly from the definitions. Since $\mathbf{1}_n^{-1} = \mathbf{1}_n$ and $\lVert \mathbf{1}_n \rVert_\infty = 1$, Proposition~\ref{prop:directional-weighted-relation} gives
    \[
        \WL_M(\mathbf{1}_n, k, t)
        =
        \DL_M(\mathbf{1}_n, k, t).
    \]

    It remains to prove the final equality. Suppose first that $0 \le r < \UL_M(k, t)$. Then there exists $a > r$ such that $a \mathbf{1}_n \in \I_M^k(t)$. Since $\I_M^k(t)$ is closed under decreasing coordinates, $Q_r \subseteq \I_M^k(t)$, and hence $\IL_M(r, k, t) = r$.

    Conversely, suppose that $r > 0$ and $\IL_M(r, k, t) = r$. We claim that $Q_s \subseteq \I_M^k(t)$ for every $0 \le s < r$. Otherwise, there would exist $h \in Q_s \setminus \I_M^k(t)$. Since $\I_M^k(t)$ is closed under decreasing coordinates, the box $[h, r \mathbf{1}_n]$ would be disjoint from $\I_M^k(t)$. Because $h_i \le s < r$ for every $i$, this box has positive Lebesgue measure, contradicting
    \[
        \mathcal{L}^n
        \bigl(
            Q_r \cap \I_M^k(t)
        \bigr)
        =
        r^n.
    \]
    Therefore, $s \mathbf{1}_n \in \I_M^k(t)$ for every $s < r$, so $r \le \UL_M(k, t)$. Since $\IL_M(0, k, t) = 0$, these two implications prove the final equality.
\end{proof}

Because some landscape values may be infinite, we use the following extended-distance convention. For $a, b \in [0,\infty]$, define
\[
    d_{\mathrm{ext}}(a, b)
    \coloneqq
    \begin{cases}
        |a - b|,
        & a, b < \infty,\\
        0,
        & a = b = \infty,\\
        \infty,
        & \text{otherwise}.
    \end{cases}
\]
For functions $f, g \colon X \longrightarrow [0,\infty]$, define
\[
    \lVert f - g \rVert_\infty
    \coloneqq
    \sup_{x \in X}
    d_{\mathrm{ext}}\bigl(f(x), g(x)\bigr).
\]
Accordingly, every supremum distance between landscapes below is understood in this extended sense and may take the value $\infty$.

The preceding proposition shows that the integrated landscape determines the uniform landscape. The following estimate quantifies this recovery when the uniform landscape values have a common finite upper bound over all centers $t$ and rank thresholds $k$. This assumption holds, for example, when both modules vanish outside a fixed bounded rectangle.

\begin{proposition}[Uniform--integrated distance comparison]
\label{prop:uniform-integrated-comparison}
    Let $M, N \in \Vect^{\R^n}$. Suppose that there exists a finite constant $R > 0$ such that, for every $k \in \N_{>0}$ and $t \in \R^n$,
    \[
        \max\{\UL_M(k,t),\UL_N(k,t)\}
        \le R.
    \]
    Then
    \[
        \lVert \UL_M-\UL_N \rVert_\infty
        \le
        R^{(n-1)/n}
        \lVert \IL_M-\IL_N \rVert_\infty^{1/n}.
    \]
\end{proposition}

\begin{proof}
    Set $\Delta \coloneqq \lVert \IL_M-\IL_N \rVert_\infty$. The result is immediate if $\Delta = \infty$. Otherwise, fix $k \in \N_{>0}$ and $t \in \R^n$. After interchanging $M$ and $N$ if necessary, write $a \coloneqq \UL_M(k,t)$ and $b \coloneqq \UL_N(k,t)$ with $a \ge b$. If $a=b$, there is nothing to prove, so suppose that $a>b$.

    For every $0 \le s<a$, the diagonal displacement $s\mathbf{1}_n$ belongs to $\I_M^k(t)$. Since $\I_M^k(t)$ is closed under decreasing coordinates, $[0,a)^n \subseteq \I_M^k(t)$. This set has the same Lebesgue measure as $Q_a$, so $\IL_M(a,k,t)=a$.

    On the other hand, $(b,a]^n$ is disjoint from $\I_N^k(t)$. Indeed, if $h$ belonged to both sets, set $q \coloneqq \min_i h_i$. Then $q>b$ and $q\mathbf{1}_n \preceq h$. Since $\I_N^k(t)$ is closed under decreasing coordinates, this would imply $q\mathbf{1}_n \in \I_N^k(t)$, contradicting $b=\UL_N(k,t)$. The missing box $(b,a]^n$ has Lebesgue measure $(a-b)^n$. Therefore,
    \[
        \Delta
        \ge
        \IL_M(a,k,t)-\IL_N(a,k,t)
        \ge
        \frac{(a-b)^n}{a^{n-1}}
        \ge
        \frac{(a-b)^n}{R^{n-1}}.
    \]
    Rearranging and taking the supremum over $k$ and $t$ proves the result.
\end{proof}

Proposition~\ref{prop:uniform-recovered-from-integrated} identifies the common weight and direction $\mathbf{1}_n$. We now compare the weighted and directional landscape families after restricting to weights and directions whose coordinates are uniformly bounded below. For $0 < \alpha \le 1$, define
\[
\begin{aligned}
    \mathcal{W}_n(\alpha)
    &\coloneqq
    \left\{
        w \in \mathcal{W}_n :
        \min_i w_i \ge \alpha
    \right\},\\
    S_{\infty,+}^{n-1}(\alpha)
    &\coloneqq
    \left\{
        v \in S_{\infty,+}^{n-1} :
        \min_i v_i \ge \alpha
    \right\}.
\end{aligned}
\]
For $M, N \in \Vect^{\R^n}$, define the corresponding restricted distances by
\[
\begin{aligned}
    \lVert \WL_M - \WL_N \rVert_{\infty,\alpha}
    &\coloneqq
    \sup_{\substack{
        w \in \mathcal{W}_n(\alpha)\\
        k \in \N_{>0},\, t \in \R^n
    }}
    d_{\mathrm{ext}}\bigl(
        \WL_M(w, k, t),
        \WL_N(w, k, t)
    \bigr),\\
    \lVert \DL_M - \DL_N \rVert_{\infty,\alpha}
    &\coloneqq
    \sup_{\substack{
        v \in S_{\infty,+}^{n-1}(\alpha)\\
        k \in \N_{>0},\, t \in \R^n
    }}
    d_{\mathrm{ext}}\bigl(
        \DL_M(v, k, t),
        \DL_N(v, k, t)
    \bigr).
\end{aligned}
\]
When $\alpha = 1$, both restricted index sets reduce to $\{\mathbf{1}_n\}$. Proposition~\ref{prop:uniform-recovered-from-integrated} therefore gives
\[
    \lVert \WL_M - \WL_N \rVert_{\infty,1}
    =
    \lVert \DL_M - \DL_N \rVert_{\infty,1}
    =
    \lVert \UL_M - \UL_N \rVert_\infty.
\]
The following result extends this equality to a chain of comparisons for $0 < \alpha \le 1$.

\begin{corollary}[Uniform, weighted, and directional comparison]
\label{cor:directional-weighted-comparison}
    For $M, N \in \Vect^{\R^n}$ and $0 < \alpha \le 1$,
    \[
        \lVert \UL_M - \UL_N \rVert_\infty
        \le
        \lVert \WL_M - \WL_N \rVert_{\infty,\alpha}
        \le
        \lVert \DL_M - \DL_N \rVert_{\infty,\alpha}
        \le
        \frac{1}{\alpha}
        \lVert \WL_M - \WL_N \rVert_{\infty,\alpha}.
    \]
\end{corollary}

\begin{proof}
    The first inequality follows because
    $\mathbf{1}_n \in \mathcal{W}_n(\alpha)$ and
    $\WL_M(\mathbf{1}_n, k, t) = \UL_M(k, t)$.
    The map
    $w \longmapsto w^{-1}/\lVert w^{-1} \rVert_\infty$
    is a bijection from $\mathcal{W}_n(\alpha)$ to
    $S_{\infty,+}^{n-1}(\alpha)$. Indeed, if
    $m = \min_i w_i$, then $ w^{-1} / {\lVert w^{-1} \rVert_\infty} = m w^{-1}$. Since $\lVert w \rVert_\infty = 1$, the minimum coordinate of $m w^{-1}$ is $m$, and applying the same map twice returns $w$.
    Moreover,
    \[
        1
        \le
        \lVert w^{-1} \rVert_\infty
        \le
        \frac{1}{\alpha}
    \]
    for every $w \in \mathcal{W}_n(\alpha)$.

    For corresponding elements $w \in \mathcal{W}_n(\alpha)$ and $v = w^{-1}/\lVert w^{-1} \rVert_\infty$, the identity $d_{\mathrm{ext}}(ca, cb) = c d_{\mathrm{ext}}(a, b)$ for $c > 0$, together with Proposition~\ref{prop:directional-weighted-relation}, gives
    \[
        d_{\mathrm{ext}}
        \bigl(
            \DL_M(v, k, t),
            \DL_N(v, k, t)
        \bigr)
        =
        \lVert w^{-1} \rVert_\infty
        d_{\mathrm{ext}}
        \bigl(
            \WL_M(w, k, t),
            \WL_N(w, k, t)
        \bigr).
    \]
    Taking the supremum over $w$, $k$, and $t$ proves both inequalities.
\end{proof}

We next express the integrated landscape as a radial aggregate of the directional landscape. Let $\sigma$ denote the finite Borel measure on $S_{\infty,+}^{n-1}$ for which the $\ell^\infty$-polar-coordinate formula
\[
    \int_{\R_{\ge 0}^n}
        f(h)\,d\mathcal{L}^n(h)
    =
    \int_{S_{\infty,+}^{n-1}}
    \int_0^\infty
        f(\eta v)\eta^{n-1}
        \,d\eta\,d\sigma(v)
\]
holds for every nonnegative measurable function $f \colon \R_{\ge 0}^n \longrightarrow [0,\infty]$. Applying this formula to the indicator function of $Q_1$ gives
\[
    \sigma\bigl(S_{\infty,+}^{n-1}\bigr)
    =
    n.
\]

\begin{lemma}[Radial representation of the integrated landscape]
\label{lem:integrated-radial-representation}
    Let $M \in \Vect^{\R^n}$. For every $r > 0$, $k \in \N_{>0}$, and $t \in \R^n$,
    \[
        \IL_M(r, k, t)
        =
        \frac{1}{n r^{n-1}}
        \int_{S_{\infty,+}^{n-1}}
        \min
        \left\{
            \DL_M(v, k, t),
            r
        \right\}^{n}
        \,d\sigma(v).
    \]
\end{lemma}

\begin{proof}
    Because $\I_M^k(t)$ is closed under decreasing coordinates, its intersection with the ray $\{\eta v : \eta \ge 0\}$ is either empty or an interval beginning at the origin, and its supremum is $\DL_M(v, k, t)$. Whether the endpoint is included does not affect its one-dimensional measure. Its radial extent in $Q_r \cap \I_M^k(t)$ is therefore
    \[
        \rho_M(v)
        \coloneqq
        \min
        \left\{
            \DL_M(v, k, t),
            r
        \right\}.
    \]
    Applying the $\ell^\infty$-polar-coordinate formula to the indicator function of $Q_r \cap \I_M^k(t)$ gives
    \[
    \begin{aligned}
        \mathcal{L}^n
        \bigl(
            Q_r \cap \I_M^k(t)
        \bigr)
        &=
        \int_{S_{\infty,+}^{n-1}}
        \int_0^{\rho_M(v)}
            \eta^{n-1}
            \,d\eta\,d\sigma(v)\\
        &=
        \frac{1}{n}
        \int_{S_{\infty,+}^{n-1}}
            \rho_M(v)^n
            \,d\sigma(v).
    \end{aligned}
    \]
    Dividing by $r^{n-1}$ proves the result.
\end{proof}

\begin{theorem}[Integrated--directional comparison]
\label{thm:integrated-directional-comparison}
    For $M, N \in \Vect^{\R^n}$,
    \[
        \lVert \IL_M - \IL_N \rVert_\infty
        \le
        n
        \lVert \DL_M - \DL_N \rVert_\infty.
    \]
\end{theorem}

\begin{proof}
    Let $\Delta \coloneqq \lVert \DL_M - \DL_N \rVert_\infty$. The result is immediate if $\Delta = \infty$. Otherwise, fix $r > 0$, $k \in \N_{>0}$, and $t \in \R^n$, and set
    \[
        \rho_M(v)
        \coloneqq
        \min
        \left\{
            \DL_M(v, k, t),
            r
        \right\},
        \qquad
        \rho_N(v)
        \coloneqq
        \min
        \left\{
            \DL_N(v, k, t),
            r
        \right\}.
    \]
    We have $0 \le \rho_M(v), \rho_N(v) \le r$ and $|\rho_M(v) - \rho_N(v)| \le \Delta$. By the mean value theorem,
    \[
        \left|
            \rho_M(v)^n
            -
            \rho_N(v)^n
        \right|
        \le
        n r^{n-1}\Delta.
    \]
    Lemma~\ref{lem:integrated-radial-representation} therefore gives
    \[
    \begin{aligned}
        \left|
            \IL_M(r, k, t)
            -
            \IL_N(r, k, t)
        \right|
        &\le
        \frac{1}{n r^{n-1}}
        \int_{S_{\infty,+}^{n-1}}
        \left|
            \rho_M(v)^n
            -
            \rho_N(v)^n
        \right|
        \,d\sigma(v)\\
        &\le
        \Delta\,
        \sigma\bigl(S_{\infty,+}^{n-1}\bigr)\\
        &=
        n\Delta.
    \end{aligned}
    \]
    At $r = 0$, both integrated landscapes are zero. Taking the supremum over $r$, $k$, and $t$ proves the result.
\end{proof}

The directional--weighted correspondence gives the following consequence for the integrated landscape.

\begin{corollary}[Integrated--weighted comparison]
\label{cor:integrated-weighted-comparison}
    For $M, N \in \Vect^{\R^n}$,
    \[
        \lVert \IL_M - \IL_N \rVert_\infty
        \le
        n
        \sup_{\substack{
            w \in \mathcal{W}_n\\
            k \in \N_{>0},\, t \in \R^n
        }}
        \lVert w^{-1} \rVert_\infty
        d_{\mathrm{ext}}
        \bigl(
            \WL_M(w, k, t),
            \WL_N(w, k, t)
        \bigr).
    \]
\end{corollary}

\begin{proof}
    The directions in $S_{\infty,+}^{n-1}$ having at least one zero coordinate form a set of $\sigma$-measure zero. Therefore, the proof of Theorem~\ref{thm:integrated-directional-comparison} requires only the supremum over strictly positive directions. Proposition~\ref{prop:directional-weighted-relation} and the bijection
    $w \longmapsto w^{-1}/{\lVert w^{-1} \rVert_\infty}$ from $\mathcal{W}_n$ to $S_{\infty,+}^{n-1} \cap \R_{>0}^n$ give
    \[
    \begin{aligned}
        &\sup_{\substack{
            v \in S_{\infty,+}^{n-1} \cap \R_{>0}^n\\
            k \in \N_{>0},\, t \in \R^n
        }}
        d_{\mathrm{ext}}
        \bigl(
            \DL_M(v, k, t),
            \DL_N(v, k, t)
        \bigr)\\
        &\qquad=
        \sup_{\substack{
            w \in \mathcal{W}_n\\
            k \in \N_{>0},\, t \in \R^n
        }}
        \lVert w^{-1} \rVert_\infty
        d_{\mathrm{ext}}
        \bigl(
            \WL_M(w, k, t),
            \WL_N(w, k, t)
        \bigr).
    \end{aligned}
    \]
    Substituting this equality into the proof of Theorem~\ref{thm:integrated-directional-comparison} proves the result.
\end{proof}

\subsection{Erosion distances and landscape distances}
\label{subsec:erosion-landscape-distances}

This subsection relates distances between landscape representations to erosion distances. We first define an erosion distance directly between complete families of landscape intervals and prove that it equals the rectangle-restricted erosion distance. We then identify
the supremum distance between uniform landscapes with the cube-restricted erosion distance and derive erosion bounds for the weighted, restricted directional, and integrated landscapes.

For $A \subseteq \R_{\ge 0}^n$ and $\varepsilon \ge 0$, define the \emph{diagonal $\varepsilon$-erosion} of $A$ by
\[
    \mathsf{E}_\varepsilon(A)
    \coloneqq
    \left\{
        h \in \R_{\ge 0}^n :
        h + \varepsilon \mathbf{1}_n
        \in
        A
    \right\}.
\]
Thus, $\mathsf{E}_\varepsilon(A)$ consists of the vectors whose diagonal translate by $\varepsilon\mathbf{1}_n$ lies in $A$. Diagonal erosion preserves downward-closed subsets of $\R_{\ge 0}^n$ and satisfies
\[
    \mathsf{E}_\delta\bigl(\mathsf{E}_\varepsilon(A)\bigr)
    =
    \mathsf{E}_{\delta+\varepsilon}(A).
\]

\begin{definition}[Landscape-interval erosion distance]
\label{def:landscape-interval-erosion-distance}
    Let $M, N \in \Vect^{\R^n}$ and $\varepsilon \ge 0$. We say that their landscape-interval families admit an \emph{$\varepsilon$-erosion} if, for every $t \in \R^n$ and $k \in \N_{>0}$,
    \[
        \mathsf{E}_\varepsilon\bigl(\I_M^k(t)\bigr)
        \subseteq
        \I_N^k(t)
        \qquad\text{and}\qquad
        \mathsf{E}_\varepsilon\bigl(\I_N^k(t)\bigr)
        \subseteq
        \I_M^k(t).
    \]
    The \emph{landscape-interval erosion distance} between $M$ and $N$ is
    \[
        d_{\LI}(M, N)
        \coloneqq
        \inf
        \left\{
            \varepsilon \ge 0 :
            \text{$M$ and $N$ admit an $\varepsilon$-erosion}
        \right\}.
    \]
    If no such $\varepsilon$ exists, we set $d_{\LI}(M, N) = \infty$.
\end{definition}

\begin{proposition}[Extended pseudometric]
\label{prop:landscape-interval-pseudometric}
    The landscape-interval erosion distance $d_{\LI}$ is an extended pseudometric. Explicitly, for $M, N, L \in \Vect^{\R^n}$, one has $d_{\LI}(M,M)=0$, $d_{\LI}(M,N)=d_{\LI}(N,M)$, and
    \[
        d_{\LI}(M,L)
        \le
        d_{\LI}(M,N)+d_{\LI}(N,L).
    \]
\end{proposition}

\begin{proof}
    Nonnegativity and symmetry follow directly from the definition. Since $\mathsf{E}_0(A)=A$ for every $A \subseteq \R_{\ge 0}^n$, each landscape-interval family admits a $0$-erosion with itself, so $d_{\LI}(M,M)=0$.

    For the triangle inequality, the result is immediate if either $d_{\LI}(M,N)$ or $d_{\LI}(N,L)$ is infinite. Otherwise, let $\varepsilon$ and $\delta$ be admissible erosion parameters for the pairs $(M,N)$ and $(N,L)$, respectively. Diagonal erosion preserves inclusions and satisfies the composition identity given above. Hence, for every $t \in \R^n$ and $k \in \N_{>0}$,
    \[
    \begin{aligned}
        \mathsf{E}_{\varepsilon+\delta}
        \bigl(\I_M^k(t)\bigr)
        &=
        \mathsf{E}_{\delta}
        \bigl(\mathsf{E}_{\varepsilon}(\I_M^k(t))\bigr)\\
        &\subseteq
        \mathsf{E}_{\delta}\bigl(\I_N^k(t)\bigr)
        \subseteq
        \I_L^k(t).
    \end{aligned}
    \]
    Similarly,
    \[
        \mathsf{E}_{\varepsilon+\delta}
        \bigl(\I_L^k(t)\bigr)
        \subseteq
        \mathsf{E}_{\varepsilon}\bigl(\I_N^k(t)\bigr)
        \subseteq
        \I_M^k(t).
    \]
    Thus, the landscape-interval families of $M$ and $L$ admit an $(\varepsilon+\delta)$-erosion, and therefore $d_{\LI}(M,L)\le\varepsilon+\delta$. Taking infima over the admissible parameters $\varepsilon$ and $\delta$ proves the triangle inequality.
\end{proof}

The distance may vanish even when the complete landscape-interval families are different.

\begin{example}[Zero distance between distinct landscape-interval families]
\label{ex:landscape-interval-zero-distance}
    Let $M \coloneqq \mathds{1}_{\{0\}} \in \Vect^\R$ be the interval module supported at the single point $0$, and let $N_0$ denote the zero persistence module. The only nonempty landscape interval of $M$ is $\I_M^1(0)=\{0\}$, whereas every landscape interval of $N_0$ is empty. For every $\varepsilon>0$, we have $\mathsf{E}_\varepsilon(\{0\})=\varnothing$. Hence, the two families admit an $\varepsilon$-erosion for every $\varepsilon>0$, and consequently $d_{\LI}(M,N_0)=0$. However, $\I_M^1(0)\neq\I_{N_0}^1(0)$, so their landscape-interval families are different. In this example, the infimum defining the distance is zero, but no $0$-erosion exists.
\end{example}

\begin{theorem}[Equivalence of landscape-interval and rectangle erosion]
\label{thm:landscape-interval-rectangle-erosion}
    Let $M, N \in \Vect^{\R^n}$ and $\varepsilon \ge 0$. The landscape-interval families of $M$ and $N$ admit an $\varepsilon$-erosion if and only if there is an $\varepsilon$-erosion between $\rk_M^{\mathcal{R}_n}$ and $\rk_N^{\mathcal{R}_n}$. Consequently,
    \[
        d_{\LI}(M, N)
        =
        d_E^{\mathcal{R}_n}(M, N).
    \]
\end{theorem}

\begin{proof}
    By the definition of a landscape interval and Lemma~\ref{lem:generalized-rank-extrema}, for every $X \in \{M, N\}$, $t \in \R^n$, $h \in \R_{\ge 0}^n$, and $k \in \N_{>0}$,
    \[
        h \in \I_X^k(t)
        \quad\Longleftrightarrow\quad
        \rk_X^{\mathcal{R}_n}\bigl(R(t, h)\bigr) \ge k.
    \]

    Suppose first that the landscape-interval families of $M$ and $N$ admit an $\varepsilon$-erosion. Fix $t \in \R^n$ and $h \in \R_{\ge 0}^n$. We claim that
    \[
        \rk_M^{\mathcal{R}_n}\bigl(R(t, h)^\varepsilon\bigr)
        \le
        \rk_N^{\mathcal{R}_n}\bigl(R(t, h)\bigr).
    \]
    Otherwise, the latter rank is finite, and we may take $k \coloneqq \rk_N^{\mathcal{R}_n}\bigl(R(t, h)\bigr)+1$. Since $R(t, h)^\varepsilon = R(t, h+\varepsilon\mathbf{1}_n)$, the strict inequality gives $h+\varepsilon\mathbf{1}_n \in \I_M^k(t)$, and hence $h \in \mathsf{E}_\varepsilon\bigl(\I_M^k(t)\bigr)$. The assumed erosion then gives $h \in \I_N^k(t)$, which implies
    \[
        \rk_N^{\mathcal{R}_n}\bigl(R(t, h)\bigr)
        \ge
        k
        =
        \rk_N^{\mathcal{R}_n}\bigl(R(t, h)\bigr)+1,
    \]
    a contradiction. Interchanging $M$ and $N$ gives
    \[
        \rk_N^{\mathcal{R}_n}\bigl(R(t, h)^\varepsilon\bigr)
        \le
        \rk_M^{\mathcal{R}_n}\bigl(R(t, h)\bigr).
    \]
    Thus, there is an $\varepsilon$-erosion between $\rk_M^{\mathcal{R}_n}$ and $\rk_N^{\mathcal{R}_n}$.

    Conversely, suppose that there is an $\varepsilon$-erosion between $\rk_M^{\mathcal{R}_n}$ and $\rk_N^{\mathcal{R}_n}$. Fix $t \in \R^n$, $k \in \N_{>0}$, and $h \in \mathsf{E}_\varepsilon\bigl(\I_M^k(t)\bigr)$. Then $h+\varepsilon\mathbf{1}_n \in \I_M^k(t)$, so
    \[
        \rk_M^{\mathcal{R}_n}\bigl(R(t, h)^\varepsilon\bigr)
        =
        \rk_M^{\mathcal{R}_n}\bigl(R(t, h+\varepsilon\mathbf{1}_n)\bigr)
        \ge
        k.
    \]
    The rectangle-erosion inequality gives
    \[
        \rk_N^{\mathcal{R}_n}\bigl(R(t, h)\bigr)
        \ge
        \rk_M^{\mathcal{R}_n}\bigl(R(t, h)^\varepsilon\bigr)
        \ge
        k.
    \]
    Hence, $h \in \I_N^k(t)$. Therefore,
    \[
        \mathsf{E}_\varepsilon\bigl(\I_M^k(t)\bigr)
        \subseteq
        \I_N^k(t).
    \]
    Interchanging $M$ and $N$ gives the reverse inclusion required in the definition of an $\varepsilon$-erosion between the landscape-interval families.

    The two types of $\varepsilon$-erosion therefore exist for exactly the same values of $\varepsilon$. Taking infima proves
    $d_{\LI}(M, N)=d_E^{\mathcal{R}_n}(M, N)$.
\end{proof}

We next compare the distance between uniform landscapes with the erosion distance over the cube family $\mathcal{C}_n$. By Proposition~\ref{prop:uniform-recovered-from-integrated}, $\UL_M(k, t) = \DL_M(\mathbf{1}_n, k, t)$.
Hence, by the definition of the directional landscape and Lemma~\ref{lem:generalized-rank-extrema},
\[
    \UL_M(k, t)
    =
    \sup
    \left\{
        r \ge 0 :
        \rk_M^{\mathcal{C}_n}\bigl(C(t, r)\bigr) \ge k
    \right\}.
\]

For $2$-parameter persistence modules, the equality between the uniform-landscape supremum distance and the cube-restricted erosion distance can be obtained from the generalized rank invariant landscape (GRIL) framework \cite{XinMukherjeeSamagaDey2023GRIL}. In that construction, a $1$-worm centered at $t \in \R^2$ with scale $r \ge 0$ is the closed square
\[
    \left\{
        x \in \R^2 :
        \lVert x-t\rVert_\infty \le r
    \right\}
    =
    C(t, r).
\]
Therefore, the GRIL value associated with $1$-worms is
\[
    \operatorname{GRIL}_M(t, k, 1)
    =
    \sup
    \left\{
        r \ge 0 :
        \rk_M^{\mathrm{gen}}\bigl(C(t, r)\bigr) \ge k
    \right\}
    =
    \UL_M(k, t).
\]
Restricting the worm family to $1$-worms likewise reduces the corresponding worm-erosion distance to $d_E^{\mathcal{C}_2}$. Thus, the proof of \cite[Proposition~B.9]{XinMukherjeeSamagaDey2023GRIL}, applied to this subfamily, gives this equality when $n = 2$. The following theorem establishes it directly for persistence modules indexed by $\R^n$ in every dimension $n \ge 1$.

\begin{theorem}[Uniform landscape and cube erosion]
\label{thm:uniform-erosion-distance}
    For $M, N \in \Vect^{\R^n}$,
    \[
        \lVert \UL_M-\UL_N\rVert_\infty
        =
        d_E^{\mathcal{C}_n}(M, N).
    \]
\end{theorem}

\begin{proof}
    For every $X \in \{M, N\}$, $t \in \R^n$, and $0 \le r \le s$, Proposition~\ref{prop:generalized-rank-monotonicity} gives
    \[
        \rk_X^{\mathcal{C}_n}\bigl(C(t, s)\bigr)
        \le
        \rk_X^{\mathcal{C}_n}\bigl(C(t, r)\bigr).
    \]
    Thus, if the cube rank is at least $k$ at radius $s$, then it is at least $k$ at every smaller radius.

    Suppose first that there is an $\varepsilon$-erosion between $\rk_M^{\mathcal{C}_n}$ and $\rk_N^{\mathcal{C}_n}$. Fix $t \in \R^n$ and $k \in \N_{>0}$, and set $\alpha \coloneqq \UL_M(k, t)$ and $\beta \coloneqq \UL_N(k, t)$. By symmetry, assume that $\alpha \le \beta$. If $\alpha = \beta = \infty$, then $d_{\mathrm{ext}}(\alpha,\beta)=0$. Otherwise, suppose for a contradiction that $d_{\mathrm{ext}}(\alpha,\beta)>\varepsilon$. We may then choose $\eta>0$ such that $\alpha+\varepsilon+\eta<\beta$. Since $\alpha+\eta>\alpha$, while $\alpha+\varepsilon+\eta<\beta$, we have
    \[
        \rk_M^{\mathcal{C}_n}\bigl(C(t,\alpha+\eta)\bigr)
        <
        k
        \qquad\text{and}\qquad
        \rk_N^{\mathcal{C}_n}\bigl(C(t,\alpha+\varepsilon+\eta)\bigr)
        \ge
        k.
    \]
    However, $C(t,\alpha+\eta)^\varepsilon=C(t,\alpha+\varepsilon+\eta)$, so the erosion inequality gives
    \[
        \rk_N^{\mathcal{C}_n}\bigl(C(t,\alpha+\varepsilon+\eta)\bigr)
        \le
        \rk_M^{\mathcal{C}_n}\bigl(C(t,\alpha+\eta)\bigr),
    \]
    a contradiction. Therefore, $d_{\mathrm{ext}}(\alpha,\beta)\le\varepsilon$. Taking the supremum over $k$ and $t$, and then the infimum over all erosion parameters, gives
    \[
        \lVert\UL_M-\UL_N\rVert_\infty
        \le
        d_E^{\mathcal{C}_n}(M,N).
    \]

    Conversely, let $\Delta \coloneqq \lVert\UL_M-\UL_N\rVert_\infty$. If $\Delta=\infty$, there is nothing to prove. Suppose that $\Delta<\infty$, and fix $\eta>0$. We show that there is a $(\Delta+\eta)$-erosion. Fix $t \in \R^n$ and $r \ge 0$, and suppose for a contradiction that
    \[
        \rk_N^{\mathcal{C}_n}\bigl(C(t,r+\Delta+\eta)\bigr)
        >
        \rk_M^{\mathcal{C}_n}\bigl(C(t,r)\bigr).
    \]
    Set $q \coloneqq \rk_M^{\mathcal{C}_n}\bigl(C(t,r)\bigr)$. The strict inequality forces $q<\infty$, so we may set $k \coloneqq q+1$. It follows that
    \[
        \UL_N(k,t)
        \ge
        r+\Delta+\eta,
        \qquad
        \UL_M(k,t)
        \le
        r.
    \]
    Consequently,
    \[
        d_{\mathrm{ext}}\bigl(\UL_M(k,t),\UL_N(k,t)\bigr)
        \ge
        \Delta+\eta,
    \]
    contradicting the definition of $\Delta$. Hence,
    \[
        \rk_N^{\mathcal{C}_n}\bigl(C(t,r+\Delta+\eta)\bigr)
        \le
        \rk_M^{\mathcal{C}_n}\bigl(C(t,r)\bigr).
    \]
    Interchanging $M$ and $N$ gives the other erosion inequality. Since $C(t,r)^{\Delta+\eta}=C(t,r+\Delta+\eta)$, there is a $(\Delta+\eta)$-erosion. Therefore, $d_E^{\mathcal{C}_n}(M,N)\le\Delta+\eta$. Letting $\eta\to0$ gives $d_E^{\mathcal{C}_n}(M,N)\le\Delta$, proving the result.
\end{proof}

When $n = 1$, the family $\mathcal{C}_1$ consists precisely of the nonempty closed bounded intervals in $\R$, and the uniform landscape coincides with the one-parameter persistence landscape. For finite persistence diagrams whose birth and death coordinates are finite, the equality shown in Theorem \ref{thm:uniform-erosion-distance} was previously established in \cite[Theorem~3.19]{ayhan2027equivalence}. Equivalently, this setting corresponds to finite barcodes consisting entirely of bounded intervals. Under the barcode--module correspondence, the erosion distance used there agrees with $d_E^{\mathcal{C}_1}$ for the associated persistence modules. The following corollary records the identity directly in the module-level framework used here, under the extended-rank and extended-distance conventions, without requiring the modules to admit finite persistence diagrams.

\begin{corollary}[One-parameter erosion identity]
\label{cor:one-parameter-erosion-distance}
    For $M, N \in \Vect^\R$,
    \[
        \lVert \lambda_M-\lambda_N\rVert_\infty
        =
        d_E^{\mathcal{C}_1}(M, N),
    \]
    where $\mathcal{C}_1=\{[a,b]\subseteq\R:a\le b\}$.
\end{corollary}

\begin{proof}
    When $n = 1$, the uniform landscape coincides with the one-parameter persistence landscape, and $\mathcal{C}_1$ is the family of nonempty closed bounded intervals in $\R$. The result follows from Theorem~\ref{thm:uniform-erosion-distance}.
\end{proof}

We next compare the restricted directional distance with the cube- and rectangle-restricted erosion distances. Recall from Subsection~\ref{subsec:relations-among-landscapes} that, for $0 < \alpha \le 1$, the set $S_{\infty,+}^{n-1}(\alpha)$ consists of the directions whose coordinates are bounded below by $\alpha$, and that $\lVert \DL_M - \DL_N \rVert_{\infty,\alpha}$ denotes the corresponding restricted distance.

\begin{theorem}[Restricted directional erosion comparison]
\label{thm:directional-erosion-bound}
    Let $M, N \in \Vect^{\R^n}$. For every $0 < \alpha \le 1$,
    \[
        d_E^{\mathcal{C}_n}(M, N)
        \le
        \lVert \DL_M - \DL_N \rVert_{\infty,\alpha}
        \le
        \frac{1}{\alpha}
        d_E^{\mathcal{R}_n}(M, N).
    \]
\end{theorem}

\begin{proof}
    The lower bound follows from Corollary~\ref{cor:directional-weighted-comparison} and Theorem~\ref{thm:uniform-erosion-distance}. It remains to prove the upper bound.
    Suppose that there is an $\varepsilon$-erosion between $\rk_M^{\mathcal{R}_n}$ and $\rk_N^{\mathcal{R}_n}$. Fix $v \in S_{\infty,+}^{n-1}(\alpha)$, $k \in \N_{>0}$, and $t \in \R^n$. If $\DL_M(v,k,t)\le\varepsilon/\alpha$, then $\DL_M(v,k,t)\le\DL_N(v,k,t)+\varepsilon/\alpha$. Otherwise, choose $s$ such that $\varepsilon/\alpha<s<\DL_M(v,k,t)$. Then $sv \in \I_M^k(t)$, and $h \coloneqq sv-\varepsilon\mathbf{1}_n$ belongs to $\R_{\ge 0}^n$. By Theorem~\ref{thm:landscape-interval-rectangle-erosion}, $h \in \I_N^k(t)$. Since $v_i\ge\alpha$ for every $i$,
    \[
        \left(s-\frac{\varepsilon}{\alpha}\right)v
        \preceq
        sv-\varepsilon\mathbf{1}_n
        =
        h.
    \]
    Since $\I_N^k(t)$ is closed under decreasing coordinates, it follows that $(s-\varepsilon/\alpha)v \in \I_N^k(t)$. Letting $s$ increase to $\DL_M(v,k,t)$ gives
    \[
        \DL_M(v,k,t)
        \le
        \DL_N(v,k,t)
        +
        \frac{\varepsilon}{\alpha}.
    \]
    Interchanging $M$ and $N$ gives the reverse inequality, and hence
    \[
        d_{\mathrm{ext}}\bigl(\DL_M(v,k,t),\DL_N(v,k,t)\bigr)
        \le
        \frac{\varepsilon}{\alpha}.
    \]
    Taking the supremum over $v$, $k$, and $t$, and then the infimum over all rectangle-erosion parameters $\varepsilon$, proves the result.
\end{proof}

The dependence on $\alpha$ cannot in general be improved, and a uniform erosion bound fails when directions arbitrarily close to the coordinate axes are allowed; see Example~\ref{ex:directional-interleaving-instability} below. In contrast, the weighted landscape admits an erosion bound that is uniform over all admissible weights.

\begin{theorem}[Weighted erosion comparison]
\label{thm:weighted-erosion-comparison}
    For $M, N \in \Vect^{\R^n}$,
    \[
        d_E^{\mathcal{C}_n}(M, N)
        \le
        \lVert \WL_M - \WL_N \rVert_\infty
        \le
        d_E^{\mathcal{R}_n}(M, N).
    \]
\end{theorem}

\begin{proof}
    The lower bound follows from Corollary~\ref{cor:directional-weighted-comparison} and Theorem~\ref{thm:uniform-erosion-distance}. It remains to prove the upper bound.
    Let $\varepsilon$ be an erosion parameter between $\rk_M^{\mathcal{R}_n}$ and $\rk_N^{\mathcal{R}_n}$. Fix $w \in \mathcal{W}_n$, $k \in \N_{>0}$, and $t \in \R^n$. If $\WL_M(w, k, t) \le \varepsilon$, then $\WL_M(w, k, t) \le \WL_N(w, k, t) + \varepsilon$. Otherwise, choose $\varepsilon < s < \WL_M(w, k, t)$. Then $sw^{-1} \in \I_M^k(t)$. Since $0 < w_i \le 1$ for every $i$, the vector $h \coloneqq sw^{-1} - \varepsilon\mathbf{1}_n$ belongs to $\R_{\ge 0}^n$. Theorem~\ref{thm:landscape-interval-rectangle-erosion} gives $h \in \I_N^k(t)$.

    Moreover, $(s - \varepsilon)w^{-1} \preceq sw^{-1} - \varepsilon\mathbf{1}_n = h$. Since $\I_N^k(t)$ is closed under decreasing coordinates, $(s - \varepsilon)w^{-1} \in \I_N^k(t)$, and hence $\WL_N(w, k, t) \ge s - \varepsilon$. Letting $s$ increase to $\WL_M(w, k, t)$ gives
    \[
        \WL_M(w, k, t)
        \le
        \WL_N(w, k, t) + \varepsilon.
    \]
    If $\WL_M(w, k, t) = \infty$, taking arbitrarily large finite $s$ gives $\WL_N(w, k, t) = \infty$. Interchanging $M$ and $N$ therefore gives
    \[
        d_{\mathrm{ext}}\bigl(
            \WL_M(w, k, t),
            \WL_N(w, k, t)
        \bigr)
        \le
        \varepsilon.
    \]
    Taking the supremum over $w$, $k$, and $t$, and then the infimum over erosion parameters, proves the result.
\end{proof}

Finally, the rectangle-restricted erosion distance also controls the integrated landscape.

\begin{theorem}[Integrated erosion bound]
\label{thm:integrated-erosion-bound}
    For $M, N \in \Vect^{\R^n}$,
    \[
        \lVert \IL_M - \IL_N \rVert_\infty
        \le
        n\,d_E^{\mathcal{R}_n}(M, N).
    \]
\end{theorem}

\begin{proof}
    Let $\varepsilon$ be an admissible parameter for $d_E^{\mathcal{R}_n}(M, N)$, and fix $k \in \N_{>0}$ and $t \in \R^n$. At $r = 0$, both integrated landscapes are zero. If $0 < r \le \varepsilon$, then $0 \le \IL_M(r, k, t), \IL_N(r, k, t) \le r$, and hence
    \[
        \left|
            \IL_M(r, k, t)
            -
            \IL_N(r, k, t)
        \right|
        \le
        r
        \le
        n\varepsilon.
    \]

    Now suppose that $r > \varepsilon$. If $h \in Q_r \cap \I_M^k(t) \cap [\varepsilon, r]^n$, then Theorem~\ref{thm:landscape-interval-rectangle-erosion} gives $h - \varepsilon\mathbf{1}_n \in \I_N^k(t)$, and this translated point also belongs to $Q_r$. Therefore, translation by $-\varepsilon\mathbf{1}_n$ maps $Q_r \cap \I_M^k(t) \cap [\varepsilon, r]^n$ into $Q_r \cap \I_N^k(t)$. The remaining portion of $Q_r \cap \I_M^k(t)$ is contained in $Q_r \setminus [\varepsilon, r]^n$, whose Lebesgue measure is $r^n - (r - \varepsilon)^n$. Since translation preserves Lebesgue measure,
    \[
        \mathcal{L}^n\bigl(Q_r \cap \I_M^k(t)\bigr)
        \le
        \mathcal{L}^n\bigl(Q_r \cap \I_N^k(t)\bigr)
        +
        r^n - (r - \varepsilon)^n.
    \]
    Because an admissible erosion parameter satisfies the corresponding inclusion with $M$ and $N$ interchanged, the same argument gives the reverse inequality. Applying the preceding inequality in both directions and dividing by $r^{n-1}$ gives
    \[
        \left|
            \IL_M(r, k, t)
            -
            \IL_N(r, k, t)
        \right|
        \le
        \frac{r^n - (r - \varepsilon)^n}{r^{n-1}}.
    \]
    For $\varepsilon > 0$, the mean value theorem gives a point $\xi\in(r-\varepsilon,r)$ such that
    \[
        r^n - (r - \varepsilon)^n
        =
        n\xi^{n-1}\varepsilon
        \le
        nr^{n-1}\varepsilon.
    \]
    Therefore, $|\IL_M(r, k, t) - \IL_N(r, k, t)| \le n\varepsilon$. Taking the supremum over $r$, $k$, and $t$, and then the infimum over admissible erosion parameters, proves the result.
\end{proof}

\subsection{Consequences for interleaving stability}
\label{subsec:interleaving-stability}

The equivalence between landscape-interval erosion and rectangle erosion, together with the erosion--interleaving stability theorem, immediately gives interleaving stability for complete families of landscape intervals. Combined with the comparisons established in the preceding subsection, it also yields interleaving-stability bounds for the associated landscapes.

For pointwise finite-dimensional persistence modules, Vipond proved that the uniform landscape and each fixed admissible weighted landscape are $1$-Lipschitz with respect to the interleaving distance \cite[Theorem~30 and Corollary~33]{JMLR:v21:Vipond}. Vipond also assembled the complete family over all admissible weights into the weighted landscape space and proved that, for finitely presented modules, the resulting map is injective and $1$-Lipschitz after modules with the same rank invariant are identified \cite[Definition~35 and Proposition~36]{JMLR:v21:Vipond}. Since the fixed-weight bound is independent of the weight vector, it also gives the same estimate for the combined weighted landscape $\WL_M$ on pointwise finite-dimensional modules. The following corollary recovers these bounds, extends them to arbitrary persistence modules under our extended-rank convention, and gives corresponding bounds for landscape-interval families and the restricted directional and integrated landscapes.

\begin{corollary}[Interleaving stability]
\label{cor:landscape-interleaving-stability}
    Let $M, N \in \Vect^{\R^n}$. Then
    \[
        d_{\LI}(M, N)
        =
        d_E^{\mathcal{R}_n}(M, N)
        \le
        d_I(M, N).
    \]
     Moreover, for every $0 < \alpha \le 1$,
    \[
    \begin{gathered}
        d_E^{\mathcal{C}_n}(M, N)
        =
        \lVert \UL_M - \UL_N \rVert_\infty
        \le
        \lVert \WL_M - \WL_N \rVert_\infty
        \le
        d_{\LI}(M, N)
        \le
        d_I(M, N),\\
        d_E^{\mathcal{C}_n}(M, N)
        \le
        \lVert \DL_M - \DL_N \rVert_{\infty,\alpha}
        \le
        \frac{1}{\alpha}d_{\LI}(M, N)
        \le
        \frac{1}{\alpha}d_I(M, N),\\
        \lVert \IL_M - \IL_N \rVert_\infty
        \le
        n\,d_{\LI}(M, N)
        \le
        n\,d_I(M, N).
    \end{gathered}
    \]
\end{corollary}

\begin{proof}
    Theorems~\ref{thm:landscape-interval-rectangle-erosion} and \ref{thm:erosion-interleaving-stability} give
    \[
        d_{\LI}(M,N)
        =
        d_E^{\mathcal{R}_n}(M,N)
        \le
        d_I(M,N).
    \]
    By Theorem~\ref{thm:uniform-erosion-distance}, $d_E^{\mathcal{C}_n}(M, N) = \lVert \UL_M - \UL_N \rVert_\infty$. Corollary~\ref{cor:directional-weighted-comparison} gives $\lVert \UL_M - \UL_N \rVert_\infty \le \lVert \WL_M - \WL_N \rVert_\infty$, while Theorem~\ref{thm:weighted-erosion-comparison} gives $\lVert \WL_M - \WL_N \rVert_\infty \le d_E^{\mathcal{R}_n}(M, N) = d_{\LI}(M, N)$. This proves the first chain.

    The second chain follows from Theorem~\ref{thm:directional-erosion-bound} after substituting $d_E^{\mathcal{R}_n}(M, N) = d_{\LI}(M, N)$ and using $d_{\LI}(M, N) \le d_I(M, N)$. The third follows in the same way from Theorem~\ref{thm:integrated-erosion-bound}.
\end{proof}

The factor $1/\alpha$ in the restricted directional estimate is optimal, and no uniform bound holds when directions arbitrarily close to the coordinate axes are allowed, as the following example shows.

\begin{example}[Optimality of the directional bound and failure of unrestricted stability]
\label{ex:directional-interleaving-instability}
    Let $0 < b < L$, set $J_{L,b} \coloneqq (-L, L) \times (-b, b) \subseteq \R^2$, and let $M_{L,b} \coloneqq \mathds{1}_{J_{L,b}}$. Let $N_0$ denote the zero persistence module.

    An $\varepsilon$-interleaving between $M_{L,b}$ and $N_0$ exists if and only if the transition morphism $\varphi_{M_{L,b}}^{2\varepsilon\mathbf{1}_2}$ is zero. If $\varepsilon \ge b$, no point $x \in J_{L,b}$ satisfies $x + 2\varepsilon\mathbf{1}_2 \in J_{L,b}$, so this transition morphism is zero. If $0 \le \varepsilon < b$, both $(-\varepsilon, -\varepsilon)$ and $(\varepsilon, \varepsilon)$ belong to $J_{L,b}$, and the corresponding component of the transition morphism is the identity on $\F$. Therefore, $d_I(M_{L,b}, N_0) = b$.

    At $t = \mathbf{0}$ and $k = 1$, the landscape interval is $\I_{M_{L,b}}^1(\mathbf{0}) = [0, L) \times [0, b)$. In particular, $\UL_{M_{L,b}}(1, \mathbf{0}) = b$, whereas $\UL_{N_0}(1, \mathbf{0}) = 0$. Together with the preceding results, this shows that the relevant quantities satisfy the chain below.
    \[
        b
        \le
        \lVert \UL_{M_{L,b}} - \UL_{N_0} \rVert_\infty
        =
        d_E^{\mathcal{C}_2}(M_{L,b}, N_0)
        \le
        d_E^{\mathcal{R}_2}(M_{L,b}, N_0)
        =
        d_{\LI}(M_{L,b}, N_0)
        \le
        d_I(M_{L,b}, N_0)
        =
        b
    \]
    Thus, all quantities in this chain equal $b$. Now let $b/L \le \alpha \le 1$ and set $v_\alpha \coloneqq (1, \alpha)$. Then $\DL_{M_{L,b}}(v_\alpha, 1, \mathbf{0}) = \frac{b}{\alpha}$.
    Since $v_\alpha \in S_{\infty,+}^{1}(\alpha)$, Theorem~\ref{thm:directional-erosion-bound} gives
    \[
        \frac{b}{\alpha}
        \le
        \lVert \DL_{M_{L,b}} - \DL_{N_0} \rVert_{\infty,\alpha}
        \le
        \frac{1}{\alpha}
        d_E^{\mathcal{R}_2}(M_{L,b}, N_0)
        =
        \frac{b}{\alpha}.
    \]
    Therefore, equality holds throughout, and the factor $1/\alpha$ cannot be improved.

    Finally, take $\alpha = b/L$. Then $v_\alpha = (1, b/L)$ has strictly positive coordinates and $\DL_{M_{L,b}}(v_\alpha, 1, \mathbf{0}) = L$. Consequently, the ratio of the unrestricted directional distance to either $d_E^{\mathcal{R}_2}(M_{L,b}, N_0)$ or $d_I(M_{L,b}, N_0)$ is at least $L/b$. Since $L/b$ can be arbitrarily large, no uniform Lipschitz bound holds when all directions are allowed. Taking $L = 1$ and letting $b \to 0$ further shows that the unrestricted directional landscape is not continuous at the zero persistence module with respect to either distance.
\end{example}

%% file: sections/finite_characterization_of_landscape_intervals.tex
\section{Finite Characterization of Landscape Intervals}
\label{sec:finite-characterization}

This section studies when the closure of a landscape interval can be reconstructed from finitely many directional values. We first consider a nonempty bounded landscape interval and show that its closure is determined by its radial boundary. We then reduce this reconstruction to the maximal points of the closure, which may be infinite in number. Finally, for bounded-below, rectangular-constructible persistence modules, we prove that only finitely many maximal points, and hence finitely many directions, are required to determine the closure of each nonempty landscape interval. The next section considers approximations from prescribed finite sets of directions when the maximal directions are not known.

Throughout this section, unless stated otherwise, fix $M \in \Vect^{\R^n}$, $t \in \R^n$, and $k \in \N_{>0}$. For brevity, write $\I \coloneqq \I_M^k(t)$.

\subsection{Radial characterization of landscape intervals}
\label{subsec:radial-characterization}

Throughout this subsection, assume that $\I$ is nonempty and bounded. By Proposition~\ref{prop:landscape-interval-structure}, $\mathbf{0} \in \I$, and $\I$ is closed under decreasing coordinates: if $h \in \I$ and $\mathbf{0} \preceq h' \preceq h$, then $h' \in \I$. Consequently, the intersection of $\I$ with each ray from the origin is a segment beginning at $\mathbf{0}$, possibly without its outer endpoint. The following map records that endpoint.

\begin{definition}[Radial boundary map]
\label{def:radial-boundary-map}
    The \emph{radial boundary map} of $\I$ is the function
    \[
        \beta
        \colon
        S_{\infty,+}^{n-1}
        \longrightarrow
        \R_{\ge 0}^n
    \]
    defined by
    \[
        \beta(v)
        \coloneqq
        \left(
            \sup
            \left\{
                \epsilon \ge 0 :
                \epsilon v \in \I
            \right\}
        \right)v
        =
        \DL_M(v, k, t)\,v.
    \]
    We call $\beta(v)$ the \emph{radial boundary point} of $\I$ in the direction $v$.
\end{definition}

Lemma~\ref{lem:integrated-radial-representation} used the radial extents of $\I$ to express its measure. The following proposition shows that the corresponding boundary points determine $\overline{\I}$ exactly.

\begin{proposition}[Radial reconstruction]
\label{prop:radial-reconstruction}
    The closure of $\I$ satisfies
    \[
        \overline{\I}
        =
        \bigcup_{v \in S_{\infty,+}^{n-1}}
        \left\{
            h \in \R_{\ge 0}^n :
            h \preceq \beta(v)
        \right\}.
    \]
\end{proposition}

\begin{proof}
    For each $v \in S_{\infty,+}^{n-1}$, the set $\I \cap \{\epsilon v : \epsilon \ge 0\}$ is a segment beginning at $\mathbf{0}$ whose closure has outer endpoint $\beta(v)$. Hence, $\beta(v) \in \overline{\I}$. Since $\overline{\I}$ is closed under decreasing coordinates, every $h \preceq \beta(v)$ also belongs to $\overline{\I}$. Conversely, let $h \in \overline{\I} \setminus \{\mathbf{0}\}$, and set $r \coloneqq \lVert h \rVert_\infty$ and $v \coloneqq h/r$. For every $0 \le s < r$, a point of $\I$ sufficiently close to $h$ lies coordinatewise above $s v$, so $s v \in \I$. Therefore, $r \le \DL_M(v, k, t)$ and $h = r v \preceq \beta(v)$. The case $h = \mathbf{0}$ is immediate.
\end{proof}

The reconstruction formula constrains the relative positions of radial boundary points. In particular, one radial boundary point cannot be strictly smaller than another in every coordinate.

\begin{lemma}[Coordinatewise comparison of radial boundary points]
\label{lem:no-strict-radial-separation}
    For any $v_1, v_2 \in S_{\infty,+}^{n-1}$, there exist $i, j \in \{1, \ldots, n\}$ such that
    \[
        \beta(v_1)_i \le \beta(v_2)_i
        \qquad\text{and}\qquad
        \beta(v_1)_j \ge \beta(v_2)_j.
    \]
    Equivalently, neither radial boundary point is strictly smaller than the other in every coordinate.
\end{lemma}

\begin{proof}
    Write $\beta(v_\ell) = r_\ell v_\ell$ for $\ell = 1, 2$. Suppose, to the contrary, that $\beta(v_1)_i < \beta(v_2)_i$ for every $i \in \{1, \ldots, n\}$. We may then choose $\delta > 0$ sufficiently small that $(r_1 + \delta)v_1 \preceq \beta(v_2)$. Since $\beta(v_2) \in \overline{\I}$ and $\overline{\I}$ is closed under decreasing coordinates, we have $(r_1 + \delta)v_1 \in \overline{\I}$. The argument used in the proof of Proposition~\ref{prop:radial-reconstruction} then gives $s v_1 \in \I$ for every $0 \le s < r_1 + \delta$. Choosing $s$ such that $r_1 < s < r_1 + \delta$ contradicts the definition of $r_1 = \DL_M(v_1, k, t)$. Thus, $\beta(v_1)$ cannot be strictly smaller than $\beta(v_2)$ in every coordinate. Interchanging $v_1$ and $v_2$ proves the other inequality.
\end{proof}

The preceding lemma also controls how the radial boundary varies with the direction and yields continuity at every strictly positive direction.

\begin{proposition}[Continuity of the radial boundary map]
\label{prop:radial-boundary-continuity}
    The radial boundary map $\beta$ is continuous on $S_{\infty,+}^{n-1} \cap \R_{>0}^n$ with respect to the $\ell^\infty$-norm.
\end{proposition}

\begin{proof}
    Fix $u_0 \in S_{\infty,+}^{n-1} \cap \R_{>0}^n$, and write $\beta(u_0) = r_0 u_0$ and $\beta(u) = r u$. For $u$ sufficiently close to $u_0$, every coordinate of $u$ is positive. By Lemma~\ref{lem:no-strict-radial-separation}, there exist $i, j \in \{1, \ldots, n\}$ such that $r_0(u_0)_i \le r u_i$ and $r_0(u_0)_j \ge r u_j$. Hence,
    \[
        r_0
        \min_{1 \le \ell \le n}
        \frac{(u_0)_\ell}{u_\ell}
        \le
        r
        \le
        r_0
        \max_{1 \le \ell \le n}
        \frac{(u_0)_\ell}{u_\ell}.
    \]
    As $u \to u_0$, both bounds converge to $r_0$, so $r \to r_0$. Since $\lVert u \rVert_\infty = 1$, it follows that
    \[
        \lVert \beta(u) - \beta(u_0) \rVert_\infty
        \le
        |r - r_0|
        +
        r_0 \lVert u - u_0 \rVert_\infty
        \longrightarrow
        0.
    \]
    Therefore, $\beta$ is continuous at $u_0$.
\end{proof}

Taking the $\ell^\infty$-norm of the radial boundary map gives the corresponding continuity statement for the directional landscape.

\begin{corollary}[Continuity of the directional landscape]
\label{cor:directional-landscape-continuity}
    The function $\DL_M(\bullet, k, t)$ is continuous on $S_{\infty,+}^{n-1} \cap \R_{>0}^n$.
\end{corollary}

\begin{proof}
    Since $\lVert v \rVert_\infty = 1$, we have $\DL_M(v, k, t) = \lVert \beta(v) \rVert_\infty$, so the result follows from Proposition~\ref{prop:radial-boundary-continuity}.
\end{proof}

\subsection{Reconstruction from maximal directions}
\label{subsec:maximal-directions}

We continue to assume that $\I$ is nonempty and bounded. Proposition~\ref{prop:radial-reconstruction} reconstructs $\overline{\I}$ using radial boundary points from every direction. However, a boundary point that is dominated coordinatewise by another contributes nothing new to this reconstruction. This motivates focusing on the maximal points of $\overline{\I}$ and their associated directions.

\begin{definition}[Maximal landscape points and directions]
\label{def:maximal-landscape-directions}
    A point $p \in \overline{\I}$ is \emph{maximal} if, whenever $q \in \overline{\I}$ satisfies $p \preceq q$, one has $q = p$. We denote the set of maximal points of $\overline{\I}$ by $\operatorname{Max}(\overline{\I})$. For each $p \in \operatorname{Max}(\overline{\I}) \setminus \{\mathbf{0}\}$, define $v_p \coloneqq p/\lVert p \rVert_\infty \in S_{\infty,+}^{n-1}$. We call $v_p$ the \emph{maximal landscape direction associated with $p$}, and write
    \[
        V_{\max}
        \coloneqq
        \left\{
            v_p :
            p \in
            \operatorname{Max}(\overline{\I})
            \setminus
            \{\mathbf{0}\}
        \right\}.
    \]
\end{definition}

For every
$p \in \operatorname{Max}(\overline{\I})
\setminus \{\mathbf{0}\}$, the radial boundary map recovers $p$ from
its associated direction: $\beta(v_p) = p$. Indeed,
Proposition~\ref{prop:radial-reconstruction} gives a direction
$v \in S_{\infty,+}^{n-1}$ such that $p \preceq \beta(v)$. Since
$\beta(v) \in \overline{\I}$ and $p$ is maximal, we must have
$\beta(v) = p$. Normalizing this equality gives $v = v_p$. Thus,
$p \mapsto v_p$ is a bijection from
$\operatorname{Max}(\overline{\I}) \setminus \{\mathbf{0}\}$ to
$V_{\max}$, with inverse $v \mapsto \beta(v)$.

Because $\I$ is bounded, its closure $\overline{\I}$ is compact. This ensures that every point of $\overline{\I}$ is dominated by a maximal point.

\begin{lemma}[Domination by a maximal point]
\label{lem:maximal-domination}
    For every $h \in \overline{\I}$, there exists $p \in \operatorname{Max}(\overline{\I})$ such that $h \preceq p$. In particular, $\operatorname{Max}(\overline{\I})$ is nonempty.
\end{lemma}

\begin{proof}
    Fix $h \in \overline{\I}$. The set $\{q \in \overline{\I} : h \preceq q\}$ is nonempty and compact, so the continuous function $q \mapsto \sum_{i = 1}^n q_i$ attains its maximum on this set at some point $p$. If $p \preceq q$ for some $q \in \overline{\I}$ with $q \neq p$, then $q$ belongs to the same set and $\sum_{i = 1}^n q_i > \sum_{i = 1}^n p_i$, contradicting the choice of $p$. Thus, $p$ is maximal and $h \preceq p$.
\end{proof}

For $p \in \R_{\ge 0}^n$, let
$[\mathbf{0}, p]
\coloneqq
\{h \in \R_{\ge 0}^n : h \preceq p\}
=
\prod_{i = 1}^n [0, p_i]$
denote the closed rectangle with opposite corners $\mathbf{0}$ and $p$. The maximal points now reduce the radial reconstruction to a canonical family of directions.

\begin{theorem}[Maximal-direction reconstruction]
\label{thm:maximal-direction-reconstruction}
    The closure of the landscape interval satisfies
    \[
        \overline{\I}
        =
        \bigcup_{p \in \operatorname{Max}(\overline{\I})}
        [\mathbf{0}, p].
    \]
    If $\overline{\I} \neq \{\mathbf{0}\}$, then $V_{\max}$ is nonempty and
    \[
        \overline{\I}
        =
        \bigcup_{v \in V_{\max}}
        [\mathbf{0}, \beta(v)].
    \]
    If $\overline{\I} = \{\mathbf{0}\}$, then
    $\beta(v) = \mathbf{0}$ for every
    $v \in S_{\infty,+}^{n-1}$, and any single direction gives the same reconstruction.
\end{theorem}

\begin{proof}
    Since $\overline{\I}$ is closed under decreasing coordinates,
    $[\mathbf{0}, p] \subseteq \overline{\I}$ for every
    $p \in \operatorname{Max}(\overline{\I})$. Conversely, Lemma~\ref{lem:maximal-domination} gives, for every $h \in \overline{\I}$, a maximal point $p$ such that $h \preceq p$, and hence $h \in [\mathbf{0}, p]$. This proves
    the first equality. If $\overline{\I} \neq \{\mathbf{0}\}$, then $\mathbf{0}$ is not
    maximal, and the point--direction correspondence above converts the first equality into the second and shows that $V_{\max}$ is nonempty. Finally, if
    $\overline{\I} = \{\mathbf{0}\}$, then
    $\beta(v) \in \overline{\I}$ implies
    $\beta(v) = \mathbf{0}$ for every
    $v \in S_{\infty,+}^{n-1}$.
\end{proof}

The preceding theorem reconstructs $\overline{\I}$ from
$V_{\max}$, but it does not imply that $V_{\max}$ is finite. The following example shows that a nonempty bounded landscape interval may have infinitely many maximal directions.

\begin{example}[Infinitely many maximal landscape directions]
\label{ex:infinitely-many-maximal-directions}
    Consider the interval module from
    Example~\ref{ex:triangular-landscape-interval}. At
    $t = \mathbf{0}$ and $k = 1$, its landscape interval is
    \[
        \I_M^1(\mathbf{0})
        =
        \left\{
            h \in \R_{\ge 0}^2 :
            h_1 + h_2 \le a
        \right\}.
    \]
    This landscape interval is closed, and its maximal points are
    \[
        \operatorname{Max}
        \bigl(\I_M^1(\mathbf{0})\bigr)
        =
        \left\{
            (s, a - s) :
            0 \le s \le a
        \right\}.
    \]
    Distinct points on this line segment have distinct normalized directions. Therefore, $V_{\max}$ is infinite even though the landscape interval is nonempty and compact.
\end{example}

\subsection{Rectangular constructibility and finite grid structure}
\label{subsec:constructibility-finite-grid}

The preceding example shows that compactness alone does not guarantee finitely many maximal landscape directions. To obtain finiteness, we impose constructibility with respect to a finite rectangular stratification. Constructibility conditions based on finite stratifications are standard in the theory of modules over posets and multiparameter persistence
\cite{Miller2025Homological,botnan2023introductionmultiparameterpersistence}. The condition used here does not require pointwise finite-dimensionality and will lead to a finite rectangular description of landscape intervals.

\begin{definition}[Finite rectangular stratification]
\label{def:rectangular-stratification}
    A \emph{finite rectangular stratification} of $\R^n$ is a finite family $\mathcal{S}$ of pairwise disjoint subsets of $\R^n$ satisfying the following conditions:
    \begin{enumerate}
        \item $\R^n = \bigsqcup_{S \in \mathcal{S}} S$.

        \item Every $S \in \mathcal{S}$ has the form
        $S = J_1 \times \cdots \times J_n$, where each $J_i$ is
        either a singleton or a nonempty open interval, possibly unbounded.

        \item If $S, T \in \mathcal{S}$ and
        $T \cap \overline{S} \neq \varnothing$, then
        $T \subseteq \overline{S}$.
    \end{enumerate}
    The \emph{dimension} of $S = J_1 \times \cdots \times J_n$ is the number of its nonsingleton factors.
\end{definition}

We now define constructibility relative to a finite rectangular stratification.

\begin{definition}[Rectangular-constructible persistence module]
\label{def:rectangular-constructible}
    Let $\mathcal{S}$ be a finite rectangular stratification of $\R^n$. A persistence module $M \in \Vect^{\R^n}$ is \emph{$\mathcal{S}$-rectangular-constructible} if, for every
    stratum $C \in \mathcal{S}$ and every $p, q \in C$ with $p \preceq q$, the structure map
    $M_p^q \colon M(p) \longrightarrow M(q)$ is an isomorphism.

    We call $M$ \emph{rectangular-constructible} if it is $\mathcal{S}$-rectangular-constructible for some finite rectangular stratification $\mathcal{S}$ of $\R^n$.
\end{definition}

Rectangular constructibility controls where a persistence module can change, but it does not guarantee that its landscape intervals are bounded. Although the finite-reconstruction argument below requires only that the fixed landscape interval $\I$ be bounded, we impose the following convenient module-level support condition, which guarantees this boundedness for every $t$ and $k$.

\begin{definition}[Bounded below]
\label{def:bounded-below}
    A persistence module $M \in \Vect^{\R^n}$ is
    \emph{bounded below} if there exists $a \in \R^n$ such that $M(p) = 0$ whenever $a \npreceq p$. Equivalently, $M(p)$ can be nonzero only if $a \preceq p$.
\end{definition}

\begin{proposition}[Boundedness of landscape intervals]
\label{prop:bounded-landscape-interval}
    Suppose that $M \in \Vect^{\R^n}$ is bounded below. Then, for every $t \in \R^n$ and $k \in \N_{>0}$, the landscape interval $\I_M^k(t)$ is either empty or bounded.
\end{proposition}

\begin{proof}
    Let $a \in \R^n$ satisfy $M(p) = 0$ whenever
    $a \npreceq p$, and suppose that $\I_M^k(t)$ is nonempty. If $h \in \I_M^k(t)$, then
    $\rk_M(t - h, t + h) \ge k > 0$, so $M(t - h) \neq 0$. Therefore, $a \preceq t - h$, and hence
    $\mathbf{0} \preceq h \preceq t - a$. Thus,
    $\I_M^k(t) \subseteq [\mathbf{0}, t - a]$, proving that $\I_M^k(t)$ is bounded.
\end{proof}

We next define the grid structure used in the reconstruction.

\begin{definition}[Finite rectangular grid stratification]
\label{def:finite-rectangular-grid}
    For each $i \in \{1, \ldots, n\}$, let $A_i \subseteq \R$ be a finite set, possibly empty, and let $\mathcal{P}(A_i)$ be the partition of $\R$ consisting of the singleton sets $\{a\}$ for $a \in A_i$ and the connected components of $\R \setminus A_i$. The \emph{finite rectangular grid stratification} induced by $A_1, \ldots, A_n$ is
    \[
        \mathcal{G}(A_1, \ldots, A_n)
        \coloneqq
        \left\{
            J_1 \times \cdots \times J_n :
            J_i \in \mathcal{P}(A_i)
            \text{ for every } i
        \right\}.
    \]
    The elements of $A_i$ are called the \emph{breakpoints} in the $i$th coordinate.
\end{definition}

For example, let $K$ be a finite simplicial complex with two filtration functions $f_1,f_2 \colon K \to \R$, so that $f_i(\tau) \le f_i(\sigma)$ whenever $\tau$ is a face of $\sigma$. At $(a,b) \in \R^2$, include a simplex $\sigma$ when $f_1(\sigma) \le a$ and $f_2(\sigma) \le b$. This defines a bifiltration whose simplices can change only when a coordinate crosses one of the corresponding function values. Taking $A_i \coloneqq f_i(K)$ therefore gives a grid $\mathcal{G}(A_1,A_2)$ on whose strata the bifiltration is constant. Applying homology in any fixed degree gives a rectangular-constructible persistence module.

We say that a stratification $\mathcal{G}$ \emph{refines} $\mathcal{S}$ if every stratum of $\mathcal{G}$ is contained in a stratum of $\mathcal{S}$. Since landscape intervals use structure maps with endpoints $t - h$ and $t + h$, we need a grid compatible with both $\mathcal{S}$ and its reflection about $t$.

\begin{lemma}[Common grid refinement]
\label{lem:common-reflected-grid}
    Let $\mathcal{S}$ be a finite rectangular stratification of $\R^n$, and let $t \in \R^n$. Define the reflection $\rho_t \colon \R^n \longrightarrow \R^n$ by $\rho_t(x) = 2t - x$, and set $\rho_t(\mathcal{S}) \coloneqq \{\rho_t(S) : S \in \mathcal{S}\}$. There exists a finite rectangular grid stratification $\mathcal{G}_t$ that refines both $\mathcal{S}$ and $\rho_t(\mathcal{S})$ and has $t_i$ as a breakpoint in the $i$th coordinate for every $i \in \{1, \ldots, n\}$. Consequently, the translated grid
    \[
        \mathcal{T}_t
        \coloneqq
        \{G - t : G \in \mathcal{G}_t\}
    \]
    has $0$ as a breakpoint in every coordinate. Moreover, if $M$ is $\mathcal{S}$-rectangular-constructible, then it is also $\mathcal{G}_t$-rectangular-constructible.
\end{lemma}

\begin{proof}
    The reflected family $\rho_t(\mathcal{S})$ is again a finite rectangular stratification. For each $i \in \{1, \ldots, n\}$, let $A_i$ consist of $t_i$ together with all singleton values and finite endpoints appearing in the $i$th coordinate factors of $\mathcal{S}$ or $\rho_t(\mathcal{S})$, and set $\mathcal{G}_t \coloneqq \mathcal{G}(A_1, \ldots, A_n)$. If $G \in \mathcal{G}_t$ meets a stratum $C$ of either $\mathcal{S}$ or $\rho_t(\mathcal{S})$, then every coordinate factor of $G$ is contained in the corresponding coordinate factor of $C$. Hence, $G \subseteq C$, so $\mathcal{G}_t$ refines both stratifications.

    Translating by $-t$ sends each breakpoint $t_i$ to $0$, giving the asserted property of $\mathcal{T}_t$. Finally, because $\mathcal{G}_t$ refines $\mathcal{S}$, comparable points in the same stratum of $\mathcal{G}_t$ lie in the same stratum of $\mathcal{S}$. Thus, $\mathcal{S}$-rectangular-constructibility of $M$ implies its $\mathcal{G}_t$-rectangular-constructibility.
\end{proof}

For the remainder of this section, assume that $M$ is bounded below and $\mathcal{S}$-rectangular-constructible for a finite rectangular stratification $\mathcal{S}$, and assume that $\I$ is nonempty. By Proposition~\ref{prop:bounded-landscape-interval}, $\I$ is bounded. Let $\mathcal{G}_t$ and $\mathcal{T}_t$ be the grid stratifications from Lemma~\ref{lem:common-reflected-grid}. As $h$ varies within a stratum of $\mathcal{T}_t$, the points $t - h$ and $t + h$ remain in fixed constructibility strata. This gives the following rank-constancy result.

\begin{lemma}[Rank constancy on translated-grid strata]
\label{lem:centered-rank-grid-constant}
    If $h, h' \in \R_{\ge 0}^n$ belong to the same stratum of $\mathcal{T}_t$, then
    \[
        \rk_M(t - h, t + h)
        =
        \rk_M(t - h', t + h').
    \]
\end{lemma}

\begin{proof}
    Let $d$ be the coordinatewise maximum of $h$ and $h'$. Since each stratum of $\mathcal{T}_t$ is a product of intervals and singleton sets, $d$ belongs to the same stratum as $h$ and $h'$. Moreover, $\mathbf{0} \preceq h \preceq d$, so $t - d \preceq t - h \preceq t + h \preceq t + d$. The points $t + h$ and $t + d$ lie in the same stratum of $\mathcal{G}_t$, so $M_{t + h}^{t + d}$ is an isomorphism. Because $\mathcal{G}_t$ also refines $\rho_t(\mathcal{S})$, the reflected points $t - h$ and $t - d$ lie in the same stratum of $\mathcal{S}$, so $M_{t - d}^{t - h}$ is also an isomorphism. By functoriality,
    \[
        M_{t - d}^{t + d}
        =
        M_{t + h}^{t + d}
        \circ
        M_{t - h}^{t + h}
        \circ
        M_{t - d}^{t - h}.
    \]
    Hence, $\rk_M(t - d, t + d) = \rk_M(t - h, t + h)$. Applying the same argument to $h'$ proves the result.
\end{proof}

Rank constancy implies that $\I$ is a finite union of strata of $\mathcal{T}_t$.

\begin{proposition}[Finite grid decomposition]
\label{prop:finite-grid-decomposition}
    Let
    \[
        \mathcal{C}
        \coloneqq
        \left\{
            T \in \mathcal{T}_t :
            T \subseteq \R_{\ge 0}^n
            \text{ and }
            T \cap \I \neq \varnothing
        \right\}.
    \]
    Then $\mathcal{C}$ is a finite nonempty collection of bounded strata,
    \[
        \I
        =
        \bigsqcup_{T \in \mathcal{C}} T,
        \qquad
        \overline{\I}
        =
        \bigcup_{T \in \mathcal{C}} \overline{T},
    \]
    and each $\overline{T}$ is a compact, possibly degenerate, rectangle.
\end{proposition}

\begin{proof}
    Since $0$ is a breakpoint in every coordinate, each point of $\R_{\ge 0}^n$ belongs to a unique stratum of $\mathcal{T}_t$ contained in $\R_{\ge 0}^n$. By Lemma~\ref{lem:centered-rank-grid-constant}, the rank $\rk_M(t - h, t + h)$ is constant on each such stratum. Hence, every stratum contained in $\R_{\ge 0}^n$ either lies entirely in $\I$ or is disjoint from $\I$, proving the first equality. The collection $\mathcal{C}$ is finite because $\mathcal{T}_t$ is finite and nonempty because $\I$ is nonempty. Moreover, every $T \in \mathcal{C}$ is contained in the bounded set $\I$, so $\overline{T}$ is a compact, possibly degenerate, rectangle. Finally, since $\mathcal{C}$ is finite, taking closures in the first equality gives $\overline{\I} = \bigcup_{T \in \mathcal{C}} \overline{T}$.
\end{proof}

\subsection{Finite maximal-direction reconstruction}
\label{subsec:finite-maximal-direction-reconstruction}

The finite grid decomposition expresses $\overline{\I}$ as a finite union of compact rectangles. For each $T \in \mathcal{C}$, let $c_T$ denote the upper corner of $\overline{T}$. Set $\mathcal{U} \coloneqq \{c_T : T \in \mathcal{C}\}$, and let $P$ be the set of maximal elements of $\mathcal{U}$ with respect to $\preceq$.

\begin{proposition}[Finite corner reconstruction]
\label{prop:finite-corner-reconstruction}
    The set $P$ is precisely $\operatorname{Max}(\overline{\I})$. In particular, $P$ is finite and nonempty, and
    \[
        \overline{\I}
        =
        \bigcup_{p \in P}
        [\mathbf{0}, p].
    \]
\end{proposition}

\begin{proof}
    For every $T \in \mathcal{C}$, we have $c_T \in \overline{T} \subseteq \overline{\I}$, and every point of $\overline{T}$ is dominated by $c_T$. Since $\overline{\I}$ is closed under decreasing coordinates,
    \[
        \overline{\I}
        =
        \bigcup_{u \in \mathcal{U}}
        [\mathbf{0}, u]
        =
        \bigcup_{p \in P}
        [\mathbf{0}, p],
    \]
    where the second equality follows because every element of the finite set $\mathcal{U}$ is dominated by an element of $P$. If $p \in P$ and $p \preceq h \in \overline{\I}$, then $h \preceq u$ for some $u \in \mathcal{U}$. Hence $p \preceq u$, so maximality of $p$ in $\mathcal{U}$ gives $p = u = h$. Thus, every $p \in P$ is maximal in $\overline{\I}$. Conversely, every maximal $h \in \overline{\I}$ is dominated by some $p \in P$, and therefore $h = p$. Hence, $P = \operatorname{Max}(\overline{\I})$.
\end{proof}

Proposition~\ref{prop:finite-corner-reconstruction} identifies the maximal points of $\overline{\I}$ with finitely many grid corners. Combining this description with Theorem~\ref{thm:maximal-direction-reconstruction} gives the main result of the section.

\begin{theorem}[Finite maximal-direction reconstruction]
\label{thm:finite-maximal-direction-reconstruction}
    Suppose that $M$ is bounded below and rectangular-constructible and that $\I$ is nonempty. If $\overline{\I} \neq \{\mathbf{0}\}$, then $V_{\max}$ is finite and nonempty, and
    \[
        \overline{\I}
        =
        \bigcup_{v \in V_{\max}}
        [\mathbf{0}, \beta(v)].
    \]
    If $\overline{\I} = \{\mathbf{0}\}$, then
    $\overline{\I}$ can be reconstructed from any single direction.
\end{theorem}

\begin{proof}
    By Proposition~\ref{prop:finite-corner-reconstruction}, $\operatorname{Max}(\overline{\I})$ is finite and nonempty. If $\overline{\I} \neq \{\mathbf{0}\}$, then $\mathbf{0}$ is not maximal. The point--direction correspondence following Definition~\ref{def:maximal-landscape-directions} therefore shows that $V_{\max}$ is finite and nonempty. The reconstruction formula follows from Theorem~\ref{thm:maximal-direction-reconstruction}. If $\overline{\I} = \{\mathbf{0}\}$, then $\beta(v) \in \overline{\I}$ forces $\beta(v) = \mathbf{0}$ for every
    $v \in S_{\infty,+}^{n-1}$, so any single direction reconstructs $\overline{\I}$.
\end{proof}

%% file: sections/finite_direction_approximation_and_curve_structure.tex
\section{Finite-Direction Approximation and Integrated Landscape Curves}
\label{sec:finite-directional-approximation}

This section develops the finite-direction approximations used in the numerical experiments of Section~\ref{sec:applications-and-experiments}. The reconstruction in Section~\ref{sec:finite-characterization} uses maximal directions, which need not be known in advance. We therefore construct inner and outer approximations of landscape-interval closures using finitely many directional evaluations. Their volumes give lower and upper bounds for the integrated landscape curves, and we establish convergence under dense sampling. We also derive explicit formulas for curves arising from finite unions of rectangles, making the approximating curves computable and describing their piecewise structure and shape in two parameters.

\subsection{Finite-direction approximation}
\label{subsec:finite-direction-approximations}

Fix a bounded-below persistence module $M \in \Vect^{\R^n}$, $t \in \R^n$, and $k \in \N_{>0}$, and, for brevity, write $\I \coloneqq \I_M^k(t)$. If $\I$ is empty, there is nothing to approximate. We therefore assume throughout this subsection that $\I$ is nonempty. By Proposition~\ref{prop:bounded-landscape-interval}, $\I$ is then bounded. We continue to denote its radial boundary map by $\beta$.

For each $i \in \{1, \ldots, n\}$, let $\mathbf{e}_i$ be the $i$th standard basis vector and set $a_i \coloneqq \DL_M(\mathbf{e}_i, k, t)$. If $h \in \overline{\I}$, then $h_i\mathbf{e}_i \preceq h$. Because $\overline{\I}$ is closed under decreasing coordinates, $h_i\mathbf{e}_i \in \overline{\I}$. Since $\beta(\mathbf{e}_i) = a_i\mathbf{e}_i$ is the radial boundary point on this axis, $h_i \le a_i$. Therefore,
\[
    \overline{\I}
    \subseteq
    B
    \coloneqq
    \prod_{i = 1}^n [0, a_i].
\]

Each sampled radial boundary point $\beta(v)$ gives a rectangle $[\mathbf{0}, \beta(v)]$ contained in $\overline{\I}$. The set $\beta(v) + \R_{>0}^n$ consists of all points that are strictly larger than $\beta(v)$ in every coordinate. Removing these sets from $B$ produces the upper approximation.

\begin{definition}[Finite-direction approximations]
\label{def:finite-direction-approximations}
    Let $V \subseteq S_{\infty,+}^{n-1}$ be finite and nonempty. The \emph{lower} and \emph{upper finite-direction approximations} of $\overline{\I}$ determined by $V$ are
    \[
        \I_V^-
        \coloneqq
        \bigcup_{v \in V}
        [\mathbf{0}, \beta(v)]
        \qquad\text{and}\qquad
        \I_V^+
        \coloneqq
        B
        \setminus
        \bigcup_{v \in V}
        \bigl(\beta(v) + \R_{>0}^n\bigr).
    \]
\end{definition}

\begin{proposition}[Finite-direction bounds]
\label{prop:finite-direction-bounds}
    For every finite nonempty set $V \subseteq S_{\infty,+}^{n-1}$,
    \[
        \I_V^-
        \subseteq
        \overline{\I}
        \subseteq
        \I_V^+.
    \]
    Moreover, if $V \subseteq V'$, then
    \[
        \I_V^-
        \subseteq
        \I_{V'}^-
        \subseteq
        \overline{\I}
        \subseteq
        \I_{V'}^+
        \subseteq
        \I_V^+.
    \]
\end{proposition}

\begin{proof}
    Proposition~\ref{prop:radial-reconstruction} gives $[\mathbf{0}, \beta(v)] \subseteq \overline{\I}$ for every $v \in V$, and hence $\I_V^- \subseteq \overline{\I}$. Now let $h \in \overline{\I}$. Then $h \in B$. If $h \notin \I_V^+$, there exists $v \in V$ such that $\beta(v)_i < h_i$ for every $i$. By Proposition~\ref{prop:radial-reconstruction}, $h \preceq \beta(u)$ for some $u \in S_{\infty,+}^{n-1}$. It follows that $\beta(v)_i < \beta(u)_i$ for every $i$, contradicting Lemma~\ref{lem:no-strict-radial-separation}. Therefore, $\overline{\I} \subseteq \I_V^+$. Finally, if $V \subseteq V'$, the definitions give $\I_V^- \subseteq \I_{V'}^-$ and $\I_{V'}^+ \subseteq \I_V^+$, which proves the final chain.
\end{proof}

To define the corresponding integrated landscape curves and study their structure, we first record that both approximations are finite unions of rectangles.

\begin{lemma}[Finite rectangular form]
\label{lem:upper-approximation-rectangles}
    The sets $\I_V^-$ and $\I_V^+$ are finite unions of compact rectangles of the form $[\mathbf{0}, p]$.
\end{lemma}

\begin{proof}
    The assertion for $\I_V^-$ follows directly from its definition. For the upper approximation, its definition and De Morgan's law give
    \[
    \begin{aligned}
        \I_V^+
        &=
        B
        \setminus
        \bigcup_{v \in V}
        \bigcap_{i = 1}^n
        \left\{
            h \in \R_{\ge 0}^n :
            h_i > \beta(v)_i
        \right\}\\
        &=
        B
        \cap
        \bigcap_{v \in V}
        \bigcup_{i = 1}^n
        \left\{
            h \in \R_{\ge 0}^n :
            h_i \le \beta(v)_i
        \right\}.
    \end{aligned}
    \]
    Thus, for each $v \in V$, at least one of the inequalities $h_i \le \beta(v)_i$ must hold. There are only finitely many ways to choose one such inequality for each $v$. For any fixed choice, the selected inequalities, together with $0 \le h_i \le a_i$ for $i \in \{1,\ldots,n\}$, determine a compact
    rectangle of the form $[\mathbf{0},p]$. Taking the union over all possible choices proves that $\I_V^+$ is a finite union of compact rectangles.
\end{proof}

The integrated landscape is defined using $\I$, whereas the finite-direction approximations enclose $\overline{\I}$. This distinction does not affect their volumes. Indeed, on each ray from the origin, $\I$ and $\overline{\I}$ differ by at most the radial boundary point $\beta(v)$. Applying the $\ell^\infty$-polar-coordinate formula to the indicator function of $\overline{\I} \setminus \I$ therefore gives $\mathcal{L}^n(\overline{\I} \setminus \I) = 0$. Consequently,
\[
    \mathcal{L}^n\bigl(Q_r \cap \I\bigr)
    =
    \mathcal{L}^n\bigl(Q_r \cap \overline{\I}\bigr)
\]
for every $r \in \R_{\ge 0}$.

By Lemma~\ref{lem:upper-approximation-rectangles}, both finite-direction approximations are Lebesgue-measurable. For brevity, write $\gamma(r) \coloneqq \IL_M(r, k, t)$ for the exact integrated landscape curve.

\begin{definition}[Finite-direction approximation curves]
\label{def:finite-direction-approximation-curves}
    Let $V \subseteq S_{\infty,+}^{n-1}$ be finite and nonempty. The \emph{lower} and \emph{upper finite-direction approximation curves} determined by $V$ are the functions $\gamma_V^-, \gamma_V^+ \colon \R_{\ge 0} \longrightarrow [0,\infty)$ defined by
    \[
        \gamma_V^\pm(r)
        \coloneqq
        \begin{cases}
            \displaystyle
            \frac{
                \mathcal{L}^n\bigl(Q_r \cap \I_V^\pm\bigr)
            }{
                r^{n-1}
            },
            & r > 0,\\[8pt]
            0,
            & r = 0.
        \end{cases}
    \]
\end{definition}

\begin{corollary}[Finite-direction curve bounds]
\label{cor:finite-direction-curve-bounds}
    For every finite nonempty set $V \subseteq S_{\infty,+}^{n-1}$ and every $r \in \R_{\ge 0}$,
    \[
        \gamma_V^-(r)
        \le
        \gamma(r)
        \le
        \gamma_V^+(r).
    \]
    Moreover, if $V \subseteq V'$, then
    \[
        \gamma_V^-(r)
        \le
        \gamma_{V'}^-(r)
        \le
        \gamma(r)
        \le
        \gamma_{V'}^+(r)
        \le
        \gamma_V^+(r)
    \]
    for every $r \in \R_{\ge 0}$.
\end{corollary}

\begin{proof}
    At $r = 0$, all three curves vanish. For $r > 0$, the result follows from Proposition~\ref{prop:finite-direction-bounds}, the equality of the truncated volumes of $\I$ and $\overline{\I}$ established above, and monotonicity of Lebesgue measure.
\end{proof}

The lower and upper approximations also provide midpoint estimates. For $\overline{\I}$, placing the estimated boundary halfway between their radial endpoints along each ray gives an error in radial extent of at most half the gap between those endpoints. Similarly, for the integrated landscape curve, the average $(\gamma_V^-(r)+\gamma_V^+(r))/2$ approximates $\gamma(r)$ with absolute error at most $(\gamma_V^+(r)-\gamma_V^-(r))/2$.

The finite-direction approximations are consistent under dense sampling. Suppose that
$V_1 \subseteq V_2 \subseteq \cdots$ are finite nonempty subsets of $S_{\infty,+}^{n-1} \cap \R_{>0}^n$ whose union is dense. Then
\[
    \mathcal{L}^n
    \bigl(
        \I_{V_m}^+
        \setminus
        \I_{V_m}^-
    \bigr)
    \longrightarrow
    0.
\]
Consequently, for every $r \in \R_{\ge 0}$,
\[
    \gamma_{V_m}^-(r)
    \nearrow
    \gamma(r)
    \qquad\text{and}\qquad
    \gamma_{V_m}^+(r)
    \searrow
    \gamma(r).
\]

Indeed, set $\I_\infty^- \coloneqq \bigcup_m \I_{V_m}^-$ and $\I_\infty^+ \coloneqq \bigcap_m \I_{V_m}^+$. The continuity of $\beta$ established in
Proposition~\ref{prop:radial-boundary-continuity}, together with the density of $\bigcup_m V_m$, shows that, along every strictly positive direction, $\I_\infty^-$, $\overline{\I}$, and $\I_\infty^+$ have the same radial extent. They may differ only at the radial boundary point. Directions with at least one zero coordinate have $\sigma$-measure zero, so $\I_\infty^-$, $\overline{\I}$, and $\I_\infty^+$ agree up to sets of Lebesgue measure zero. Since $(\I_{V_m}^-)_m$ is increasing and $(\I_{V_m}^+)_m$ is decreasing inside the bounded rectangle $B$, continuity of Lebesgue measure from below and from above gives $\mathcal{L}^n \bigl( \I_{V_m}^+ \setminus \I_{V_m}^- \bigr) \longrightarrow 0$. For $r > 0$, $0 \le \gamma_{V_m}^+(r)-\gamma_{V_m}^-(r) \le \frac{ \mathcal{L}^n \bigl( \I_{V_m}^+ \setminus \I_{V_m}^- \bigr) }{ r^{n-1} }$, so the curve gap also tends to zero. The monotonicity and convergence to $\gamma(r)$ now follow from Corollary~\ref{cor:finite-direction-curve-bounds}.

In practical terms, each sampled direction provides another point on the radial boundary of the landscape interval. Sampling directions more finely gives a more complete picture of its shape. The gap between the lower and upper approximations then narrows, and the curves computed from them approach the exact integrated landscape curve.

\subsection{Structure of integrated landscape curves}
\label{subsec:integrated-curve-structure}

Definition~\ref{def:finite-direction-approximations} and Lemma~\ref{lem:upper-approximation-rectangles} show that $\I_V^-$ and $\I_V^+$ are finite unions of rectangles with one corner at $\mathbf{0}$. If $M$ is also rectangular-constructible with respect to a finite rectangular stratification, Theorem~\ref{thm:finite-maximal-direction-reconstruction} gives the same form for $\overline{\I}$. Since replacing $\I$ by $\overline{\I}$ does not change its truncated volumes, the exact and approximate curves can all be studied through the following common construction.

Let $P \subseteq \R_{\ge 0}^n$ be finite and nonempty, and define
\[
    \mathcal{R}_P
    \coloneqq
    \bigcup_{p \in P} [\mathbf{0}, p],
    \qquad
    \Gamma_P(r)
    \coloneqq
    \frac{
        \mathcal{L}^n\bigl(Q_r \cap \mathcal{R}_P\bigr)
    }{
        r^{n-1}
    }
    \quad\text{for } r > 0.
\]
Set $\Gamma_P(0) \coloneqq 0$.

The coordinate values of the points in $P$ divide the scale parameter into finitely many ranges. On each range, $\Gamma_P$ has a fixed algebraic form.

\begin{theorem}[Piecewise form of integrated landscape curves]
\label{thm:piecewise-form-integrated-curves}
    The function $\Gamma_P$ is continuous on $[0,\infty)$ and, on $(0,\infty)$, is piecewise of the form
    \[
        \Gamma_P(r)
        =
        \sum_{\ell = -(n-1)}^1 c_\ell r^\ell,
    \]
    where the coefficients may depend on the piece. If $\mathcal{L}^n(\mathcal{R}_P) = 0$, then $\Gamma_P$ is identically zero. Otherwise, there exist $0 < \tau_P \le T_P$ such that
    \[
        \Gamma_P(r) = r
        \quad\text{for } 0 \le r \le \tau_P,
        \qquad
        \Gamma_P(r)
        =
        \frac{\mathcal{L}^n(\mathcal{R}_P)}{r^{n-1}}
        \quad\text{for } r \ge T_P.
    \]
    Moreover, $[0,\infty)$ admits a finite subdivision into intervals on each of which $\Gamma_P$ is strictly increasing, constant, or strictly decreasing.
\end{theorem}

The proof, based on inclusion--exclusion for finite unions of rectangles, is given in Appendix~\ref{app:piecewise-integrated-curve-proof}.

In arbitrary dimensions, Theorem~\ref{thm:piecewise-form-integrated-curves} permits several changes in monotonicity. In two dimensions, however, the staircase
geometry of a finite union of rectangles with one corner at $\mathbf{0}$ forces these pieces to occur in a single order.

\begin{theorem}[Two-parameter curve structure]
\label{thm:two-parameter-curve-shape}
    Let $P \subseteq \R_{\ge 0}^2$ be finite and nonempty. If $\mathcal{L}^2(\mathcal{R}_P) > 0$, then $\Gamma_P$ is unimodal: there exist $0 < r_- \le r_+ < \infty$ such that $\Gamma_P$ is nondecreasing on $[0,r_-]$, constant on $[r_-,r_+]$, and nonincreasing on $[r_+,\infty)$.

    More precisely, $\Gamma_P(r)=r$ on its initial linear segment. After this segment, every piece has the form $\alpha+b/r$, with coefficients depending on the piece. Its strictly increasing, constant, and strictly decreasing pieces have the respective forms
    \[
        \alpha-\frac{c}{r},
        \qquad
        \alpha,
        \qquad
        \alpha+\frac{d}{r},
    \]
    where $c,d > 0$, and these three types can occur only in this order.
\end{theorem}

The proof is given in Appendix~\ref{app:two-parameter-curve-proof}. Its main idea is that, after redundant rectangles are removed, their upper-right corners
form a decreasing staircase. After the initial linear segment of $\Gamma_P$, the truncated area is piecewise affine with nonincreasing slopes. Consequently, the increasing pieces of $\Gamma_P$ precede any constant pieces, which in turn precede the decreasing pieces.

\begin{example}[Exact and approximate two-parameter curves]
\label{ex:two-parameter-exact-approximate-curves}
    Let
    \[
        P
        \coloneqq
        \{(1, 8), (2, 7), (3, 5), (4, 3), (6, 2)\}.
    \]
    For each $p=(p_1,p_2)\in P$, let $M^p$ be the interval module supported on $[-p,p]\coloneqq[-p_1,p_1]\times[-p_2,p_2]$, and set $M\coloneqq\bigoplus_{p\in P}M^p$. For $h\in\R_{\ge 0}^2$, the structure map from $M^p(-h)$ to $M^p(h)$ has rank one if and only if $h\preceq p$. Therefore,
    \[
    \begin{aligned}
        \rk_M(-h,h)
        &=
        \#\{p\in P:h\preceq p\},\\
        \I_M^1(\mathbf{0})
        &=
        \bigcup_{p\in P}[\mathbf{0},p]
        =
        \mathcal{R}_P.
    \end{aligned}
    \]

    Let
    \[
        v_1
        \coloneqq
        \left(\frac{2}{7},1\right),
        \qquad
        v_2
        \coloneqq
        \left(1,\frac{3}{4}\right),
        \qquad
        V
        \coloneqq
        \{v_1,v_2\}.
    \]
    The corresponding radial boundary points are $\beta(v_1)=(2,7)$ and $\beta(v_2)=(4,3)$, while the axial bounding box is $B=[0,6]\times[0,8]$. Hence,
    \[
    \begin{aligned}
        \I_V^-
        &=
        [\mathbf{0},(2,7)]
        \cup
        [\mathbf{0},(4,3)],\\
        \I_V^+
        &=
        [\mathbf{0},(2,8)]
        \cup
        [\mathbf{0},(4,7)]
        \cup
        [\mathbf{0},(6,3)].
    \end{aligned}
    \]

    Write $\gamma(r)\coloneqq\IL_M(r,1,\mathbf{0})$. Computing the normalized truncated areas of $\I_V^-$, $\I_M^1(\mathbf{0})$, and $\I_V^+$ gives
    \[
    \begin{array}{c|c|c|c}
        r
        &
        \gamma_V^-(r)
        &
        \gamma(r)
        &
        \gamma_V^+(r)
        \\ \hline
        0 \le r \le 3
        &
        r
        &
        r
        &
        r
        \\[2pt]
        3 \le r \le 4
        &
        \displaystyle 5-\frac{6}{r}
        &
        \displaystyle 6-\frac{9}{r}
        &
        r
        \\[6pt]
        4 \le r \le 5
        &
        \displaystyle 2+\frac{6}{r}
        &
        \displaystyle 5-\frac{5}{r}
        &
        \displaystyle 7-\frac{12}{r}
        \\[6pt]
        5 \le r \le 6
        &
        \displaystyle 2+\frac{6}{r}
        &
        4
        &
        \displaystyle 7-\frac{12}{r}
        \\[6pt]
        6 \le r \le 7
        &
        \displaystyle 2+\frac{6}{r}
        &
        \displaystyle 2+\frac{12}{r}
        &
        \displaystyle 4+\frac{6}{r}
        \\[6pt]
        7 \le r \le 8
        &
        \displaystyle \frac{20}{r}
        &
        \displaystyle 1+\frac{19}{r}
        &
        \displaystyle 2+\frac{20}{r}
        \\[6pt]
        r \ge 8
        &
        \displaystyle \frac{20}{r}
        &
        \displaystyle \frac{27}{r}
        &
        \displaystyle \frac{36}{r}
    \end{array}
    \]

    The exact curve has two reciprocal-affine increasing pieces after its initial linear segment, followed by a constant piece and then two reciprocal-affine decreasing pieces before its final tail. The table also makes the bounds $\gamma_V^-(r)\le\gamma(r)\le\gamma_V^+(r)$ explicit.
\end{example}

%% file: sections/applications_and_experiments.tex
\section{Applications and Experiments}
\label{sec:applications-and-experiments}

This section illustrates landscape intervals and their associated summaries through synthetic examples and classification experiments. We first use synthetic point clouds to show how a second filtration parameter can help distinguish geometric structure from the effects of noise and incomplete sampling that may be obscured by one-parameter persistence. We then evaluate the uniform and integrated landscapes in two classification settings: shape classification on a subset of SHREC15 and medical-image classification on MedMNIST.

The bifiltrations used below share a common construction. Let $K$ be a finite simplicial or cubical complex, and let $\phi_1$ and $\phi_2$ be real-valued functions on its cells satisfying $\phi_j(\tau) \le \phi_j(\sigma)$ whenever $\tau$ is a face of $\sigma$. For $u=(u_1,u_2) \in \R^2$, define
\[
    K_u
    \coloneqq
    \left\{
        \sigma \in K :
        \phi_1(\sigma) \le u_1
        \text{ and }
        \phi_2(\sigma) \le u_2
    \right\}.
\]
The face condition ensures that $K_u$ is a subcomplex of $K$, and $K_u \subseteq K_v$ whenever $u \preceq v$. Thus, $\{K_u\}_{u \in \R^2}$ is a bifiltration. For each homological degree $p \in \Z_{\ge 0}$, applying homology with coefficients in $\F$ gives the $2$-parameter persistence module $M_p$ defined by $M_p(u) \coloneqq H_p(K_u;\F)$. The complex $K$ and the two filtration functions are specified in each experiment.

For a fixed homological degree $p$, the condition $\rk_{M_p}(t-h,t+h) \ge k$ means that at least $k$ independent $p$-dimensional homology classes survive from $t-h$ to $t+h$. The landscape interval $\I_{M_p}^k(t)$ collects all nonnegative half-width vectors $h$ for which this condition holds. Its geometry therefore records how much the filtration parameters can vary together while these features persist symmetrically about $t$.

In the synthetic examples, we visualize landscape intervals. In the classification experiments, we evaluate the uniform and integrated landscapes separately. The uniform landscape records the largest equal change in all filtration parameters under which at least $k$ independent homological features persist. The integrated landscape curve instead gives a measure-based profile of the admissible coordinate changes as the truncation scale varies.

Throughout this section, whenever a landscape interval or integrated landscape curve is computed numerically, we use the upper finite-direction approximations $\I_V^+$ and $\gamma_V^+$ from Section~\ref{subsec:finite-direction-approximations}. The synthetic examples use seven sampled directions, whereas the two classification experiments use eleven. For readability, we henceforth refer to these computed approximations simply as landscape intervals and integrated landscape curves. The term \emph{integrated representation} refers to a feature vector obtained by sampling the corresponding curve. Each sampled direction set includes the diagonal direction, whose directional value is exactly the uniform landscape.

We do not use the unrestricted directional landscape directly as a classification feature because it has no uniform interleaving-stability bound when directions approach the coordinate axes; see Example~\ref{ex:directional-interleaving-instability}. Directional values nevertheless provide the samples used to construct the landscape intervals and integrated landscape curves studied below.

Apart from the uniform case, we also do not use weighted landscapes. A nonuniform weight rescales the filtration coordinates, so the resulting landscape value is expressed in the chosen relative units rather than directly in the original filtration scales. None of the applications considered here provides a canonical choice of such a rescaling. We therefore restrict the classification experiments to the uniform and integrated representations.

\subsection{Synthetic point-cloud illustrations}
\label{subsec:synthetic-point-clouds}

A one-parameter Vietoris--Rips filtration records how the topology of a point cloud changes as the geometric scale increases, but its output can be strongly affected by how the underlying shape is sampled. Points scattered inside a sampled circle can cause its central loop to disappear at a small scale, whereas a sufficiently large scale can bridge a gap in an incompletely sampled circle and create a persistent loop. To examine these two effects, we augment the Rips filtration with a second parameter that measures local sparsity.

Let $X = \{x_1,\ldots,x_N\} \subseteq \R^d$ be a finite point cloud, and fix $q \in \{1,\ldots,N-1\}$. For each $x_i \in X$, let $x_i^{(1)},\ldots,x_i^{(q)}$ denote its $q$ nearest neighbors, excluding $x_i$ itself. Define the \emph{local sparsity} of $x_i$ by
\[
    \rho_q(x_i)
    \coloneqq
    \left(
        \frac{1}{q}
        \sum_{j=1}^q
        \lVert x_i-x_i^{(j)}\rVert_2^2
    \right)^{1/2}.
\]
We refer to $\rho_q$ as the \emph{$q$NN local-sparsity function}. Smaller values of $\rho_q$ correspond to points in locally denser regions of the point cloud. Throughout this subsection, we use $q=5$.

For $s \ge 0$, let $X_s \coloneqq \{x \in X : \rho_q(x) \le s\}$. For $(r,s) \in \R_{\ge 0}^2$, define
\[
    K_{r,s}
    \coloneqq
    \mathrm{VR}(X_s)_r.
\]
Equivalently, a simplex $\sigma \subseteq X$ belongs to $K_{r,s}$ if and only if
\[
    \max_{x,y \in \sigma}\lVert x-y\rVert_2 \le r
    \qquad\text{and}\qquad
    \max_{x \in \sigma}\rho_q(x) \le s.
\]
If $(r,s) \preceq (r',s')$, then $K_{r,s} \subseteq K_{r',s'}$, so $\{K_{r,s}\}_{(r,s) \in \R_{\ge 0}^2}$ is a bifiltration. This is an instance of a function--Rips bifiltration \cite{LesnickWright2015}. See also \cite{botnan2023introductionmultiparameterpersistence} for broader context. A related radius--codensity bifiltration was studied using multiparameter persistence landscapes in spatial tumor data \cite{VipondEtAl2021Tumors}. Its persistent homology in degree $p$ is the $2$-parameter persistence module $M_p(r,s) \coloneqq H_p(K_{r,s};\F)$. In the examples below, we take $p=1$ and study the persistence of loops.

The two filtration parameters play different roles. The Rips parameter $r$ controls the geometric scale at which points are connected, while the sparsity threshold $s$ controls when points from increasingly sparse regions enter the filtration. Thus, the landscape interval $\I_{M_1}^{1}(t)$ records the simultaneous changes in geometric scale and local sparsity through which a loop persists symmetrically about $t$.

\paragraph{Interior noise.}
We compare a point cloud sampled from a circle with a second point cloud obtained by adding points inside it. For each point cloud, we first compute the one-parameter persistence module $r \mapsto H_1(\mathrm{VR}(X)_r;\F)$. As shown in Figure~\ref{fig:loops-1d-land-noise}, the clean sample produces a prominent $k=1$ persistence-landscape layer. After the interior points are added, the central hole is filled at a substantially smaller Rips scale, and the corresponding landscape has a much smaller peak. The one-parameter landscape therefore assigns much shorter persistence to the principal loop, even though the circular arrangement remains visible in both point clouds.

\begin{figure}[htbp]
    \centering
    \includegraphics[width=\textwidth]{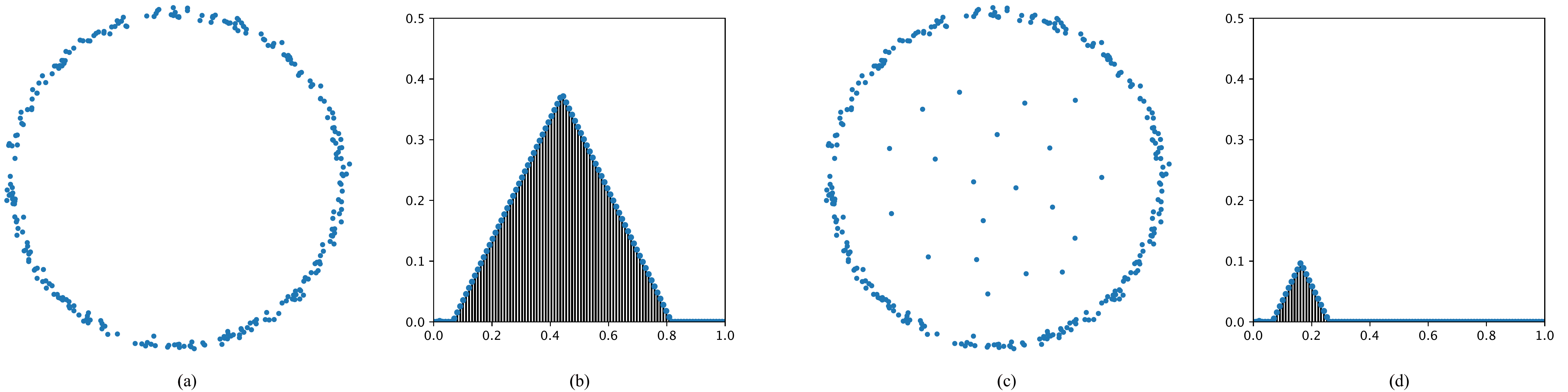}
    \caption{One-parameter persistence landscapes for a clean circle and a circle with interior noise. (a) A point cloud sampled from a circle. (b) Its $k=1$ persistence-landscape layer in homological degree $1$. (c) A point cloud with a similar circular arrangement and additional interior points. (d) Its corresponding persistence-landscape layer. The interior points cause the central hole to be filled at a smaller Rips scale, substantially reducing the height and support of the landscape.}
    \label{fig:loops-1d-land-noise}
\end{figure}

We next apply the Rips--sparsity bifiltration. At small values of the sparsity threshold, the relatively sparse interior points are excluded, allowing the circular structure to be examined before these points enter the filtration. Figure~\ref{fig:2d-land-noise} shows the landscape intervals at several parameter values for the clean and noisy point clouds. Although the shapes of the intervals differ, both point clouds produce substantial intervals at several of the displayed parameter values. The second filtration coordinate therefore records not only how the loop persists as the Rips scale changes, but also how its persistence depends on when the interior points enter the filtration.

\begin{figure}[htbp]
    \centering
    \includegraphics[width=\textwidth]
    {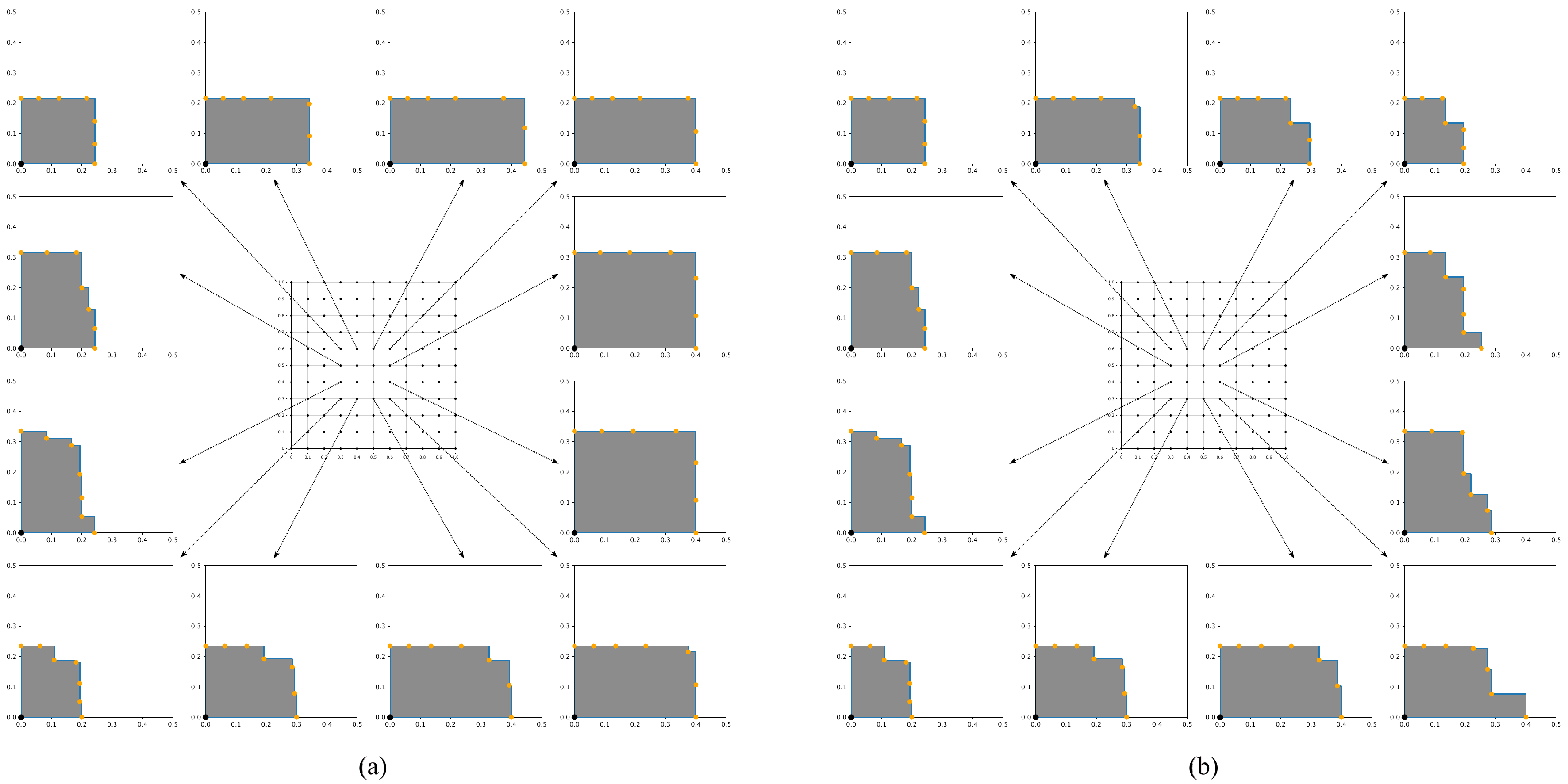}
    \caption{Landscape intervals for the clean circle and the circle with interior noise. The central grids indicate the selected parameter values, with the normalized sparsity threshold on the horizontal axis and the normalized Rips scale on the vertical axis. The surrounding panels show the corresponding landscape intervals for $k=1$ in homological degree $1$, with horizontal and vertical axes representing half-widths in the sparsity threshold and Rips scale, respectively. The gray regions show the computed intervals, the blue curves indicate their boundaries, and the orange points are the seven sampled radial boundary points, including those on the two coordinate axes. Corresponding panels in the two collections use the same coordinate scales. In both cases, substantial landscape intervals occur at several of the displayed parameter values.}
    \label{fig:2d-land-noise}
\end{figure}

\paragraph{Incomplete sampling.}
We compare a point cloud sampled around an entire circle with one in which two arcs are sampled much more sparsely. Once the Rips scale is large enough to bridge the larger separations in these sparsely sampled regions, the resulting edges complete a loop. In this example, the corresponding homology class persists over a range of scales comparable to that of the principal loop in the first point cloud. The resulting one-parameter persistence landscapes are therefore similar despite the different sampling patterns, as shown in Figure~\ref{fig:loops-1d-land-incomplete}.

\begin{figure}[htbp]
    \centering
    \includegraphics[width=\textwidth]{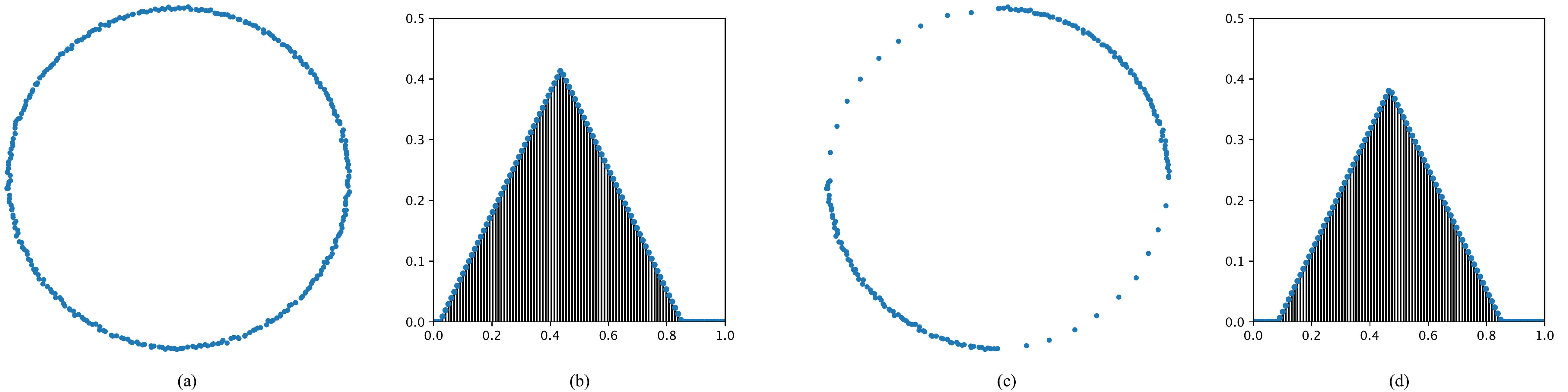}
    \caption{One-parameter persistence landscapes for two circular point clouds with different sampling patterns. (a) A point cloud sampled around the entire circle. (b) Its $k=1$ persistence-landscape layer in homological degree $1$. (c) A point cloud in which two arcs are sampled much more sparsely. (d) Its corresponding persistence-landscape layer. Despite the visible difference between the sampling patterns, the two landscape layers have similar heights and supports.}
    \label{fig:loops-1d-land-incomplete}
\end{figure}

The local-sparsity coordinate reveals the dependence of this loop on the sampling pattern. Points along the two sparsely sampled arcs tend to have larger values of $\rho_q$ because their nearest neighbors are farther away, so they enter the filtration only at larger sparsity thresholds. Forming the loop therefore requires both a sparsity threshold large enough to include these points and a Rips scale large enough to bridge the remaining separations. As shown in Figure~\ref{fig:2d-land-incomplete}, the landscape intervals for the sparsely sampled point cloud are mostly empty or concentrated near the origin, whereas the point cloud sampled around the entire circle produces substantially larger intervals at the displayed parameter values. The two-parameter landscape intervals therefore reveal how the loop depends on the sampling pattern, a distinction obscured by the one-parameter landscapes.

\begin{figure}[!ht]
    \centering
    \includegraphics[width=\textwidth]
    {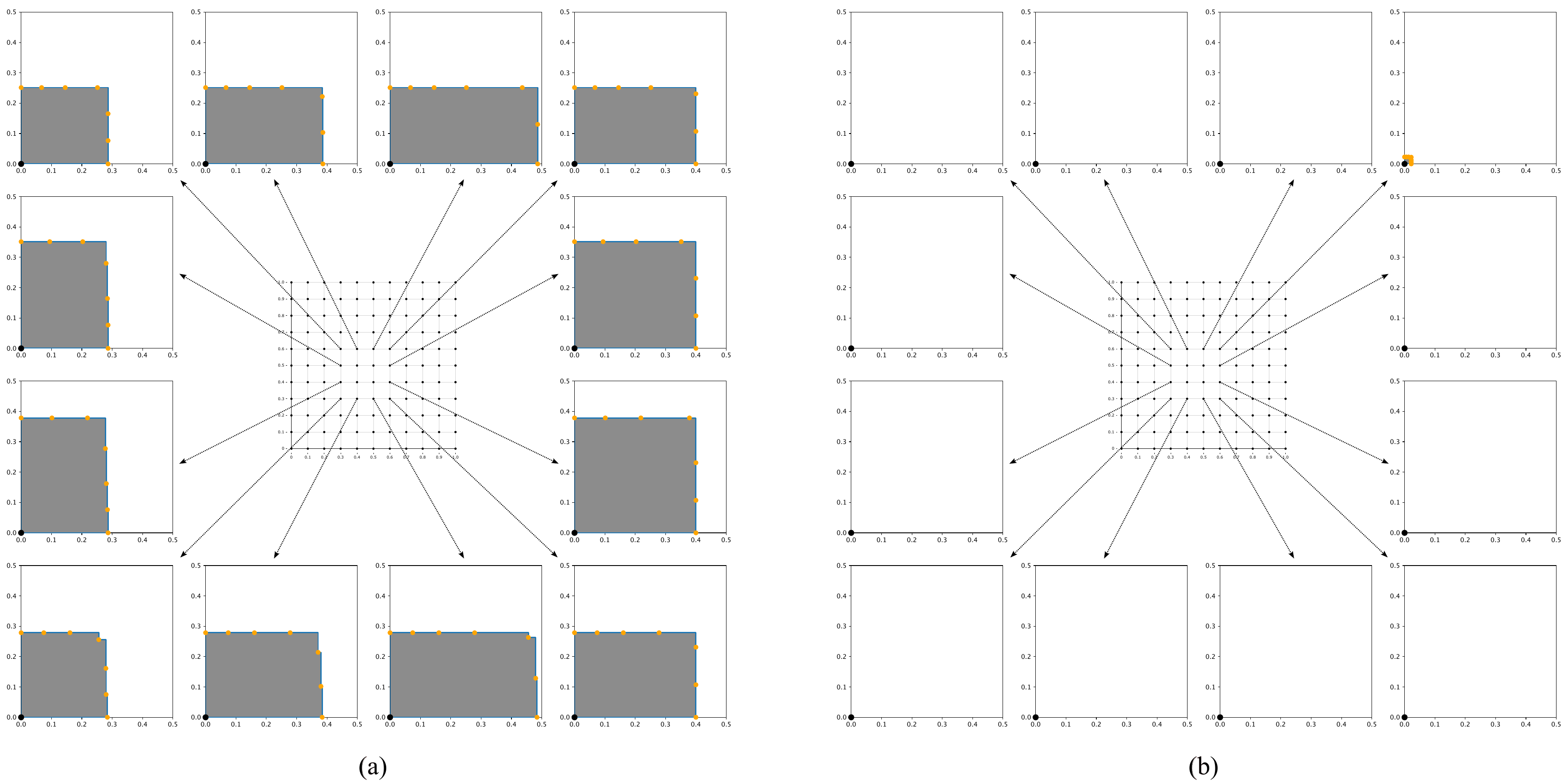}
    \caption{Landscape intervals for two circular point clouds with different sampling patterns. The left and right collections correspond to the point cloud sampled around the entire circle and the point cloud with two sparsely sampled arcs, respectively. The central grids indicate the selected parameter values, with the normalized sparsity threshold on the horizontal axis and the normalized Rips scale on the vertical axis. The surrounding panels show the corresponding landscape intervals for $k=1$ in homological degree $1$, with horizontal and vertical axes representing half-widths in the sparsity threshold and Rips scale, respectively. The gray regions show the computed intervals, the blue curves indicate their boundaries, and the orange points are the seven sampled radial boundary points, including those on the two coordinate axes. Corresponding panels in the two collections use the same coordinate scales. The landscape intervals for the sparsely sampled point cloud are mostly empty or concentrated near the origin, in contrast to the substantially larger intervals obtained from the point cloud sampled around the entire circle.}
    \label{fig:2d-land-incomplete}
\end{figure}

\subsection{Shape classification on a SHREC15 subset}
\label{subsec:shrec15}

We evaluate landscape-based feature representations on a $12$-class subset of the SHREC15 non-rigid shape dataset \cite{lian15nonrigid}. We compare one-parameter persistence landscapes constructed from individual filtration functions with uniform and integrated $2$-parameter landscapes constructed from pairs of filtration functions.

\subsubsection{Dataset and feature construction}

After removing meshes with more than one connected component, $279$ meshes remained. We used a fixed class-wise $80/20$ split consisting of $223$ training meshes and $56$ test meshes. Each mesh was rescaled coordinatewise to $[0,1]^3$ and then simplified while preserving connectedness. Representative meshes from the twelve classes are shown in Figure~\ref{fig:shrec15-samples}.

\begin{figure}[htbp]
    \centering
    \includegraphics[width=0.82\linewidth]
    {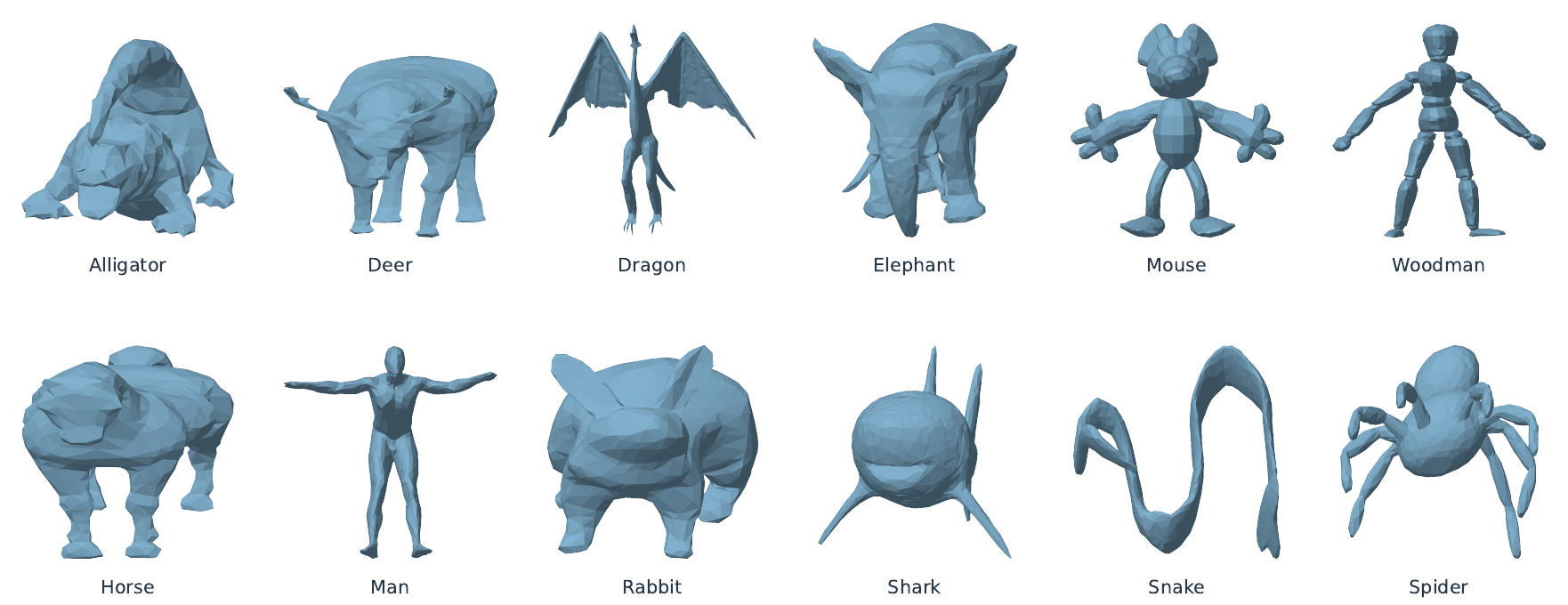}
    \caption{Representative meshes from the $12$ SHREC15 classes used in the shape-classification experiment. One mesh is shown from each class.}
    \label{fig:shrec15-samples}
\end{figure}

On each mesh, we considered four scalar functions on the vertices: the heat kernel signature (HKS) \cite{sun2009concise}, the wave kernel signature (WKS) \cite{aubry2011wave}, a reversed center-distance function, and the $q$NN local-sparsity function. We used the HKS sample nearest diffusion time $0.01$ and the WKS sample nearest energy $5$.

For a vertex $v$, define
\[
    f_{\mathrm{center}}(v)
    \coloneqq
    d_{\max}
    -
    \left\lVert
        v-(0.5,0.5,0.5)
    \right\rVert_2,
\]
where $d_{\max}$ is the largest distance from a vertex of the same mesh to $(0.5,0.5,0.5)$. This reversal represents the superlevel filtration of the center distance as a lower-star filtration. For the $q$NN
function, we use the local-sparsity function defined in the preceding subsection with $q=20$, setting $f_{q\mathrm{NN}}(v) \coloneqq \rho_{20}(v)$. Smaller values of $f_{q\mathrm{NN}}$ correspond to locally denser regions of the mesh.

Each of the four vertex functions was min--max normalized separately on each mesh to $[0,1]$. We then extended each normalized function $f$ to the simplices of the mesh by the lower-star rule
\[
    f(\sigma)
    \coloneqq
    \max_{v \in \sigma} f(v).
\]
For each resulting filtration, we computed persistent homology in degrees $0$ and $1$ and transformed the barcodes into one-parameter persistence landscapes. The landscapes were sampled to form finite-dimensional feature vectors.

For the $2$-parameter representations, we considered the six function pairs HKS--WKS, HKS--center, HKS--$q$NN, WKS--center, WKS--$q$NN, and $q$NN--center. Each pair of normalized lower-star functions defined a bifiltration using the construction described at the beginning of this section. For each bifiltration, we computed persistent homology in degrees $0$ and $1$ and constructed the corresponding uniform and integrated landscape features separately in each degree. Further preprocessing information, parameter grids, directional samples, candidate resolutions, and feature dimensions are specified in Appendix~\ref{app:shrec15-details}.

\subsubsection{Classification protocol}

We classified each feature representation using support vector machines with linear and radial basis function (RBF) kernels. Feature and model parameters were selected jointly by stratified five-fold cross-validation on the training set. The same fold assignments were used for every representation. Within each fold, every feature was standardized using the mean and standard deviation estimated from the training portion and then transformed using those same values in the validation portion.

Configurations were selected first by mean cross-validation accuracy and then by mean macro F1-score. The test set was not used for feature, model, or hyperparameter selection. The selected configuration was refitted on the complete training set and evaluated on the held-out test set. For each selected configuration, we report held-out test accuracy and macro F1-score together with the mean and standard
deviation of cross-validation accuracy and macro F1-score across the five folds. The complete parameter grids and remaining tie-breaking rules are given in Appendix~\ref{app:shrec15-details}, and the selected
feature and model configurations are reported in
Table~\ref{tab:shrec15-selected-configurations}.

\subsubsection{Results and discussion}

Table~\ref{tab:shrec15-landscape-classification} compares the individual one-parameter representations with the uniform and integrated $2$-parameter representations. The integrated HKS--WKS representation in homological degree $1$ had the highest mean cross-validation accuracy, $93.72\% \pm 3.61\%$, with a test accuracy of $92.86\%$. The highest observed test accuracy was $94.64\%$, obtained by the uniform HKS--center representation in degree $1$.

In five of the six degree-$0$ blocks and all six degree-$1$ blocks, at least one $2$-parameter representation achieved a higher mean cross-validation accuracy than either corresponding one-parameter representation. The exception was WKS--center in degree $0$. The two largest differences relative to the stronger one-parameter baseline occurred in degree $0$. For HKS--center, the highest mean accuracy among the two one-parameter representations was $52.93\%$, compared with $73.58\%$ for the uniform representation. For $q$NN--center, the corresponding values were $46.70\%$ and $66.87\%$.

The numerical performance also depended strongly on the filtration functions and homological degree. Among the one-parameter representations, WKS had the highest mean cross-validation accuracy in degree $0$, whereas HKS had the highest in degree $1$. When one constituent function already provided a strong one-parameter representation, the difference obtained by adding a second function was often smaller. For example, HKS achieved $88.77\%$ in degree $1$, while the best HKS--center and HKS--$q$NN representations achieved $89.70\%$ and $90.59\%$, respectively. By contrast, center distance and $q$NN were less effective individually, while their joint $2$-parameter representations had substantially higher mean accuracies in both degrees. These observations suggest that the effect of combining two filtration functions depends on both the information carried by the individual functions and the additional information captured by their joint variation.

Neither the uniform nor the integrated representation exhibited a consistent numerical advantage over the other. Their relative performance varied with the filtration pair and homological degree, and several differences were small relative to the variation across cross-validation folds. Because the same cross-validation results were used to choose each configuration, the reported cross-validation values are not an independent evaluation. The held-out test set provides a separate evaluation, but it contains only $56$ shapes, so one additional misclassification changes the test accuracy by approximately $1.79$ percentage points. We therefore interpret these comparisons descriptively and do not regard small numerical differences as evidence that one representation generally performs better than another.

In these experiments, two-parameter landscapes often improved classification compared with the corresponding one-parameter representations. Several degree-$0$ comparisons showed substantial gains, while improvements over the strong degree-$1$ HKS representation were more limited. Both uniform and integrated landscapes provided useful classification features, and their relative performance depended on the filtration pair and homological degree.

\begin{table*}[p]
\centering
\caption{Classification performance of one- and two-parameter persistence landscape representations on the SHREC15 subset. Results are grouped by homological degree. Within each four-row block, the two individual one-parameter representations are compared with the
integrated and uniform $2$-parameter representations constructed from the corresponding pair of filtration functions.}
\label{tab:shrec15-landscape-classification}
\footnotesize
\setlength{\tabcolsep}{2.5pt}
\renewcommand{\arraystretch}{0.90}
\begin{adjustbox}{max width=\textwidth}
\begin{tabular}{@{}llcccc@{}}
\toprule
Filtration pair
& Representation
& CV accuracy (\%)
& CV macro F1-score (\%)
& Test accuracy (\%)
& Test macro F1-score (\%) \\
\midrule
\multicolumn{6}{c}{\textit{Homological degree $0$}} \\
\addlinespace[1pt]

HKS--center ($H_0$)
& $1$-parameter: HKS
& $52.93 \pm 7.94$
& $51.93 \pm 7.93$
& $62.50$
& $62.16$ \\
& $1$-parameter: center
& $46.70 \pm 7.50$
& $46.31 \pm 6.52$
& $44.64$
& $44.21$ \\
& Integrated
& $73.12 \pm 3.56$
& $72.04 \pm 2.48$
& $69.64$
& $70.45$ \\
& Uniform
& $\mathbf{73.58 \pm 8.12}$
& $\mathbf{72.97 \pm 7.86}$
& $\mathbf{82.14}$
& $\mathbf{81.70}$ \\
\addlinespace[1pt]

HKS--$q$NN ($H_0$)
& $1$-parameter: HKS
& $52.93 \pm 7.94$
& $51.93 \pm 7.93$
& $\mathbf{62.50}$
& $\mathbf{62.16}$ \\
& $1$-parameter: $q$NN
& $45.76 \pm 4.05$
& $44.26 \pm 4.54$
& $44.64$
& $44.87$ \\
& Integrated
& $61.47 \pm 8.68$
& $60.72 \pm 7.93$
& $58.93$
& $59.16$ \\
& Uniform
& $\mathbf{62.80 \pm 3.44}$
& $\mathbf{61.47 \pm 3.01}$
& $\mathbf{62.50}$
& $61.71$ \\
\addlinespace[1pt]

HKS--WKS ($H_0$)
& $1$-parameter: HKS
& $52.93 \pm 7.94$
& $51.93 \pm 7.93$
& $62.50$
& $62.16$ \\
& $1$-parameter: WKS
& $80.27 \pm 5.77$
& $80.14 \pm 5.58$
& $67.86$
& $66.27$ \\
& Integrated
& $\mathbf{86.99 \pm 3.30}$
& $\mathbf{86.77 \pm 3.40}$
& $87.50$
& $87.72$ \\
& Uniform
& $86.12 \pm 2.49$
& $86.30 \pm 2.82$
& $\mathbf{91.07}$
& $\mathbf{90.38}$ \\
\addlinespace[1pt]

$q$NN--center ($H_0$)
& $1$-parameter: $q$NN
& $45.76 \pm 4.05$
& $44.26 \pm 4.54$
& $44.64$
& $44.87$ \\
& $1$-parameter: center
& $46.70 \pm 7.50$
& $46.31 \pm 6.52$
& $44.64$
& $44.21$ \\
& Integrated
& $64.59 \pm 3.43$
& $63.95 \pm 2.99$
& $\mathbf{60.71}$
& $\mathbf{59.98}$ \\
& Uniform
& $\mathbf{66.87 \pm 5.79}$
& $\mathbf{66.34 \pm 6.96}$
& $58.93$
& $58.74$ \\
\addlinespace[1pt]

WKS--center ($H_0$)
& $1$-parameter: WKS
& $\mathbf{80.27 \pm 5.77}$
& $\mathbf{80.14 \pm 5.58}$
& $67.86$
& $66.27$ \\
& $1$-parameter: center
& $46.70 \pm 7.50$
& $46.31 \pm 6.52$
& $44.64$
& $44.21$ \\
& Integrated
& $77.61 \pm 7.08$
& $76.52 \pm 7.81$
& $83.93$
& $84.35$ \\
& Uniform
& $78.48 \pm 5.81$
& $77.53 \pm 6.23$
& $\mathbf{89.29}$
& $\mathbf{88.95}$ \\
\addlinespace[1pt]

WKS--$q$NN ($H_0$)
& $1$-parameter: WKS
& $80.27 \pm 5.77$
& $80.14 \pm 5.58$
& $67.86$
& $66.27$ \\
& $1$-parameter: $q$NN
& $45.76 \pm 4.05$
& $44.26 \pm 4.54$
& $44.64$
& $44.87$ \\
& Integrated
& $\mathbf{84.73 \pm 3.40}$
& $\mathbf{83.86 \pm 3.07}$
& $\mathbf{83.93}$
& $\mathbf{83.54}$ \\
& Uniform
& $81.61 \pm 4.80$
& $80.59 \pm 4.40$
& $71.43$
& $71.71$ \\
\midrule
\multicolumn{6}{c}{\textit{Homological degree $1$}} \\
\addlinespace[1pt]

HKS--center ($H_1$)
& $1$-parameter: HKS
& $88.77 \pm 6.44$
& $88.95 \pm 6.42$
& $92.86$
& $92.59$ \\
& $1$-parameter: center
& $29.62 \pm 4.95$
& $24.73 \pm 5.11$
& $16.07$
& $11.88$ \\
& Integrated
& $\mathbf{89.70 \pm 5.05}$
& $\mathbf{89.69 \pm 4.94}$
& $92.86$
& $92.88$ \\
& Uniform
& $88.80 \pm 5.66$
& $88.98 \pm 5.72$
& $\mathbf{94.64}$
& $\mathbf{94.68}$ \\
\addlinespace[1pt]

HKS--$q$NN ($H_1$)
& $1$-parameter: HKS
& $88.77 \pm 6.44$
& $88.95 \pm 6.42$
& $\mathbf{92.86}$
& $92.59$ \\
& $1$-parameter: $q$NN
& $56.47 \pm 5.68$
& $54.80 \pm 5.80$
& $62.50$
& $61.89$ \\
& Integrated
& $89.70 \pm 4.80$
& $89.81 \pm 4.87$
& $91.07$
& $91.01$ \\
& Uniform
& $\mathbf{90.59 \pm 4.11}$
& $\mathbf{90.73 \pm 4.02}$
& $\mathbf{92.86}$
& $\mathbf{92.93}$ \\
\addlinespace[1pt]

HKS--WKS ($H_1$)
& $1$-parameter: HKS
& $88.77 \pm 6.44$
& $88.95 \pm 6.42$
& $\mathbf{92.86}$
& $92.59$ \\
& $1$-parameter: WKS
& $79.38 \pm 3.19$
& $78.70 \pm 2.83$
& $87.50$
& $87.24$ \\
& Integrated
& $\mathbf{93.72 \pm 3.61}$
& $\mathbf{93.63 \pm 3.71}$
& $\mathbf{92.86}$
& $\mathbf{93.09}$ \\
& Uniform
& $93.27 \pm 3.20$
& $93.15 \pm 3.26$
& $91.07$
& $90.57$ \\
\addlinespace[1pt]

$q$NN--center ($H_1$)
& $1$-parameter: $q$NN
& $56.47 \pm 5.68$
& $54.80 \pm 5.80$
& $62.50$
& $61.89$ \\
& $1$-parameter: center
& $29.62 \pm 4.95$
& $24.73 \pm 5.11$
& $16.07$
& $11.88$ \\
& Integrated
& $\mathbf{70.88 \pm 4.34}$
& $\mathbf{70.74 \pm 5.03}$
& $\mathbf{80.36}$
& $\mathbf{80.30}$ \\
& Uniform
& $68.62 \pm 3.07$
& $67.82 \pm 3.38$
& $73.21$
& $72.80$ \\
\addlinespace[1pt]

WKS--center ($H_1$)
& $1$-parameter: WKS
& $79.38 \pm 3.19$
& $78.70 \pm 2.83$
& $87.50$
& $87.24$ \\
& $1$-parameter: center
& $29.62 \pm 4.95$
& $24.73 \pm 5.11$
& $16.07$
& $11.88$ \\
& Integrated
& $\mathbf{84.32 \pm 4.82}$
& $\mathbf{83.03 \pm 6.30}$
& $\mathbf{91.07}$
& $\mathbf{90.58}$ \\
& Uniform
& $83.01 \pm 5.07$
& $81.76 \pm 6.61$
& $87.50$
& $86.80$ \\
\addlinespace[1pt]

WKS--$q$NN ($H_1$)
& $1$-parameter: WKS
& $79.38 \pm 3.19$
& $78.70 \pm 2.83$
& $87.50$
& $87.24$ \\
& $1$-parameter: $q$NN
& $56.47 \pm 5.68$
& $54.80 \pm 5.80$
& $62.50$
& $61.89$ \\
& Integrated
& $\mathbf{86.58 \pm 3.90}$
& $\mathbf{86.34 \pm 3.94}$
& $85.71$
& $85.76$ \\
& Uniform
& $85.71 \pm 6.15$
& $84.75 \pm 7.21$
& $\mathbf{91.07}$
& $\mathbf{91.03}$ \\
\bottomrule
\end{tabular}
\end{adjustbox}
\par\smallskip
\begin{minipage}{\textwidth}
\footnotesize
\textit{Notes.}
All values are percentages. Cross-validation entries are reported as mean $\pm$ standard deviation over five folds for the selected configurations. Boldface identifies the numerically largest displayed
value in each metric column within a four-row block. Repeated one-parameter rows refer to the same baseline experiment. Selected feature and SVM configurations are reported in Appendix~\ref{app:shrec15-selected-configurations}.
\end{minipage}
\end{table*}

\subsection{Medical image classification on MedMNIST}
\label{subsec:medmnist-classification}

To evaluate the landscape representations in a larger-scale image-classification setting, we use MedMNIST v2, an eighteen-dataset benchmark for two- and three-dimensional biomedical image classification \cite{medmnistv1,medmnistv2}. For each image or volume, we construct a $2$-parameter persistence module using the G--LoG bifiltration and evaluate the resulting uniform and integrated landscape features. For comparison using the same bifiltration, we include published results for the multiparameter persistence-image (MPI) representation introduced in \cite{CarriereBlumberg2020} and applied to the G--LoG bifiltration in \cite{wang2026glog}.

\subsubsection{Dataset and feature construction}

We used all twelve two-dimensional image datasets and all six three-dimensional volume datasets in MedMNIST v2. The collection includes binary, multiclass, and multilabel classification tasks. Each two-dimensional image has spatial resolution $28 \times 28$, and each three-dimensional volume has spatial resolution $28 \times 28 \times 28$. We used the official training,
validation, and test splits provided by MedMNIST. Representative samples are shown in Figure~\ref{fig:medmnist-overview}.

\begin{figure*}[!t]
    \centering
    \includegraphics[width=\textwidth]
    {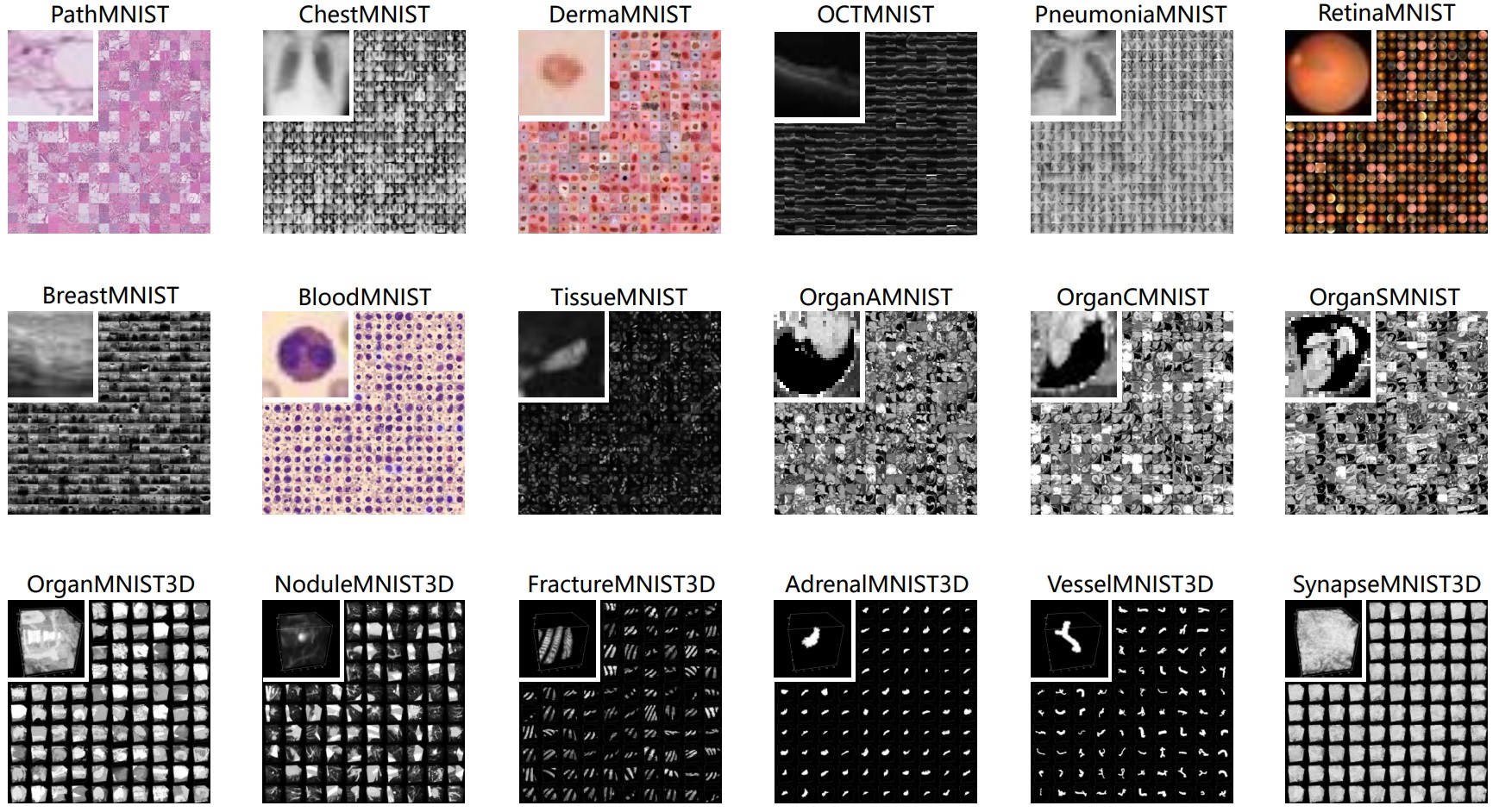}
    \caption{Representative samples from MedMNIST v2. The top two rows show samples from the twelve two-dimensional datasets, and the bottom row shows samples from the six three-dimensional datasets. Reproduced from \cite[Figure~1]{medmnistv2}, without modification other than scaling, under the
    \href{https://creativecommons.org/licenses/by/4.0/}
    {Creative Commons Attribution 4.0 International license}.}
    \label{fig:medmnist-overview}
\end{figure*}

We followed the G--LoG bifiltration for medical-image classification introduced in \cite{wang2026glog}. Let $I$ denote the resulting grayscale image or scalar-valued volume. The two filtration functions were
\[
    f_{\mathrm{G},\sigma}
    \coloneqq
    \mathcal{G}_{\sigma} * I,
    \qquad
    f_{\mathrm{LoG}}
    \coloneqq
    \Delta\bigl(\mathcal{G}_{1} * I\bigr),
\]
where $\mathcal{G}_{\sigma}$ denotes a Gaussian kernel with standard deviation $\sigma$ and $*$ denotes convolution. For the Gaussian-smoothed intensity, we considered $\sigma \in \{0,0.5,1,1.5\}$. When $\sigma=0$, no smoothing was applied. The Gaussian scale in the Laplacian-of-Gaussian coordinate was fixed at $1$, so the values of $\sigma$ reported below refer only to the Gaussian-smoothed intensity coordinate.

Both filtration coordinates were normalized using global extrema computed from the official training split. The resulting transformations were then applied to the validation and test splits. After normalization, each pixel or voxel was regarded as a top-dimensional cubical cell with grade $\bigl(f_{\mathrm{G},\sigma},f_{\mathrm{LoG}}\bigr)$. For $u \in \R^2$, we defined $K_u$ to be the cubical subcomplex generated by the top-dimensional cells whose grades are coordinatewise at most $u$. Applying homology produced a $2$-parameter persistence module.
For the two-dimensional datasets, we computed homology in degrees $0$ and $1$. For the three-dimensional datasets, we computed homology in degrees $0$, $1$, and $2$. Features from all degrees were included
in the feature vector for the corresponding image or volume.

For each resulting module $M_p$, we constructed uniform and integrated landscape features at finitely many parameter values and landscape layers. For each image or volume, the resulting values were concatenated over the parameter values, landscape layers, and homological degrees. The integrated landscape curves were sampled to form finite-dimensional feature vectors. The parameter grid,
direction set, curve-sampling resolution, and resulting feature dimensions are specified in
Appendix~\ref{app:medmnist-details}.

\subsubsection{Classification protocol}

We evaluated each landscape representation using the same multilayer-perceptron architecture and optimization settings. The official validation split was used for early stopping and selection of the training epoch, while the test split was used only for final evaluation.

Each combination of dataset, Gaussian scale, and landscape representation was trained five times using fixed random seeds. We report mean test AUC and accuracy in Table~\ref{tab:medmnist-results}, with the corresponding sample standard deviations in Table~\ref{tab:medmnist-landscape-variability}. The complete network architecture, training procedure, random seeds, and task-specific definitions of AUC and accuracy are given in Appendix~\ref{app:medmnist-details}.

\subsubsection{Results and discussion}

Panels~(a) and~(b) of Table~\ref{tab:medmnist-results} report the results for the twelve two-dimensional datasets, and panel~(c) reports the results for the six three-dimensional datasets. The published comparison values were obtained under the experimental protocols of their original sources, and comparable run-level variability was unavailable. We therefore restrict our discussion to the numerical patterns shown in the table and do not treat these comparisons as statistical evidence that one representation is generally more effective than another.

Our detailed discussion focuses on the G--LoG MPI representation and the two landscape representations because all three are constructed from the same G--LoG bifiltration. The convolutional-network and automated-machine-learning benchmarks use the original images or volumes, whereas Topo-Med uses a different topological representation. These benchmarks serve as broader reference points and are not discussed in detail. We report all four Gaussian scales rather than selecting one from the test results, and our discussion emphasizes patterns that persist across them.

In these experiments, the landscape representations often had higher mean AUC and accuracy than the published G--LoG MPI values on the two-dimensional datasets. This pattern held across all four Gaussian scales on several datasets, including OCTMNIST and OrganCMNIST. The improvements were not universal, however: PathMNIST consistently favored G--LoG MPI, while some other comparisons showed smaller differences or depended on the Gaussian scale.

In contrast, the published G--LoG MPI accuracy was generally higher on the three-dimensional datasets. This pattern held across all four Gaussian scales on OrganMNIST3D and FractureMNIST3D. Both landscape representations had higher mean AUC on VesselMNIST3D at every Gaussian scale, but the corresponding accuracy comparisons were mixed. Thus, the improvements observed on many two-dimensional datasets did not extend consistently to the three-dimensional datasets. Compared with most two-dimensional datasets, the three-dimensional datasets contain fewer training examples \cite{medmnistv2}, which may make learning from the landscape features more difficult. The present experiments do not separate this possible effect from differences in the data and classification tasks.

The uniform and integrated landscapes generally followed similar dataset-level patterns, but which one had the larger mean depended on the dataset and Gaussian scale and could differ between AUC and accuracy. Under the sampling scheme used here, the integrated feature vector was $50$ times longer than the uniform feature vector. Although both representations were evaluated using the same hidden-layer architecture and training settings, the larger integrated input produced more trainable parameters in the first network layer. The comparison therefore holds the architecture and training procedure fixed, but not the total number of trainable parameters.

\input{sections/medmnist_main_table}

%% file: sections/medmnist_main_table.tex
\begin{table*}[p]
\centering
\captionsetup{skip=4pt,font=footnotesize}
\caption{MedMNIST v2 classification results. Among the three G--LoG representations, boldface marks column maxima at each fixed $\sigma$. Underlining marks maxima across all four $\sigma$ values. Ties are included.}
\label{tab:medmnist-results}
\begingroup
\fontsize{8}{9}\selectfont
\setlength{\tabcolsep}{2.5pt}
\renewcommand{\arraystretch}{0.80}
\setlength{\arrayrulewidth}{0.3pt}
{\fontsize{9}{10}\selectfont\textbf{(a) First six two-dimensional datasets}\par}
\vspace{1pt}
\begin{adjustbox}{max width=\textwidth}
\begin{tabular}{@{}ll|*{5}{rr|}rr@{}}
\hline
\multicolumn{2}{l|}{\rule{0pt}{7.4pt}Method} & \multicolumn{2}{c|}{PathMNIST} & \multicolumn{2}{c|}{ChestMNIST} & \multicolumn{2}{c|}{DermaMNIST} & \multicolumn{2}{c|}{OCTMNIST} & \multicolumn{2}{c|}{PneumoniaMNIST} & \multicolumn{2}{c}{RetinaMNIST} \\
\cline{3-14}
\rule{0pt}{6.8pt} & $\sigma$ & AUC & ACC & AUC & ACC & AUC & ACC & AUC & ACC & AUC & ACC & AUC & ACC \\
\hline
\multicolumn{2}{l|}{\rule{0pt}{6.8pt}ResNet-18 (28)} & $0.983$ & $0.907$ & $0.768$ & $0.947$ & $0.917$ & $0.735$ & $0.943$ & $0.743$ & $0.944$ & $0.854$ & $0.717$ & $0.524$ \\
\multicolumn{2}{l|}{ResNet-18 (224)} & $0.989$ & $0.909$ & $0.773$ & $0.947$ & $0.920$ & $0.754$ & $0.958$ & $0.763$ & $0.956$ & $0.864$ & $0.710$ & $0.493$ \\
\multicolumn{2}{l|}{ResNet-50 (28)} & $0.990$ & $0.911$ & $0.769$ & $0.947$ & $0.913$ & $0.735$ & $0.952$ & $0.762$ & $0.948$ & $0.854$ & $0.726$ & $0.528$ \\
\multicolumn{2}{l|}{ResNet-50 (224)} & $0.989$ & $0.892$ & $0.773$ & $0.948$ & $0.912$ & $0.731$ & $0.958$ & $0.776$ & $0.962$ & $0.884$ & $0.716$ & $0.511$ \\
\multicolumn{2}{l|}{auto-sklearn} & $0.934$ & $0.716$ & $0.649$ & $0.779$ & $0.902$ & $0.719$ & $0.887$ & $0.601$ & $0.942$ & $0.855$ & $0.690$ & $0.515$ \\
\multicolumn{2}{l|}{AutoKeras} & $0.959$ & $0.834$ & $0.742$ & $0.937$ & $0.915$ & $0.749$ & $0.955$ & $0.763$ & $0.947$ & $0.878$ & $0.719$ & $0.503$ \\
\multicolumn{2}{l|}{Google AutoML Vision} & $0.944$ & $0.728$ & $0.778$ & $0.948$ & $0.914$ & $0.768$ & $0.963$ & $0.771$ & $0.991$ & $0.946$ & $0.750$ & $0.531$ \\
\multicolumn{2}{l|}{Topo-Med (MLP)} & $0.942$ & $0.683$ & $0.787$ & $0.530$ & $0.904$ & $0.669$ & $0.710$ & $0.450$ & $0.845$ & $0.762$ & $0.728$ & $0.458$ \\
\hline
\rule{0pt}{6.8pt}\raisebox{-1.2\normalbaselineskip}[0pt][0pt]{\shortstack[l]{G--LoG MPI\\(MLP)}} & $0$ & $\underline{\mathbf{0.955}}$ & $\underline{\mathbf{0.753}}$ & $0.606$ & $\underline{\mathbf{0.947}}$ & $0.809$ & $\mathbf{0.703}$ & $0.859$ & $0.522$ & $\mathbf{0.907}$ & $\mathbf{0.817}$ & $0.611$ & $\underline{\mathbf{0.518}}$ \\
 & $0.5$ & $\mathbf{0.954}$ & $\underline{\mathbf{0.753}}$ & $0.608$ & $\underline{\mathbf{0.947}}$ & $0.806$ & $\underline{\mathbf{0.708}}$ & $0.872$ & $0.567$ & $\underline{\mathbf{0.910}}$ & $\underline{\mathbf{0.825}}$ & $0.624$ & $\mathbf{0.500}$ \\
 & $1$ & $\mathbf{0.940}$ & $\mathbf{0.708}$ & $0.610$ & $\underline{\mathbf{0.947}}$ & $0.815$ & $0.700$ & $0.869$ & $0.548$ & $\mathbf{0.906}$ & $\underline{\mathbf{0.825}}$ & $\mathbf{0.643}$ & $\mathbf{0.498}$ \\
 & $1.5$ & $\mathbf{0.939}$ & $\mathbf{0.702}$ & $0.603$ & $\underline{\mathbf{0.947}}$ & $0.816$ & $0.698$ & $0.854$ & $0.529$ & $0.891$ & $0.795$ & $\mathbf{0.638}$ & $0.483$ \\
\hline
\rule{0pt}{6.8pt}\raisebox{-1.2\normalbaselineskip}[0pt][0pt]{\shortstack[l]{Uniform landscape\\(MLP)}} & $0$ & $0.947$ & $0.731$ & $\mathbf{0.665}$ & $\underline{\mathbf{0.947}}$ & $\underline{\mathbf{0.844}}$ & $0.701$ & $0.899$ & $0.626$ & $0.892$ & $0.777$ & $0.643$ & $0.510$ \\
 & $0.5$ & $0.941$ & $0.706$ & $\underline{\mathbf{0.666}}$ & $\underline{\mathbf{0.947}}$ & $\underline{\mathbf{0.844}}$ & $0.702$ & $0.908$ & $\underline{\mathbf{0.635}}$ & $0.895$ & $0.795$ & $0.631$ & $0.495$ \\
 & $1$ & $0.930$ & $0.682$ & $\mathbf{0.665}$ & $\underline{\mathbf{0.947}}$ & $\mathbf{0.840}$ & $\mathbf{0.703}$ & $0.908$ & $\mathbf{0.623}$ & $0.895$ & $0.813$ & $0.624$ & $0.495$ \\
 & $1.5$ & $0.925$ & $0.675$ & $\mathbf{0.661}$ & $\underline{\mathbf{0.947}}$ & $\mathbf{0.840}$ & $\mathbf{0.703}$ & $0.902$ & $0.606$ & $0.889$ & $0.797$ & $0.625$ & $\mathbf{0.488}$ \\
\hline
\rule{0pt}{6.8pt}\raisebox{-1.2\normalbaselineskip}[0pt][0pt]{\shortstack[l]{Integrated landscape\\(MLP)}} & $0$ & $0.947$ & $0.732$ & $0.655$ & $\underline{\mathbf{0.947}}$ & $0.841$ & $0.700$ & $\mathbf{0.902}$ & $\mathbf{0.628}$ & $0.899$ & $0.802$ & $\underline{\mathbf{0.645}}$ & $0.497$ \\
 & $0.5$ & $0.937$ & $0.710$ & $0.658$ & $\underline{\mathbf{0.947}}$ & $0.839$ & $0.698$ & $\mathbf{0.910}$ & $0.625$ & $0.902$ & $0.804$ & $\mathbf{0.634}$ & $0.495$ \\
 & $1$ & $0.927$ & $0.685$ & $0.655$ & $\underline{\mathbf{0.947}}$ & $0.835$ & $0.700$ & $\underline{\mathbf{0.912}}$ & $0.608$ & $\mathbf{0.906}$ & $0.805$ & $0.616$ & $0.489$ \\
 & $1.5$ & $0.925$ & $0.672$ & $0.652$ & $\underline{\mathbf{0.947}}$ & $0.838$ & $0.697$ & $\mathbf{0.905}$ & $\mathbf{0.620}$ & $\mathbf{0.898}$ & $\mathbf{0.814}$ & $0.621$ & $0.476$ \\
\hline
\end{tabular}
\end{adjustbox}
\par\vspace{2pt}
{\fontsize{9}{10}\selectfont\textbf{(b) Remaining six two-dimensional datasets}\par}
\vspace{1pt}
\begin{adjustbox}{max width=\textwidth}
\begin{tabular}{@{}ll|*{5}{rr|}rr@{}}
\hline
\multicolumn{2}{l|}{\rule{0pt}{7.4pt}Method} & \multicolumn{2}{c|}{BreastMNIST} & \multicolumn{2}{c|}{BloodMNIST} & \multicolumn{2}{c|}{TissueMNIST} & \multicolumn{2}{c|}{OrganAMNIST} & \multicolumn{2}{c|}{OrganCMNIST} & \multicolumn{2}{c}{OrganSMNIST} \\
\cline{3-14}
\rule{0pt}{6.8pt} & $\sigma$ & AUC & ACC & AUC & ACC & AUC & ACC & AUC & ACC & AUC & ACC & AUC & ACC \\
\hline
\multicolumn{2}{l|}{\rule{0pt}{6.8pt}ResNet-18 (28)} & $0.901$ & $0.863$ & $0.998$ & $0.958$ & $0.930$ & $0.676$ & $0.997$ & $0.935$ & $0.992$ & $0.900$ & $0.972$ & $0.782$ \\
\multicolumn{2}{l|}{ResNet-18 (224)} & $0.891$ & $0.833$ & $0.998$ & $0.963$ & $0.933$ & $0.681$ & $0.998$ & $0.951$ & $0.994$ & $0.920$ & $0.974$ & $0.778$ \\
\multicolumn{2}{l|}{ResNet-50 (28)} & $0.857$ & $0.812$ & $0.997$ & $0.956$ & $0.931$ & $0.680$ & $0.997$ & $0.935$ & $0.992$ & $0.905$ & $0.972$ & $0.770$ \\
\multicolumn{2}{l|}{ResNet-50 (224)} & $0.866$ & $0.842$ & $0.997$ & $0.950$ & $0.932$ & $0.680$ & $0.998$ & $0.947$ & $0.993$ & $0.911$ & $0.975$ & $0.785$ \\
\multicolumn{2}{l|}{auto-sklearn} & $0.836$ & $0.803$ & $0.984$ & $0.878$ & $0.828$ & $0.532$ & $0.963$ & $0.762$ & $0.976$ & $0.829$ & $0.945$ & $0.672$ \\
\multicolumn{2}{l|}{AutoKeras} & $0.871$ & $0.831$ & $0.998$ & $0.961$ & $0.941$ & $0.703$ & $0.994$ & $0.905$ & $0.990$ & $0.879$ & $0.974$ & $0.813$ \\
\multicolumn{2}{l|}{Google AutoML Vision} & $0.919$ & $0.861$ & $0.998$ & $0.966$ & $0.924$ & $0.673$ & $0.990$ & $0.886$ & $0.988$ & $0.877$ & $0.964$ & $0.749$ \\
\multicolumn{2}{l|}{Topo-Med (MLP)} & $0.821$ & $0.737$ & $0.973$ & $0.798$ & $0.837$ & $0.450$ & $0.921$ & $0.523$ & $0.894$ & $0.489$ & $0.910$ & $0.532$ \\
\hline
\rule{0pt}{6.8pt}\raisebox{-1.2\normalbaselineskip}[0pt][0pt]{\shortstack[l]{G--LoG MPI\\(MLP)}} & $0$ & $\mathbf{0.820}$ & $0.731$ & $0.977$ & $0.847$ & $0.849$ & $0.537$ & $0.934$ & $0.591$ & $0.935$ & $0.593$ & $0.927$ & $0.605$ \\
 & $0.5$ & $0.784$ & $\underline{\mathbf{0.814}}$ & $0.979$ & $0.844$ & $0.845$ & $0.533$ & $0.936$ & $0.584$ & $0.933$ & $0.601$ & $0.929$ & $0.588$ \\
 & $1$ & $0.752$ & $0.750$ & $0.977$ & $0.840$ & $0.840$ & $0.529$ & $0.931$ & $0.562$ & $0.931$ & $0.588$ & $0.918$ & $0.573$ \\
 & $1.5$ & $0.730$ & $0.763$ & $0.978$ & $0.846$ & $0.837$ & $0.528$ & $0.927$ & $0.559$ & $0.922$ & $0.565$ & $0.910$ & $0.565$ \\
\hline
\rule{0pt}{6.8pt}\raisebox{-1.2\normalbaselineskip}[0pt][0pt]{\shortstack[l]{Uniform landscape\\(MLP)}} & $0$ & $0.800$ & $\mathbf{0.796}$ & $\underline{\mathbf{0.990}}$ & $\underline{\mathbf{0.893}}$ & $\underline{\mathbf{0.871}}$ & $\underline{\mathbf{0.565}}$ & $\mathbf{0.955}$ & $\mathbf{0.648}$ & $\mathbf{0.952}$ & $\mathbf{0.645}$ & $\underline{\mathbf{0.945}}$ & $\mathbf{0.640}$ \\
 & $0.5$ & $\mathbf{0.832}$ & $0.783$ & $\mathbf{0.989}$ & $\mathbf{0.886}$ & $\mathbf{0.868}$ & $\mathbf{0.561}$ & $\underline{\mathbf{0.957}}$ & $\underline{\mathbf{0.649}}$ & $\underline{\mathbf{0.954}}$ & $\underline{\mathbf{0.668}}$ & $\underline{\mathbf{0.945}}$ & $\underline{\mathbf{0.646}}$ \\
 & $1$ & $\underline{\mathbf{0.851}}$ & $\mathbf{0.772}$ & $\mathbf{0.987}$ & $\mathbf{0.878}$ & $\mathbf{0.860}$ & $\mathbf{0.552}$ & $\mathbf{0.953}$ & $\mathbf{0.639}$ & $\mathbf{0.949}$ & $\mathbf{0.645}$ & $\mathbf{0.937}$ & $\mathbf{0.619}$ \\
 & $1.5$ & $\mathbf{0.837}$ & $\mathbf{0.782}$ & $\mathbf{0.988}$ & $\mathbf{0.879}$ & $\mathbf{0.857}$ & $\mathbf{0.546}$ & $\mathbf{0.950}$ & $\mathbf{0.630}$ & $\mathbf{0.947}$ & $\mathbf{0.632}$ & $\mathbf{0.934}$ & $\mathbf{0.611}$ \\
\hline
\rule{0pt}{6.8pt}\raisebox{-1.2\normalbaselineskip}[0pt][0pt]{\shortstack[l]{Integrated landscape\\(MLP)}} & $0$ & $0.801$ & $0.763$ & $0.989$ & $0.884$ & $0.867$ & $0.557$ & $0.952$ & $0.637$ & $0.951$ & $0.643$ & $0.940$ & $0.634$ \\
 & $0.5$ & $0.821$ & $0.768$ & $0.987$ & $0.880$ & $0.864$ & $0.555$ & $0.950$ & $0.637$ & $0.951$ & $0.650$ & $0.940$ & $0.625$ \\
 & $1$ & $0.839$ & $0.762$ & $0.986$ & $0.872$ & $0.856$ & $0.544$ & $0.948$ & $0.627$ & $0.944$ & $0.628$ & $0.932$ & $0.596$ \\
 & $1.5$ & $0.828$ & $0.754$ & $0.985$ & $0.867$ & $0.853$ & $0.540$ & $0.946$ & $0.612$ & $0.944$ & $\mathbf{0.632}$ & $0.929$ & $0.601$ \\
\hline
\end{tabular}
\end{adjustbox}
\par\vspace{2pt}
{\fontsize{9}{10}\selectfont\textbf{(c) Six three-dimensional datasets}\par}
\vspace{1pt}
\begin{adjustbox}{max width=\textwidth}
\begin{tabular}{@{}ll|*{5}{rr|}rr@{}}
\hline
\multicolumn{2}{l|}{\rule{0pt}{7.4pt}Method} & \multicolumn{2}{c|}{\shortstack{OrganMNIST\\3D}} & \multicolumn{2}{c|}{\shortstack{NoduleMNIST\\3D}} & \multicolumn{2}{c|}{\shortstack{FractureMNIST\\3D}} & \multicolumn{2}{c|}{\shortstack{AdrenalMNIST\\3D}} & \multicolumn{2}{c|}{\shortstack{VesselMNIST\\3D}} & \multicolumn{2}{c}{\shortstack{SynapseMNIST\\3D}} \\
\cline{3-14}
\rule{0pt}{6.8pt} & $\sigma$ & AUC & ACC & AUC & ACC & AUC & ACC & AUC & ACC & AUC & ACC & AUC & ACC \\
\hline
\multicolumn{2}{l|}{\rule{0pt}{6.8pt}ResNet-18+2.5D} & $0.977$ & $0.788$ & $0.838$ & $0.835$ & $0.587$ & $0.451$ & $0.718$ & $0.772$ & $0.748$ & $0.846$ & $0.634$ & $0.696$ \\
\multicolumn{2}{l|}{ResNet-18+3D} & $0.996$ & $0.907$ & $0.863$ & $0.844$ & $0.712$ & $0.508$ & $0.827$ & $0.721$ & $0.874$ & $0.877$ & $0.820$ & $0.745$ \\
\multicolumn{2}{l|}{ResNet-18+ACS} & $0.994$ & $0.900$ & $0.873$ & $0.847$ & $0.714$ & $0.497$ & $0.839$ & $0.754$ & $0.930$ & $0.928$ & $0.705$ & $0.722$ \\
\multicolumn{2}{l|}{ResNet-50+2.5D} & $0.974$ & $0.769$ & $0.835$ & $0.848$ & $0.552$ & $0.397$ & $0.732$ & $0.763$ & $0.751$ & $0.877$ & $0.669$ & $0.735$ \\
\multicolumn{2}{l|}{ResNet-50+3D} & $0.994$ & $0.883$ & $0.875$ & $0.847$ & $0.725$ & $0.494$ & $0.828$ & $0.745$ & $0.907$ & $0.918$ & $0.851$ & $0.795$ \\
\multicolumn{2}{l|}{ResNet-50+ACS} & $0.994$ & $0.889$ & $0.886$ & $0.841$ & $0.750$ & $0.517$ & $0.828$ & $0.758$ & $0.912$ & $0.858$ & $0.719$ & $0.709$ \\
\multicolumn{2}{l|}{auto-sklearn} & $0.977$ & $0.814$ & $0.914$ & $0.874$ & $0.628$ & $0.453$ & $0.828$ & $0.802$ & $0.910$ & $0.915$ & $0.631$ & $0.730$ \\
\multicolumn{2}{l|}{AutoKeras} & $0.979$ & $0.804$ & $0.844$ & $0.834$ & $0.642$ & $0.458$ & $0.804$ & $0.705$ & $0.773$ & $0.894$ & $0.538$ & $0.724$ \\
\multicolumn{2}{l|}{Topo-Med (MLP)} & $0.837$ & $0.554$ & $0.808$ & $0.736$ & $0.653$ & $0.480$ & $0.864$ & $0.769$ & $0.920$ & $0.887$ & $0.730$ & $0.730$ \\
\hline
\rule{0pt}{6.8pt}\raisebox{-1.2\normalbaselineskip}[0pt][0pt]{\shortstack[l]{G--LoG MPI\\(MLP)}} & $0$ & $\mathbf{0.958}$ & $\mathbf{0.684}$ & $0.564$ & $0.774$ & $\mathbf{0.697}$ & $\mathbf{0.554}$ & $0.742$ & $0.789$ & $0.813$ & $0.887$ & $0.783$ & $0.796$ \\
 & $0.5$ & $\underline{\mathbf{0.961}}$ & $\underline{\mathbf{0.702}}$ & $0.789$ & $\underline{\mathbf{0.852}}$ & $\mathbf{0.704}$ & $\mathbf{0.579}$ & $\underline{\mathbf{0.870}}$ & $\underline{\mathbf{0.847}}$ & $0.877$ & $0.898$ & $\underline{\mathbf{0.810}}$ & $\underline{\mathbf{0.827}}$ \\
 & $1$ & $0.953$ & $\mathbf{0.656}$ & $\underline{\mathbf{0.838}}$ & $\underline{\mathbf{0.852}}$ & $\mathbf{0.753}$ & $\mathbf{0.579}$ & $\mathbf{0.867}$ & $\mathbf{0.829}$ & $0.915$ & $0.908$ & $\mathbf{0.805}$ & $\mathbf{0.813}$ \\
 & $1.5$ & $0.941$ & $\mathbf{0.634}$ & $\mathbf{0.803}$ & $\mathbf{0.848}$ & $\underline{\mathbf{0.774}}$ & $\underline{\mathbf{0.588}}$ & $0.850$ & $0.829$ & $0.933$ & $\underline{\mathbf{0.937}}$ & $0.752$ & $\mathbf{0.776}$ \\
\hline
\rule{0pt}{6.8pt}\raisebox{-1.2\normalbaselineskip}[0pt][0pt]{\shortstack[l]{Uniform landscape\\(MLP)}} & $0$ & $0.956$ & $0.651$ & $\mathbf{0.815}$ & $\mathbf{0.806}$ & $0.683$ & $0.506$ & $\mathbf{0.833}$ & $\mathbf{0.805}$ & $0.892$ & $\mathbf{0.905}$ & $\mathbf{0.800}$ & $\mathbf{0.799}$ \\
 & $0.5$ & $0.960$ & $0.656$ & $\mathbf{0.802}$ & $0.831$ & $0.679$ & $0.528$ & $0.828$ & $0.798$ & $\mathbf{0.913}$ & $\mathbf{0.903}$ & $0.751$ & $0.756$ \\
 & $1$ & $\mathbf{0.957}$ & $0.636$ & $0.808$ & $0.838$ & $0.714$ & $0.551$ & $0.866$ & $0.824$ & $\mathbf{0.924}$ & $0.911$ & $0.780$ & $0.769$ \\
 & $1.5$ & $\mathbf{0.954}$ & $0.621$ & $0.790$ & $0.822$ & $0.729$ & $0.574$ & $\mathbf{0.864}$ & $0.830$ & $0.941$ & $0.914$ & $\mathbf{0.754}$ & $0.755$ \\
\hline
\rule{0pt}{6.8pt}\raisebox{-1.2\normalbaselineskip}[0pt][0pt]{\shortstack[l]{Integrated landscape\\(MLP)}} & $0$ & $0.956$ & $0.652$ & $0.783$ & $0.774$ & $0.670$ & $0.481$ & $0.818$ & $0.791$ & $\mathbf{0.895}$ & $0.883$ & $0.774$ & $0.781$ \\
 & $0.5$ & $0.960$ & $0.657$ & $0.801$ & $0.830$ & $0.684$ & $0.478$ & $0.839$ & $0.795$ & $0.906$ & $0.897$ & $0.741$ & $0.772$ \\
 & $1$ & $0.956$ & $0.649$ & $0.828$ & $0.832$ & $0.712$ & $0.541$ & $0.864$ & $0.815$ & $0.922$ & $\mathbf{0.918}$ & $0.782$ & $0.750$ \\
 & $1.5$ & $0.953$ & $0.620$ & $0.789$ & $0.824$ & $0.697$ & $0.508$ & $0.863$ & $\mathbf{0.840}$ & $\underline{\mathbf{0.943}}$ & $0.925$ & $0.751$ & $0.724$ \\
\hline
\end{tabular}
\end{adjustbox}
\par\vspace{2pt}
\begin{minipage}{\textwidth}
\fontsize{8}{9}\selectfont
\textit{Notes.} AUC and ACC denote test-set area under the ROC curve and accuracy. Landscape values are means over five runs, with sample standard deviations reported in Table~\ref{tab:medmnist-landscape-variability}. Published comparison values are reproduced as reported in \cite{wang2026glog} (G--LoG MPI), \cite{NuwagiraEtAl2025TopoMed} (Topo-Med), and \cite{medmnistv2} (other benchmarks). Comparable run-level variability was unavailable for these published results.
\end{minipage}
\endgroup
\end{table*}

%% file: sections/appendix.tex
\section{Categorical Background for Generalized Rank}
\label{app:categorical-background}

The generalized rank over a connected indexing subposet is defined as the rank of the canonical map from a limit to a colimit. We recall the categorical constructions needed for this map and prove the two
properties used in the main text.

\begin{definition}[Cone and Cocone]
Let $I$ be a small category and let $F\colon I\to\mathcal{C}$ be a functor. A \emph{cone} over $F$ consists of an object $L$ of $\mathcal{C}$ together with a morphism $\pi_x\colon L\to F(x)$ for each object $x$ of $I$. These morphisms satisfy $\pi_y=F(g)\circ\pi_x$ for every morphism $g\colon x\to y$ in $I$. We denote this cone by $\bigl(L,(\pi_x)_{x\in\operatorname{Ob}(I)}\bigr)$.

A \emph{cocone} over $F$ consists of an object $C$ of $\mathcal{C}$ together with a morphism $\iota_x\colon F(x)\to C$ for each object $x$ of $I$. These morphisms satisfy $\iota_x=\iota_y\circ F(g)$ for every morphism $g\colon x\to y$ in $I$. We denote this cocone by $\bigl(C,(\iota_x)_{x\in\operatorname{Ob}(I)}\bigr)$.

These conditions are expressed by the following commutative diagrams.
\[
\begin{tikzcd}[
    column sep=large,
    row sep=large,
    baseline=(current bounding box.center)
]
    & L \arrow[dl, "\pi_x"'] \arrow[dr, "\pi_y"] & \\
    F(x) \arrow[rr, "F(g)"'] && F(y)
\end{tikzcd}
\qquad\qquad
\begin{tikzcd}[
    column sep=large,
    row sep=large,
    baseline=(current bounding box.center)
]
    F(x)
    \arrow[rr, "F(g)"]
    \arrow[dr, "\iota_x"']
    &&
    F(y)
    \arrow[dl, "\iota_y"]
    \\
    & C &
\end{tikzcd}
\]
The objects $L$ and $C$ are called the \emph{apices} of the cone and cocone, respectively.
\end{definition}

\begin{definition}[Limit and Colimit]
Let $I$ be a small category and let $F\colon I\to\mathcal{C}$ be a functor. A \emph{limit} of $F$ is a cone $\bigl(L,(\pi_x)_{x\in\operatorname{Ob}(I)}\bigr)$ that is terminal among all cones over $F$. Explicitly, for every cone $\bigl(L',(\pi_x')_{x\in\operatorname{Ob}(I)}\bigr)$ over $F$, there exists a unique morphism $u\colon L'\to L$ such that $\pi_x'=\pi_x\circ u$ for every $x\in\operatorname{Ob}(I)$. We denote the limit cone by $\bigl(\varprojlim F,(\pi_x)_{x\in\operatorname{Ob}(I)}\bigr)$ and its apex simply by $\varprojlim F$.

A \emph{colimit} of $F$ is a cocone $\bigl(C,(\iota_x)_{x\in\operatorname{Ob}(I)}\bigr)$ that is initial among all cocones over $F$. Explicitly, for every cocone $\bigl(C',(\iota_x')_{x\in\operatorname{Ob}(I)}\bigr)$ over $F$, there exists a unique morphism $u\colon C\to C'$ such that $\iota_x'=u\circ\iota_x$ for every $x\in\operatorname{Ob}(I)$. We denote the colimit cocone by $\bigl(\varinjlim F,(\iota_x)_{x\in\operatorname{Ob}(I)}\bigr)$ and its apex simply by $\varinjlim F$.

These universal properties are illustrated by the following commutative triangles, one for each $x\in\operatorname{Ob}(I)$.
\[
\begin{tikzcd}[
    column sep=large,
    row sep=large,
    baseline=(current bounding box.center)
]
    L'
    \arrow[r, "u", dashed]
    \arrow[dr, "\pi_x'"']
    &
    L
    \arrow[d, "\pi_x"]
    \\
    &
    F(x)
\end{tikzcd}
\qquad\qquad
\begin{tikzcd}[
    column sep=large,
    row sep=large,
    baseline=(current bounding box.center)
]
    F(x)
    \arrow[r, "\iota_x"]
    \arrow[dr, "\iota_x'"']
    &
    C
    \arrow[d, "u", dashed]
    \\
    &
    C'
\end{tikzcd}
\]
In each diagram, the dashed arrow $u$ is unique.
\end{definition}

Suppose that $I$ is nonempty and connected and that $F$ has a limit and a colimit. For each $x\in\operatorname{Ob}(I)$, consider the composite
\[
    \varprojlim F
    \xrightarrow{\pi_x}
    F(x)
    \xrightarrow{\iota_x}
    \varinjlim F.
\]
If $g\colon x\to y$ is a morphism in $I$, the cone and cocone identities give
\[
    \iota_x\circ\pi_x
    =
    \iota_y\circ F(g)\circ\pi_x
    =
    \iota_y\circ\pi_y.
\]
Thus, these composites agree along every morphism in $I$. Since $I$ is connected, any two objects are joined by a zigzag of morphisms, so the composite is independent of $x$. This common composite is the canonical limit-to-colimit morphism associated with $F$ \cite{KimMemoli2021JACT}. We denote it by
\[
    \psi_F
    \colon
    \varprojlim F
    \longrightarrow
    \varinjlim F.
\]

When an interval has a minimum and a maximum, its generalized rank reduces to the rank of the structure map between these two endpoints.

\begin{lemma}
\label{lem:generalized-rank-extrema}
Let $P$ be a poset, let $M \in \Vect^P$, and let $I \subseteq P$ be an interval with minimum $a$ and maximum $b$. Then
\[
    \rk_M^{\mathrm{gen}}(I)
    =
    \rk_M(a,b).
\]
\end{lemma}

\begin{proof}
Since $a$ is the minimum of $I$, the object $M(a)$ with structure maps $M_a^x\colon M(a)\to M(x)$ forms a cone over $M|_I$. If $\bigl(L',(\pi_x')_{x\in I}\bigr)$ is any other cone, cone compatibility gives $\pi_x'=M_a^x\circ\pi_a'$ for every $x\in I$. Thus, $\pi_a'\colon L'\to M(a)$ is the required factorization. It is unique because $M_a^a$ is the identity map. Therefore, $\varprojlim_{x\in I}M(x)\cong M(a)$.

Similarly, since $b$ is the maximum of $I$, the object $M(b)$ with structure maps $M_x^b\colon M(x)\to M(b)$ is a colimit of $M|_I$. Hence, $\varinjlim_{x\in I}M(x)\cong M(b)$. Under these identifications, the canonical map from the limit to the colimit is $M_a^b$. Therefore,
\[
    \rk_M^{\mathrm{gen}}(I)
    =
    \operatorname{rank}(M_a^b)
    =
    \rk_M(a,b).
\]
\end{proof}

The next property shows that enlarging a connected indexing subposet cannot increase the generalized rank. This order-reversing property was proved for locally finite posets in \cite[Proposition~3.8]{KimMemoli2021JACT}. The argument below uses only the universal properties of limits and colimits and therefore also applies to the connected subposets of $\R^n$ considered here.

\begin{proposition}[Monotonicity of generalized rank]
\label{prop:generalized-rank-monotonicity}
Let $M \in \Vect^{\R^n}$, and let $P \subseteq P'$ be nonempty connected subposets of $\R^n$. Then
\[
    \rk_M^{\mathrm{gen}}(P')
    \le
    \rk_M^{\mathrm{gen}}(P).
\]
\end{proposition}

\begin{proof}
The inclusion $P\subseteq P'$ induces a restriction map
\[
    \rho\colon
    \varprojlim_{q\in P'}M(q)
    \longrightarrow
    \varprojlim_{p\in P}M(p)
\]
and a canonical map
\[
    \eta\colon
    \varinjlim_{p\in P}M(p)
    \longrightarrow
    \varinjlim_{q\in P'}M(q).
\]
Fix $p\in P$, and denote the corresponding limit projections by $\pi_p^P$ and $\pi_p^{P'}$ and the corresponding colimit maps by $\iota_p^P$ and $\iota_p^{P'}$. By the definitions of $\rho$ and $\eta$,
\[
    \pi_p^P\circ\rho=\pi_p^{P'},
    \qquad
    \eta\circ\iota_p^P=\iota_p^{P'}.
\]
Since $P$ and $P'$ are connected, their canonical limit-to-colimit maps satisfy
\[
    \psi_{M|_P}
    =
    \iota_p^P\circ\pi_p^P,
    \qquad
    \psi_{M|_{P'}}
    =
    \iota_p^{P'}\circ\pi_p^{P'}.
\]
Consequently,
\[
    \psi_{M|_{P'}}
    =
    \eta\circ\psi_{M|_P}\circ\rho.
\]
Thus, $\psi_{M|_{P'}}$ factors through $\psi_{M|_P}$. Since composition of linear maps cannot increase rank,
\[
    \operatorname{rank}\bigl(\psi_{M|_{P'}}\bigr)
    \le
    \operatorname{rank}\bigl(\psi_{M|_P}\bigr),
\]
which proves the result.
\end{proof}

\section{Measurability of Intervals in \texorpdfstring{$\R^n$}{R\textasciicircum n}}
\label{app:measurability-of-intervals}

We give an alternative proof of Proposition~\ref{prop:interval-measurability}. The argument first establishes the measurability of arbitrary lower and upper sets in the product order.

\begin{proof}[Alternative proof of Proposition~\ref{prop:interval-measurability}]
    Recall that $D \subseteq \R^n$ is a lower set if $y \in D$ and $x \preceq y$ imply $x \in D$. Set $\mathbf{u} \coloneqq \frac{1}{\sqrt{n}}(1, \ldots, 1)$. We first show that $\partial D$ intersects every line parallel to $\mathbf{u}$ in at most one point. Suppose, to the contrary, that $p, q \in \partial D$ satisfy $q = p + r\mathbf{u}$ for some $r > 0$. Since $q \in \overline{D}$, we may choose $y \in D$ such that $\lVert y - q\rVert_\infty < r/(2\sqrt{n})$. Then $y_i > p_i$ for every $i$. Set $\delta \coloneqq \min_i(y_i - p_i) > 0$. If $\lVert x - p\rVert_\infty < \delta$, then $x \preceq y$, and hence $x \in D$. Thus, $p \in \operatorname{int}(D)$, contradicting $p \in \partial D$.

    For $z \in \mathbf{u}^{\perp}$, let
    \[
        B_z
        \coloneqq
        \left\{
            s \in \R :
            z + s\mathbf{u} \in \partial D
        \right\}.
    \]
    Each $B_z$ contains at most one point and therefore has one-dimensional Lebesgue measure zero. The map $\Phi \colon \mathbf{u}^{\perp} \times \R \longrightarrow \R^n$ defined by $\Phi(z, s) = z + s\mathbf{u}$ is a linear isometry. Since $\partial D$ is closed, $\Phi^{-1}(\partial D)$ is measurable. Tonelli's theorem therefore gives
    \[
        \mathcal{L}^n(\partial D)
        =
        \int_{\mathbf{u}^{\perp}}
        \mathcal{L}^1(B_z)\,
        d\mathcal{L}^{n-1}(z)
        =
        0.
    \]
    Since $\operatorname{int}(D)$ is open and $D \setminus \operatorname{int}(D) \subseteq \partial D$, the completeness of Lebesgue measure implies that $D$ is Lebesgue-measurable. Applying this conclusion after reflection through the origin shows that every upper set is also Lebesgue-measurable.

    Now let $I \subseteq \R^n$ be an interval and define
    \[
        I^\downarrow
        \coloneqq
        \{x \in \R^n : x \preceq y \text{ for some } y \in I\},
        \qquad
        I^\uparrow
        \coloneqq
        \{x \in \R^n : y \preceq x \text{ for some } y \in I\}.
    \]
    Clearly, $I \subseteq I^\downarrow \cap I^\uparrow$. Conversely, if $x \in I^\downarrow \cap I^\uparrow$, then there exist $y, z \in I$ such that $z \preceq x \preceq y$. The first condition in the definition of an interval implies that $x \in I$. Hence, $I = I^\downarrow \cap I^\uparrow$. Since $I^\downarrow$ is a lower set and $I^\uparrow$ is an upper set, both are Lebesgue-measurable, and therefore so is $I$.
\end{proof}

\section{Further Properties of the Integrated Landscape}
\label{app:integrated-landscape-lipschitz}

We prove part~(4) of Theorem~\ref{thm:integrated-landscape-properties}. The argument first relates the landscape intervals at two comparable parameter values and then specializes to parameter values that differ in one coordinate.

\begin{lemma}
\label{lem:landscape-interval-translation}
    Let $M \in \Vect^{\R^n}$, $k \in \N_{>0}$, and $a, b \in \R^n$ satisfy $a \preceq b$, and set $\delta \coloneqq b - a$. If $h \in \I_M^k(a)$ and $\delta \preceq h$, then $h - \delta \in \I_M^k(b)$. Likewise, if $h \in \I_M^k(b)$ and $\delta \preceq h$, then $h - \delta \in \I_M^k(a)$.
\end{lemma}

\begin{proof}
    Suppose first that $h \in \I_M^k(a)$ and $\delta \preceq h$. Then
    \[
        a - h
        \preceq
        2b - a - h
        \preceq
        a + h.
    \]
    By functoriality,
    \[
        M_{a - h}^{a + h}
        =
        M_{2b - a - h}^{a + h}
        \circ
        M_{a - h}^{2b - a - h},
    \]
    and therefore $\rk_M(2b - a - h, a + h) \ge \rk_M(a - h, a + h) \ge k$. Since $2b - a - h = b - (h - \delta)$ and $a + h = b + (h - \delta)$, it follows that $h - \delta \in \I_M^k(b)$.

    Now suppose that $h \in \I_M^k(b)$ and $\delta \preceq h$. Then
    \[
        b - h
        \preceq
        2a - b + h
        \preceq
        b + h.
    \]
    By functoriality,
    \[
        M_{b - h}^{b + h}
        =
        M_{2a - b + h}^{b + h}
        \circ
        M_{b - h}^{2a - b + h},
    \]
    so $\rk_M(b - h, 2a - b + h) \ge \rk_M(b - h, b + h) \ge k$. Since $b - h = a - (h - \delta)$ and $2a - b + h = a + (h - \delta)$, we obtain $h - \delta \in \I_M^k(a)$.
\end{proof}

\begin{lemma}
\label{lem:one-coordinate}
    Let $M \in \Vect^{\R^n}$. Suppose that $a, b \in \R^n$ differ in at most one coordinate. Then, for every $r \in \R_{\ge 0}$ and $k \in \N_{>0}$,
    \[
        \left|
            \IL_M(r, k, a)
            -
            \IL_M(r, k, b)
        \right|
        \le
        \lVert a - b \rVert_\infty.
    \]
\end{lemma}

\begin{proof}
    The result is immediate if $a = b$ or $r = 0$. Otherwise, after interchanging $a$ and $b$ if necessary, assume that $a \preceq b$. Since they differ in one coordinate, there is an index $j$ such that $b - a = d\mathbf{e}_j$, where $d \coloneqq \lVert a - b \rVert_\infty$ and $\mathbf{e}_j$ is the $j$th standard basis vector. Set $A \coloneqq Q_r \cap \I_M^k(a)$ and $B \coloneqq Q_r \cap \I_M^k(b)$.

    Lemma~\ref{lem:landscape-interval-translation} shows that the translation $h \mapsto h - d\mathbf{e}_j$ maps $A \cap \{h : h_j \ge d\}$ into $B$. Indeed, the translated point belongs to $\I_M^k(b)$, and subtracting $d$ from a coordinate lying in $[d, r]$ keeps it in $Q_r$. Since translation preserves Lebesgue measure,
    \[
    \begin{aligned}
        \mathcal{L}^n(A)
        &\le
        \mathcal{L}^n(B)
        +
        \mathcal{L}^n\bigl(Q_r \cap \{h : h_j < d\}\bigr)\\
        &\le
        \mathcal{L}^n(B)
        +
        d r^{n-1}.
    \end{aligned}
    \]
    The second assertion of Lemma~\ref{lem:landscape-interval-translation} similarly gives $\mathcal{L}^n(B) \le \mathcal{L}^n(A) + d r^{n-1}$. Hence,
    \[
        \left|
            \mathcal{L}^n(A)
            -
            \mathcal{L}^n(B)
        \right|
        \le
        d r^{n-1}.
    \]
    Dividing by $r^{n-1}$ yields $|\IL_M(r, k, a) - \IL_M(r, k, b)| \le d = \lVert a - b \rVert_\infty$.
\end{proof}

\begin{proof}[Proof of part~(4) of Theorem~\ref{thm:integrated-landscape-properties}]
    Fix $r \in \R_{\ge 0}$ and $k \in \N_{>0}$. For $a, b \in \R^n$, define
    \[
        c_i
        \coloneqq
        (b_1, \ldots, b_i, a_{i+1}, \ldots, a_n),
        \qquad
        0 \le i \le n.
    \]
    Then $c_0 = a$, $c_n = b$, and each consecutive pair differs in at most one coordinate. Lemma~\ref{lem:one-coordinate} gives
    \[
    \begin{aligned}
        \left|
            \IL_M(r, k, a)
            -
            \IL_M(r, k, b)
        \right|
        &\le
        \sum_{i = 1}^n
        \left|
            \IL_M(r, k, c_{i-1})
            -
            \IL_M(r, k, c_i)
        \right|\\
        &\le
        \sum_{i = 1}^n |a_i - b_i|
        =
        \lVert a - b \rVert_1
        \le
        n \lVert a - b\rVert_\infty.
    \end{aligned}
    \]
\end{proof}

\subsection{Proof of the general integrated-curve structure theorem}
\label{app:piecewise-integrated-curve-proof}

\begin{proof}[Proof of Theorem~\ref{thm:piecewise-form-integrated-curves}]
    Let $P_+ \coloneqq P \cap \R_{>0}^n$. If $P_+$ is empty, then every rectangle in $\mathcal{R}_P$ has $n$-dimensional Lebesgue measure zero. Conversely, if $P_+$ is nonempty, then $\mathcal{R}_P$ contains a rectangle of positive $n$-dimensional measure. Hence, $\mathcal{L}^n(\mathcal{R}_P) = 0$ if and only if $P_+$ is empty. In this case, $\Gamma_P$ is identically zero, and all the assertions follow. We therefore assume that $P_+$ is nonempty.

    Rectangles indexed by $P \setminus P_+$ have measure zero and do not affect $\Gamma_P$. We may therefore replace $P$ by $P_+$ and write $P = \{p_1, \ldots, p_K\} \subseteq \R_{>0}^n$. For every $r \ge 0$,
    \[
        Q_r \cap \mathcal{R}_P
        =
        \bigcup_{j = 1}^K
        \bigl(Q_r \cap [\mathbf{0},p_j]\bigr)
        =
        \bigcup_{j = 1}^K
        \prod_{s = 1}^n
        [0,\min\{r,(p_j)_s\}].
    \]
    For every nonempty $A \subseteq \{1, \ldots, K\}$, the intersection of the rectangles indexed by $A$ is
    \[
        \bigcap_{j \in A}
        \prod_{s = 1}^n
        [0,\min\{r,(p_j)_s\}]
        =
        \prod_{s = 1}^n
        \left[
            0,
            \min\left\{
                r,
                \min_{j \in A}(p_j)_s
            \right\}
        \right].
    \]
    Inclusion--exclusion therefore gives
    \[
        \mathcal{L}^n\bigl(Q_r \cap \mathcal{R}_P\bigr)
        =
        \sum_{\varnothing \neq A \subseteq \{1, \ldots, K\}}
        (-1)^{|A|+1}
        \prod_{s = 1}^n
        \min\left\{
            r,
            \min_{j \in A}(p_j)_s
        \right\}.
    \]
    
    Let $\Sigma_P$ be the set of distinct coordinates of the points in $P$, written as
    \[
        \Sigma_P
        \coloneqq
        \left\{
            (p_j)_s :
            1 \le j \le K,\ 
            1 \le s \le n
        \right\}
        =
        \{q_1 < \cdots < q_m\}.
    \]
    Since $P \subseteq \R_{>0}^n$, we have $q_1 > 0$. Define
    \[
    \begin{aligned}
        \tau_P
        &\coloneqq
        \sup
        \left\{
            q \ge 0 :
            q\mathbf{1}_n \in \mathcal{R}_P
        \right\}
        =
        \max_{1 \le j \le K}
        \min_{1 \le s \le n}(p_j)_s,\\
        T_P
        &\coloneqq
        q_m
        =
        \max_{\substack{1 \le j \le K\\1 \le s \le n}}
        (p_j)_s.
    \end{aligned}
    \]
    Then $0 < \tau_P \le T_P$.

    Fix a connected component $J$ of $(0,\infty)\setminus\Sigma_P$. For every nonempty $A \subseteq \{1, \ldots, K\}$ and every $s$, the factor $\min\left\{r,\min_{j \in A}(p_j)_s\right\}$ is either $r$ throughout $J$ or a fixed coordinate value throughout $J$. Consequently,
    \[
        \prod_{s = 1}^n
        \min\left\{
            r,
            \min_{j \in A}(p_j)_s
        \right\}
        =
        c_{A,J}r^{d_{A,J}}
    \]
    for some $c_{A,J} > 0$ and $d_{A,J} \in \{0, \ldots, n\}$. After grouping terms of the same degree, we obtain $\mathcal{L}^n\bigl(Q_r \cap \mathcal{R}_P\bigr) = \sum_{d = 0}^n a_{J,d}r^d$ for suitable coefficients $a_{J,d} \in \R$. Hence, for $r \in J$,
    \[
        \Gamma_P(r)
        =
        \sum_{d = 0}^n
        a_{J,d}r^{d-(n-1)}
        =
        \sum_{\ell = -(n-1)}^1
        c_{J,\ell}r^\ell.
    \]

    If $0 < r \le \tau_P$, then $r\mathbf{1}_n \in \mathcal{R}_P$. Since $\mathcal{R}_P$ is closed under decreasing coordinates, $Q_r \subseteq \mathcal{R}_P$. Therefore, $\mathcal{L}^n(Q_r \cap \mathcal{R}_P) = r^n$ and $\Gamma_P(r) = r$. If $r \ge T_P$, then $\mathcal{R}_P \subseteq Q_r$, so $\Gamma_P(r) = \mathcal{L}^n(\mathcal{R}_P)/r^{n-1}$. The inclusion--exclusion formula is a finite sum of continuous functions, so $\Gamma_P$ is continuous on $(0,\infty)$. Since $0 \le \Gamma_P(r) \le r$, it is also continuous at $r = 0$.

    Finally, on every component $J$,
    \[
        r^n\Gamma_P'(r)
        =
        \sum_{\ell = -(n-1)}^1
        \ell c_{J,\ell}r^{\ell+n-1}.
    \]
    The right-hand side is a polynomial of degree at most $n$. If it is identically zero, then $\Gamma_P$ is constant on $J$. Otherwise, it has at most $n$ roots, and the sign of $\Gamma_P'$ is constant and nonzero between consecutive roots. Subdividing each component at these roots and adjoining $0$ and the finitely many points of $\Sigma_P$ gives a finite subdivision of $[0,\infty)$. Continuity of $\Gamma_P$ allows the resulting intervals to include their endpoints, and $\Gamma_P$ is strictly increasing, constant, or strictly decreasing on each such interval.
\end{proof}

\subsection{Proof of the two-parameter curve-structure theorem}
\label{app:two-parameter-curve-proof}

\begin{proof}[Proof of Theorem~\ref{thm:two-parameter-curve-shape}]
    Rectangles of zero area do not affect $\Gamma_P$, while a point of $P$ that is coordinatewise dominated by another point may be removed without changing $\mathcal{R}_P$. After making these reductions and relabeling the remaining points, write $P=\{(x_i,y_i):1 \le i \le K\}$, where
    \[
        0=x_0<x_1<\cdots<x_K<x_{K+1}=\infty,
        \qquad
        \infty=y_0>y_1>\cdots>y_K>y_{K+1}=0.
    \]
    Then $\mathcal{R}_P=\bigcup_{i=1}^K[0,x_i]\times[0,y_i]$. For this ordering, the quantities from Theorem~\ref{thm:piecewise-form-integrated-curves} are
    \[
        \tau_P
        =
        \max_{1 \le i \le K}\min\{x_i,y_i\},
        \qquad
        T_P
        =
        \max\{x_K,y_1\}.
    \]
    That theorem gives $\Gamma_P(r)=r$ for $0 \le r \le \tau_P$ and $\Gamma_P(r)=\mathcal{L}^2(\mathcal{R}_P)/r$ for $r \ge T_P$. If $\tau_P=T_P$, the result follows immediately, so suppose that $\tau_P<T_P$.

    Let $\Sigma_P\coloneqq\{x_i,y_i:1 \le i \le K\}$, and fix a connected component $U$ of $(\tau_P,T_P)\setminus\Sigma_P$. For $r \in U$, there are unique indices $j$ and $m$ such that
    \[
        x_{j-1}<r<x_j,
        \qquad
        y_{m+1}<r<y_m.
    \]
    We must have $j>m$. Indeed, if $j \le m$, then $x_j>r$ and $y_j \ge y_m>r$, so $(r,r)\preceq(x_j,y_j)$. This would imply $(r,r)\in\mathcal{R}_P$, contradicting $r>\tau_P$.

    Set $A(r)\coloneqq\mathcal{L}^2\bigl(Q_r\cap\mathcal{R}_P\bigr)$. Over $[0,x_m]$, the vertical height of $Q_r\cap\mathcal{R}_P$ is $r$; over $(x_{i-1},x_i]$, it is $y_i$ for $m+1 \le i \le j-1$; and over $(x_{j-1},r]$, it is $y_j$. Therefore,
    \[
    \begin{aligned}
        A(r)
        &=
        x_mr
        +
        \sum_{i=m+1}^{j-1}
        (x_i-x_{i-1})y_i
        +
        (r-x_{j-1})y_j\\
        &=
        (x_m+y_j)r
        +
        \sum_{i=m+1}^{j-1}
        (x_i-x_{i-1})y_i
        -
        x_{j-1}y_j,
    \end{aligned}
    \]
    where an empty sum is understood to be zero. Thus, $A$ is affine on $U$, and $\Gamma_P(r)=A(r)/r$ has the form $\alpha_U+\beta_U/r$.

    As $r$ increases through $(\tau_P,T_P)$, the index $j$ is nondecreasing and $m$ is nonincreasing. Hence, both $x_m$ and $y_j$ are nonincreasing, so the slope $A'(r)=x_m+y_j$ is nonincreasing from one component to the next. The preceding theorem shows that $\Gamma_P$, and hence $A$, is continuous. On each component, set $F(r)\coloneqq rA'(r)-A(r)$. Since $A$ is affine there, $F$ is constant and $\Gamma_P'(r)=F(r)/r^2$. Consider two consecutive components separated by a breakpoint $q$, and let $s_-$ and $s_+$ be the slopes of the affine pieces ending and beginning at $q$, respectively. Since the slopes of $A$ are nonincreasing, $s_+ \le s_-$. By continuity of $A$, the constant values of $F$ on these components are $q s_- - A(q)$ and $q s_+ - A(q)$, respectively. Thus, the constant values of $F$ on successive components are nonincreasing.

    Since $r^2>0$, the sign of $\Gamma_P'$ is the sign of $F$, so positive, zero, and negative signs can occur only in this order, with some cases possibly absent. Together with the initial and final formulas for $\Gamma_P$, this proves the asserted unimodality. Finally, on a component where $\Gamma_P(r)=\alpha_U+\beta_U/r$, we have $\Gamma_P'(r)=-\beta_U/r^2$. Hence, $\beta_U<0$ on a strictly increasing piece, $\beta_U=0$ on a constant piece, and $\beta_U>0$ on a strictly decreasing piece. This gives the stated forms and completes the proof.
\end{proof}

\section{Experimental Details}
\label{app:experimental-details}

This appendix records the preprocessing, feature-sampling, and model-selection details for the classification experiments in Section~\ref{sec:applications-and-experiments}. The code used for these experiments, together with the fixed SHREC15 training--test split manifest, is available on \href{https://github.com/pfanghaha/landscape-interval}{GitHub}. Persistent homology was computed over $\mathbb{F}_{11}$ using GUDHI's default coefficient field.

\subsection{SHREC15 shape classification}
\label{app:shrec15-details}

\paragraph{Mesh preprocessing and scalar functions.}
The mesh simplification described in Subsection~\ref{subsec:shrec15} was performed using PyMeshLab's quadric edge-collapse decimation \cite{CignoniEtAl2008MeshLab}. Each
mesh was simplified to approximately $1000$ vertices, and connectedness was checked after simplification. The resulting meshes contained between $807$ and $1236$ vertices, with an average of approximately $1025$ vertices.

The HKS and WKS were computed using $100$ Laplace--Beltrami eigenpairs. From $64$ HKS time samples and $128$ WKS energy samples, we selected the samples nearest diffusion time $0.01$ and energy $5$, respectively. The $q$NN local-sparsity function used $q=20$ nearest neighbors. Each scalar function was min--max normalized separately on each mesh.

\paragraph{One-parameter features.}
Persistent homology in degrees $0$ and $1$ was computed using GUDHI \cite{MariaEtAl2014GUDHI}. Intervals with infinite death time were truncated at $1$, the upper endpoint of the normalized filtration range. For each homological degree, we retained the first $L$ landscape layers and sampled each layer at $R$ equally spaced values in $[0,1]$. The resulting feature vector had dimension $LR$. We considered
\[
    L \in \{1, \ldots, 30\},
    \qquad
    R \in \{50,100,150,200,300\}.
\]

\paragraph{Two-parameter features.}
The primary parameter set consisted of the $900$ points in a $30 \times 30$ grid on $[0.2,1]^2$. We placed these points in a nested farthest-point order so that each initial segment remained distributed across the square, and considered
\[
    P \in \{100,300,500,700,900\}
\]
retained grid points. Every $2$-parameter representation also included $200$ fixed parameter values: $100$ points of the form $(j/100,1.5)$ and $100$ points of the form $(1.5,j/100)$, where $j \in \{0,\ldots,99\}$. Because the normalized filtration values lie in $[0,1]$, these additional points sample regions in which one filtration coordinate has passed the end of its range. Thus, every representation used $P+200$ parameter values.

At each parameter value and landscape layer, we used the direction set
\[
    V
    \coloneqq
    \left\{
        \frac{
            \bigl(\cos(j\pi/20),\sin(j\pi/20)\bigr)
        }{
            \max\{\cos(j\pi/20),\sin(j\pi/20)\}
        } :
        j=0,\ldots,10
    \right\}.
\]
Hence, $V$ consists of $11$ $\ell^\infty$-unit directions equally spaced by angle in the first quadrant and includes both coordinate axes and the diagonal. For each parameter value $t$ and sampled
direction $v$, we restricted the bifiltration to the line through $t$ in direction $v$ and computed the resulting one-parameter persistence landscape using GUDHI \cite{MariaEtAl2014GUDHI}. The diagonal value gave the uniform landscape, while the complete direction set was used to construct the upper finite-direction approximation of the integrated landscape curve.

For each homological degree, we considered the first $L$ landscape layers, with $L \in \{1,\ldots,30\}$. A uniform feature vector contained one value for each of the $P+200$ parameter values and each layer, and therefore had dimension $(P+200)L$. Integrated landscape curves were first evaluated at $300$ equally
spaced scales in $[0,1]$. During model selection, these curves were linearly resampled at $R \in \{50,100,150,200,300\}$ equally spaced scales, producing feature vectors of dimension $(P+200)LR$. The uniform and integrated representations were evaluated separately, as were homological degrees $0$ and $1$.

\paragraph{Classification and model selection.}
We used linear and radial basis function support vector machines implemented in scikit-learn \cite{PedregosaEtAl2011ScikitLearn}. Each feature was standardized to have zero mean and unit variance. Within each cross-validation fold, the standardization parameters were estimated from the training portion and then applied to the validation portion.

Feature and model parameters were selected jointly by stratified five-fold cross-validation on the training set, using the same shuffled folds with random state $0$ for every representation. For the one-parameter representations, we selected $L$ and $R$; for the uniform representations, we selected $P$ and $L$; and for the integrated representations, we selected $P$, $L$, and $R$. We considered
\[
    C \in \{0.001,0.01,0.1,1,10,100\}.
\]
For the RBF kernel, we considered
\[
    \gamma
    \in
    \left\{
        0.1\gamma_{\mathrm{scale}},
        \gamma_{\mathrm{scale}},
        10\gamma_{\mathrm{scale}}
    \right\},
\]
where $\gamma_{\mathrm{scale}}$ is the standard scikit-learn \texttt{scale} value computed from the training portion of the corresponding fold after standardization.

Configurations were selected first by mean cross-validation accuracy and then by mean macro F1-score. Any remaining exact ties were resolved in favor of fewer features, the linear kernel, and the smaller value of $C$, in that order. The test set was not used for feature, model, or hyperparameter selection. After selection, the chosen configuration was fitted to the complete training set and evaluated on the held-out test set.

\subsubsection{Selected configurations}
\label{app:shrec15-selected-configurations}

Table~\ref{tab:shrec15-selected-configurations} reports the feature and SVM configurations selected by cross-validation. For a one-parameter representation, the feature parameters are $(L,R)$. For a uniform representation, they are $(P,L)$. For an integrated representation, they are $(P,L,R)$. Here, $P$ denotes the number of retained points from the primary parameter grid. Every $2$-parameter representation also includes the $200$ fixed parameter values described above.

\begin{table*}[p]
\centering
\caption{Feature and SVM configurations selected for the SHREC15 representations in homological degrees $0$ and $1$.}
\label{tab:shrec15-selected-configurations}
\scriptsize
\setlength{\tabcolsep}{3pt}
\renewcommand{\arraystretch}{1.02}
\begin{adjustbox}{max width=\textwidth}
\begin{tabular}{@{}lllccc@{}}
\toprule
Filtration function or pair
& Representation
& Feature parameters
& Kernel
& $C$
& $\gamma$ \\
\midrule
\multicolumn{6}{c}{\textit{Homological degree $0$}} \\
\addlinespace[2pt]
HKS
& One-parameter
& $(4,300)$
& RBF
& $100$
& $\gamma_{\mathrm{scale}}$ \\
WKS
& One-parameter
& $(9,200)$
& Linear
& $10^{-2}$
& -- \\
Center
& One-parameter
& $(8,200)$
& Linear
& $10^{-2}$
& -- \\
$q$NN
& One-parameter
& $(14,50)$
& RBF
& $100$
& $0.1\gamma_{\mathrm{scale}}$ \\
\addlinespace[2pt]
HKS--center
& Integrated
& $(700,5,300)$
& Linear
& $10^{-3}$
& -- \\
HKS--center
& Uniform
& $(100,7)$
& Linear
& $10^{-2}$
& -- \\
\addlinespace[1pt]
HKS--$q$NN
& Integrated
& $(100,12,150)$
& RBF
& $10$
& $\gamma_{\mathrm{scale}}$ \\
HKS--$q$NN
& Uniform
& $(100,11)$
& Linear
& $10^{-1}$
& -- \\
\addlinespace[1pt]
HKS--WKS
& Integrated
& $(100,9,50)$
& Linear
& $10^{-3}$
& -- \\
HKS--WKS
& Uniform
& $(100,4)$
& Linear
& $10^{-1}$
& -- \\
\addlinespace[1pt]
$q$NN--center
& Integrated
& $(900,8,300)$
& Linear
& $10^{-3}$
& -- \\
$q$NN--center
& Uniform
& $(700,10)$
& Linear
& $10^{-3}$
& -- \\
\addlinespace[1pt]
WKS--center
& Integrated
& $(100,8,50)$
& Linear
& $10^{-3}$
& -- \\
WKS--center
& Uniform
& $(300,10)$
& Linear
& $10^{-2}$
& -- \\
\addlinespace[1pt]
WKS--$q$NN
& Integrated
& $(100,11,50)$
& Linear
& $10^{-3}$
& -- \\
WKS--$q$NN
& Uniform
& $(100,13)$
& Linear
& $10^{-3}$
& -- \\
\midrule
\multicolumn{6}{c}{\textit{Homological degree $1$}} \\
\addlinespace[2pt]
HKS
& One-parameter
& $(11,150)$
& RBF
& $10$
& $\gamma_{\mathrm{scale}}$ \\
WKS
& One-parameter
& $(23,200)$
& Linear
& $10^{-2}$
& -- \\
Center
& One-parameter
& $(1,150)$
& RBF
& $1$
& $10\gamma_{\mathrm{scale}}$ \\
$q$NN
& One-parameter
& $(25,50)$
& RBF
& $10$
& $10\gamma_{\mathrm{scale}}$ \\
\addlinespace[2pt]
HKS--center
& Integrated
& $(100,13,50)$
& Linear
& $10^{-3}$
& -- \\
HKS--center
& Uniform
& $(100,10)$
& Linear
& $10^{-2}$
& -- \\
\addlinespace[1pt]
HKS--$q$NN
& Integrated
& $(100,11,50)$
& Linear
& $10^{-3}$
& -- \\
HKS--$q$NN
& Uniform
& $(100,10)$
& RBF
& $100$
& $0.1\gamma_{\mathrm{scale}}$ \\
\addlinespace[1pt]
HKS--WKS
& Integrated
& $(100,14,50)$
& Linear
& $10^{-3}$
& -- \\
HKS--WKS
& Uniform
& $(100,11)$
& Linear
& $10^{-2}$
& -- \\
\addlinespace[1pt]
$q$NN--center
& Integrated
& $(300,17,300)$
& Linear
& $10^{-3}$
& -- \\
$q$NN--center
& Uniform
& $(500,9)$
& Linear
& $10^{-2}$
& -- \\
\addlinespace[1pt]
WKS--center
& Integrated
& $(100,20,200)$
& Linear
& $10^{-3}$
& -- \\
WKS--center
& Uniform
& $(100,14)$
& Linear
& $10^{-2}$
& -- \\
\addlinespace[1pt]
WKS--$q$NN
& Integrated
& $(700,22,150)$
& Linear
& $10^{-3}$
& -- \\
WKS--$q$NN
& Uniform
& $(700,25)$
& Linear
& $10^{-3}$
& -- \\
\bottomrule
\end{tabular}
\end{adjustbox}
\par\smallskip
\begin{minipage}{\textwidth}
\scriptsize
\textit{Notes.}
A dash indicates that $\gamma$ is not applicable because the linear kernel was selected. The quantity $\gamma_{\mathrm{scale}}$ is the standard scale value defined above.
\end{minipage}
\end{table*}

\subsection{MedMNIST classification}
\label{app:medmnist-details}

\paragraph{Image preprocessing and normalization.}
For the two-dimensional datasets, each image was converted to grayscale and its pixel values were divided by $255$. The three-dimensional datasets were already scalar-valued volumes, so their voxel values were converted to floating-point values without division by $255$. For the Gaussian-smoothed intensity coordinate, we used $\sigma \in \{0,0.5,1,1.5\}$. The Laplacian-of-Gaussian coordinate used scale $1$, and its response remained signed before normalization. The Gaussian and Laplacian-of-Gaussian filters used reflected boundary handling.

For each dataset, Gaussian scale, and filtration coordinate, we computed the global minimum and maximum over all pixels or voxels in the official training split. These values were used to min--max normalize the corresponding coordinate for the training, validation, and test samples. Thus, normalization of the validation and test splits used only values obtained from the training split.

\paragraph{Landscape features.}
All directional computations were carried out within the parameter box $[0,3]^2$. The primary parameter set contained the $100$ points of a $10 \times 10$ Cartesian grid on $[0.2,1]^2$. We added $50$ points of the form $(x,1.5)$ and $50$ points of the form $(1.5,y)$, where the varying coordinate was equally spaced in $[0,1)$. Thus, $P=200$ parameter values were used for every dataset, Gaussian scale, and homological degree.

At each parameter value and landscape layer, we used the direction set
\[
    V
    \coloneqq
    \left\{
        \frac{
            \bigl(\cos(j\pi/20),\sin(j\pi/20)\bigr)
        }{
            \max\{\cos(j\pi/20),\sin(j\pi/20)\}
        } :
        j=0,\ldots,10
    \right\}.
\]
Thus, $V$ consists of $11$ $\ell^\infty$-unit directions equally spaced by angle in the first quadrant and includes both coordinate axes and the diagonal. Using the same line-restriction procedure as
in the SHREC15 experiment, we computed the sampled directional values and corresponding radial boundary points. Essential intervals were truncated at the endpoint of the bounded slice in $[0,3]^2$.

We used the first $L=7$ landscape layers. The diagonal boundary point gave the uniform landscape value, while the complete collection of sampled boundary points was used to compute the integrated landscape curve. Each integrated landscape curve was sampled at $R=50$ equally spaced scales in $[0,1]$, including both endpoints. The uniform feature dimensions were $2PL=2{,}800$ for the two-dimensional datasets and $3PL=4{,}200$ for the three-dimensional datasets. The corresponding integrated dimensions were $2PLR=140{,}000$ and $3PLR=210{,}000$. The uniform and integrated representations were evaluated separately and were not concatenated.

\paragraph{Classification protocol.}
The feature vectors were evaluated without additional standardization. We used a multilayer perceptron with three fully connected hidden layers containing $256$, $128$, and $64$ neurons, respectively, each with the ReLU activation function. The same hidden-layer architecture was used for every dataset and both landscape representations. Only the input dimension and number of output units depended on the representation and classification task.

The network was trained using the Adam optimizer with learning rate $0.001$ and batch size $32$ for at most $100$ epochs. At the end of each epoch, AUC was computed on the official validation split. Training stopped if the validation AUC did not improve for $20$ consecutive epochs, and the weights from the epoch with the highest validation AUC were restored before evaluation on the test split. The test split was not used for model selection or early stopping.

For the single-label binary and multiclass datasets, the output layer used the softmax activation function, and the network was trained using categorical cross-entropy. Accuracy was computed from the class with the largest predicted probability. For binary tasks, AUC was computed using the positive-class probability. For multiclass tasks, AUC was the unweighted mean of the one-versus-rest class AUCs.

ChestMNIST was treated according to its official multilabel task. Its output layer contained $14$ logits, one for each binary label, and the network was trained using binary cross-entropy. Sigmoid probabilities were used for evaluation. AUC was averaged over the $14$ labels, and accuracy was computed by thresholding each probability at $0.5$ and averaging the resulting label-wise accuracies.

Because of memory requirements, PathMNIST, ChestMNIST, OCTMNIST, and TissueMNIST used disk-backed feature loading. Features were stored in blocks of $256$ samples and shuffled during training in pools of $8{,}192$ samples. The remaining datasets used global per-epoch shuffling. This difference affected data loading only. The network architecture, optimization settings, and evaluation procedure were unchanged.

Each combination of dataset, Gaussian scale, and landscape representation was trained five times using the fixed random seeds $0,1,2,3$, and $4$. The same seeds were used for the uniform and integrated representations. Across the eighteen datasets, four Gaussian scales, and two representations, this produced $144$ experimental configurations and $720$ network-training runs.

\paragraph{Results across repeated runs.}
Table~\ref{tab:medmnist-landscape-variability} reports the mean and sample standard deviation of test AUC and accuracy over the five runs for each dataset, Gaussian scale, and landscape representation.

\input{sections/medmnist_appendix_table}

%% file: sections/medmnist_appendix_table.tex
\begin{table*}[p]
\centering
\captionsetup{skip=4pt}
\caption{Uniform- and integrated-landscape results on MedMNIST v2.}
\label{tab:medmnist-landscape-variability}
\begingroup
\fontsize{9}{10}\selectfont
\setlength{\tabcolsep}{3pt}
\renewcommand{\arraystretch}{0.64}
\setlength{\arrayrulewidth}{0.25pt}
\begin{tabular*}{\textwidth}{@{\extracolsep{\fill}}lc|cc|cc@{}}
\hline
\rule{0pt}{7.5pt}Dataset & $\sigma$ & \multicolumn{2}{c|}{Uniform landscape} & \multicolumn{2}{c}{Integrated landscape} \\
\cline{3-6}
\rule{0pt}{7.2pt} & & AUC & ACC & AUC & ACC \\
\hline
\rule{0pt}{7.2pt}PathMNIST & $0$ & $0.947 \pm 0.004$ & $0.731 \pm 0.015$ & $0.947 \pm 0.004$ & $0.732 \pm 0.006$ \\
 & $0.5$ & $0.941 \pm 0.003$ & $0.706 \pm 0.011$ & $0.937 \pm 0.003$ & $0.710 \pm 0.012$ \\
 & $1$ & $0.930 \pm 0.003$ & $0.682 \pm 0.008$ & $0.927 \pm 0.002$ & $0.685 \pm 0.010$ \\
 & $1.5$ & $0.925 \pm 0.005$ & $0.675 \pm 0.016$ & $0.925 \pm 0.004$ & $0.672 \pm 0.008$ \\
\hline
\rule{0pt}{7.2pt}ChestMNIST & $0$ & $0.665 \pm 0.003$ & $0.947 \pm 0.000$ & $0.655 \pm 0.003$ & $0.947 \pm 0.000$ \\
 & $0.5$ & $0.666 \pm 0.004$ & $0.947 \pm 0.000$ & $0.658 \pm 0.004$ & $0.947 \pm 0.000$ \\
 & $1$ & $0.665 \pm 0.002$ & $0.947 \pm 0.000$ & $0.655 \pm 0.006$ & $0.947 \pm 0.000$ \\
 & $1.5$ & $0.661 \pm 0.004$ & $0.947 \pm 0.000$ & $0.652 \pm 0.002$ & $0.947 \pm 0.000$ \\
\hline
\rule{0pt}{7.2pt}DermaMNIST & $0$ & $0.844 \pm 0.004$ & $0.701 \pm 0.004$ & $0.841 \pm 0.004$ & $0.700 \pm 0.003$ \\
 & $0.5$ & $0.844 \pm 0.002$ & $0.702 \pm 0.004$ & $0.839 \pm 0.004$ & $0.698 \pm 0.005$ \\
 & $1$ & $0.840 \pm 0.003$ & $0.703 \pm 0.004$ & $0.835 \pm 0.003$ & $0.700 \pm 0.003$ \\
 & $1.5$ & $0.840 \pm 0.005$ & $0.703 \pm 0.004$ & $0.838 \pm 0.003$ & $0.697 \pm 0.008$ \\
\hline
\rule{0pt}{7.2pt}OCTMNIST & $0$ & $0.899 \pm 0.007$ & $0.626 \pm 0.016$ & $0.902 \pm 0.003$ & $0.628 \pm 0.010$ \\
 & $0.5$ & $0.908 \pm 0.003$ & $0.635 \pm 0.008$ & $0.910 \pm 0.003$ & $0.625 \pm 0.021$ \\
 & $1$ & $0.908 \pm 0.003$ & $0.623 \pm 0.014$ & $0.912 \pm 0.005$ & $0.608 \pm 0.004$ \\
 & $1.5$ & $0.902 \pm 0.007$ & $0.606 \pm 0.017$ & $0.905 \pm 0.003$ & $0.620 \pm 0.009$ \\
\hline
\rule{0pt}{7.2pt}PneumoniaMNIST & $0$ & $0.892 \pm 0.015$ & $0.777 \pm 0.024$ & $0.899 \pm 0.008$ & $0.802 \pm 0.025$ \\
 & $0.5$ & $0.895 \pm 0.010$ & $0.795 \pm 0.036$ & $0.902 \pm 0.002$ & $0.804 \pm 0.023$ \\
 & $1$ & $0.895 \pm 0.005$ & $0.813 \pm 0.018$ & $0.906 \pm 0.005$ & $0.805 \pm 0.045$ \\
 & $1.5$ & $0.889 \pm 0.013$ & $0.797 \pm 0.013$ & $0.898 \pm 0.006$ & $0.814 \pm 0.011$ \\
\hline
\rule{0pt}{7.2pt}RetinaMNIST & $0$ & $0.643 \pm 0.004$ & $0.510 \pm 0.013$ & $0.645 \pm 0.007$ & $0.497 \pm 0.023$ \\
 & $0.5$ & $0.631 \pm 0.006$ & $0.495 \pm 0.020$ & $0.634 \pm 0.006$ & $0.495 \pm 0.026$ \\
 & $1$ & $0.624 \pm 0.006$ & $0.495 \pm 0.014$ & $0.616 \pm 0.007$ & $0.489 \pm 0.014$ \\
 & $1.5$ & $0.625 \pm 0.003$ & $0.488 \pm 0.033$ & $0.621 \pm 0.010$ & $0.476 \pm 0.018$ \\
\hline
\rule{0pt}{7.2pt}BreastMNIST & $0$ & $0.800 \pm 0.004$ & $0.796 \pm 0.019$ & $0.801 \pm 0.006$ & $0.763 \pm 0.015$ \\
 & $0.5$ & $0.832 \pm 0.006$ & $0.783 \pm 0.036$ & $0.821 \pm 0.006$ & $0.768 \pm 0.019$ \\
 & $1$ & $0.851 \pm 0.010$ & $0.772 \pm 0.102$ & $0.839 \pm 0.017$ & $0.762 \pm 0.037$ \\
 & $1.5$ & $0.837 \pm 0.012$ & $0.782 \pm 0.032$ & $0.828 \pm 0.010$ & $0.754 \pm 0.026$ \\
\hline
\rule{0pt}{7.2pt}BloodMNIST & $0$ & $0.990 \pm 0.000$ & $0.893 \pm 0.003$ & $0.989 \pm 0.001$ & $0.884 \pm 0.007$ \\
 & $0.5$ & $0.989 \pm 0.000$ & $0.886 \pm 0.003$ & $0.987 \pm 0.001$ & $0.880 \pm 0.006$ \\
 & $1$ & $0.987 \pm 0.000$ & $0.878 \pm 0.002$ & $0.986 \pm 0.000$ & $0.872 \pm 0.003$ \\
 & $1.5$ & $0.988 \pm 0.001$ & $0.879 \pm 0.005$ & $0.985 \pm 0.001$ & $0.867 \pm 0.005$ \\
\hline
\rule{0pt}{7.2pt}TissueMNIST & $0$ & $0.871 \pm 0.000$ & $0.565 \pm 0.001$ & $0.867 \pm 0.001$ & $0.557 \pm 0.002$ \\
 & $0.5$ & $0.868 \pm 0.000$ & $0.561 \pm 0.003$ & $0.864 \pm 0.001$ & $0.555 \pm 0.003$ \\
 & $1$ & $0.860 \pm 0.001$ & $0.552 \pm 0.001$ & $0.856 \pm 0.000$ & $0.544 \pm 0.002$ \\
 & $1.5$ & $0.857 \pm 0.001$ & $0.546 \pm 0.003$ & $0.853 \pm 0.001$ & $0.540 \pm 0.004$ \\
\hline
\rule{0pt}{7.2pt}OrganAMNIST & $0$ & $0.955 \pm 0.001$ & $0.648 \pm 0.007$ & $0.952 \pm 0.001$ & $0.637 \pm 0.012$ \\
 & $0.5$ & $0.957 \pm 0.001$ & $0.649 \pm 0.016$ & $0.950 \pm 0.002$ & $0.637 \pm 0.007$ \\
 & $1$ & $0.953 \pm 0.001$ & $0.639 \pm 0.012$ & $0.948 \pm 0.001$ & $0.627 \pm 0.004$ \\
 & $1.5$ & $0.950 \pm 0.001$ & $0.630 \pm 0.008$ & $0.946 \pm 0.002$ & $0.612 \pm 0.012$ \\
\hline
\rule{0pt}{7.2pt}OrganCMNIST & $0$ & $0.952 \pm 0.001$ & $0.645 \pm 0.011$ & $0.951 \pm 0.001$ & $0.643 \pm 0.009$ \\
 & $0.5$ & $0.954 \pm 0.001$ & $0.668 \pm 0.003$ & $0.951 \pm 0.001$ & $0.650 \pm 0.006$ \\
 & $1$ & $0.949 \pm 0.001$ & $0.645 \pm 0.009$ & $0.944 \pm 0.001$ & $0.628 \pm 0.013$ \\
 & $1.5$ & $0.947 \pm 0.001$ & $0.632 \pm 0.007$ & $0.944 \pm 0.002$ & $0.632 \pm 0.008$ \\
\hline
\rule{0pt}{7.2pt}OrganSMNIST & $0$ & $0.945 \pm 0.001$ & $0.640 \pm 0.011$ & $0.940 \pm 0.001$ & $0.634 \pm 0.007$ \\
 & $0.5$ & $0.945 \pm 0.002$ & $0.646 \pm 0.010$ & $0.940 \pm 0.002$ & $0.625 \pm 0.005$ \\
 & $1$ & $0.937 \pm 0.002$ & $0.619 \pm 0.008$ & $0.932 \pm 0.004$ & $0.596 \pm 0.017$ \\
 & $1.5$ & $0.934 \pm 0.001$ & $0.611 \pm 0.013$ & $0.929 \pm 0.002$ & $0.601 \pm 0.005$ \\
\hline
\rule{0pt}{7.2pt}OrganMNIST3D & $0$ & $0.956 \pm 0.003$ & $0.651 \pm 0.016$ & $0.956 \pm 0.002$ & $0.652 \pm 0.018$ \\
 & $0.5$ & $0.960 \pm 0.001$ & $0.656 \pm 0.007$ & $0.960 \pm 0.002$ & $0.657 \pm 0.023$ \\
 & $1$ & $0.957 \pm 0.002$ & $0.636 \pm 0.012$ & $0.956 \pm 0.001$ & $0.649 \pm 0.011$ \\
 & $1.5$ & $0.954 \pm 0.001$ & $0.621 \pm 0.022$ & $0.953 \pm 0.001$ & $0.620 \pm 0.012$ \\
\hline
\rule{0pt}{7.2pt}NoduleMNIST3D & $0$ & $0.815 \pm 0.014$ & $0.806 \pm 0.029$ & $0.783 \pm 0.028$ & $0.774 \pm 0.075$ \\
 & $0.5$ & $0.802 \pm 0.024$ & $0.831 \pm 0.008$ & $0.801 \pm 0.012$ & $0.830 \pm 0.011$ \\
 & $1$ & $0.808 \pm 0.014$ & $0.838 \pm 0.010$ & $0.828 \pm 0.015$ & $0.832 \pm 0.008$ \\
 & $1.5$ & $0.790 \pm 0.014$ & $0.822 \pm 0.014$ & $0.789 \pm 0.013$ & $0.824 \pm 0.006$ \\
\hline
\rule{0pt}{7.2pt}FractureMNIST3D & $0$ & $0.683 \pm 0.006$ & $0.506 \pm 0.024$ & $0.670 \pm 0.011$ & $0.481 \pm 0.030$ \\
 & $0.5$ & $0.679 \pm 0.019$ & $0.528 \pm 0.024$ & $0.684 \pm 0.012$ & $0.478 \pm 0.062$ \\
 & $1$ & $0.714 \pm 0.008$ & $0.551 \pm 0.051$ & $0.712 \pm 0.010$ & $0.541 \pm 0.026$ \\
 & $1.5$ & $0.729 \pm 0.013$ & $0.574 \pm 0.063$ & $0.697 \pm 0.029$ & $0.508 \pm 0.063$ \\
\hline
\rule{0pt}{7.2pt}AdrenalMNIST3D & $0$ & $0.833 \pm 0.006$ & $0.805 \pm 0.012$ & $0.818 \pm 0.026$ & $0.791 \pm 0.027$ \\
 & $0.5$ & $0.828 \pm 0.009$ & $0.798 \pm 0.016$ & $0.839 \pm 0.010$ & $0.795 \pm 0.029$ \\
 & $1$ & $0.866 \pm 0.002$ & $0.824 \pm 0.027$ & $0.864 \pm 0.002$ & $0.815 \pm 0.016$ \\
 & $1.5$ & $0.864 \pm 0.002$ & $0.830 \pm 0.019$ & $0.863 \pm 0.005$ & $0.840 \pm 0.015$ \\
\hline
\rule{0pt}{7.2pt}VesselMNIST3D & $0$ & $0.892 \pm 0.004$ & $0.905 \pm 0.004$ & $0.895 \pm 0.013$ & $0.883 \pm 0.029$ \\
 & $0.5$ & $0.913 \pm 0.005$ & $0.903 \pm 0.017$ & $0.906 \pm 0.008$ & $0.897 \pm 0.012$ \\
 & $1$ & $0.924 \pm 0.005$ & $0.911 \pm 0.010$ & $0.922 \pm 0.011$ & $0.918 \pm 0.010$ \\
 & $1.5$ & $0.941 \pm 0.012$ & $0.914 \pm 0.023$ & $0.943 \pm 0.009$ & $0.925 \pm 0.010$ \\
\hline
\rule{0pt}{7.2pt}SynapseMNIST3D & $0$ & $0.800 \pm 0.004$ & $0.799 \pm 0.022$ & $0.774 \pm 0.017$ & $0.781 \pm 0.030$ \\
 & $0.5$ & $0.751 \pm 0.008$ & $0.756 \pm 0.025$ & $0.741 \pm 0.009$ & $0.772 \pm 0.012$ \\
 & $1$ & $0.780 \pm 0.003$ & $0.769 \pm 0.019$ & $0.782 \pm 0.008$ & $0.750 \pm 0.067$ \\
 & $1.5$ & $0.754 \pm 0.003$ & $0.755 \pm 0.028$ & $0.751 \pm 0.004$ & $0.724 \pm 0.093$ \\
\hline
\end{tabular*}
\par\vspace{3pt}
\begin{minipage}{\textwidth}
\fontsize{9}{10}\selectfont
\textit{Notes.} Values are mean $\pm$ sample standard deviation over five runs (seeds $0,\ldots,4$). AUC and ACC denote test-set area under the ROC curve and accuracy. The means match Table~\ref{tab:medmnist-results}.
\end{minipage}
\endgroup
\end{table*}

%% file: references.bib
@article{carriere2020multiparameter,
  title={Multiparameter persistence image for topological machine learning},
  author={Carri{\`e}re, Mathieu and Blumberg, Andrew},
  journal={Advances in Neural Information Processing Systems},
  volume={33},
  pages={22432--22444},
  year={2020}
}

@article{ali2023survey,
  title={A survey of vectorization methods in topological data analysis},
  author={Ali, Dashti and Asaad, Aras and Jimenez, Maria-Jose and Nanda, Vidit and Paluzo-Hidalgo, Eduardo and Soriano-Trigueros, Manuel},
  journal={IEEE Transactions on Pattern Analysis and Machine Intelligence},
  volume={45},
  number={12},
  pages={14069--14080},
  year={2023},
  publisher={IEEE}
}

@article{adams2017persistence,
  title={Persistence images: A stable vector representation of persistent homology},
  author={Adams, Henry and Emerson, Tegan and Kirby, Michael and Neville, Rachel and Peterson, Chris and Shipman, Patrick and Chepushtanova, Sofya and Hanson, Eric and Motta, Francis and Ziegelmeier, Lori},
  journal={Journal of Machine Learning Research},
  volume={18},
  number={8},
  pages={1--35},
  year={2017}
}

@article{JMLR:v16:bubenik15a,
  author  = {Peter Bubenik},
  title   = {Statistical Topological Data Analysis using Persistence Landscapes},
  journal = {Journal of Machine Learning Research},
  year    = {2015},
  volume  = {16},
  number  = {3},
  pages   = {77--102},
  url     = {http://jmlr.org/papers/v16/bubenik15a.html}
}

@article{JMLR:v21:Vipond,
  author  = {Oliver Vipond},
  title   = {Multiparameter Persistence Landscapes},
  journal = {Journal of Machine Learning Research},
  year    = {2020},
  volume  = {21},
  number  = {61},
  pages   = {1--38},
  url     = {http://jmlr.org/papers/v21/19-054.html}
}

@article{Botnan2024SignedBarcodes,
  author  = {Botnan, Magnus Bakke and Oppermann, Steffen
             and Oudot, Steve},
  title   = {Signed Barcodes for Multi-Parameter Persistence via
             Rank Decompositions and Rank-Exact Resolutions},
  journal = {Foundations of Computational Mathematics},
  volume  = {25},
  number  = {5},
  pages   = {1815--1874},
  year    = {2025},
  doi     = {10.1007/s10208-024-09672-9},
  url     = {https://doi.org/10.1007/s10208-024-09672-9}
}

@incollection{botnan2023introductionmultiparameterpersistence,
  author    = {Botnan, Magnus Bakke and Lesnick, Michael},
  title     = {An Introduction to Multiparameter Persistence},
  booktitle = {Representations of Algebras and Related Structures},
  editor    = {Buan, Aslak Bakke and Krause, Henning and
               Solberg, {\O}yvind},
  series    = {EMS Series of Congress Reports},
  volume    = {19},
  pages     = {77--150},
  publisher = {EMS Press},
  year      = {2023},
  doi       = {10.4171/ECR/19/4},
  url       = {https://doi.org/10.4171/ECR/19/4}
}

@article{KimMemoli2021JACT,
  author  = {Woojin Kim and Facundo M{\'e}moli},
  title   = {Generalized persistence diagrams for persistence modules over posets},
  journal = {Journal of Applied and Computational Topology},
  year    = {2021},
  volume  = {5},
  number  = {4},
  pages   = {533--581},
  doi     = {10.1007/s41468-021-00075-1},
  url     = {https://doi.org/10.1007/s41468-021-00075-1}
}

@article{Patel2018GPD,
  author  = {Amit Patel},
  title   = {Generalized persistence diagrams},
  journal = {Journal of Applied and Computational Topology},
  year    = {2018},
  volume  = {1},
  number  = {3-4},
  pages   = {397--419},
  month   = jun,
  doi     = {10.1007/s41468-018-0012-6},
  url     = {https://doi.org/10.1007/s41468-018-0012-6}
}

@article{CarlssonZomorodian2009DCG,
  author  = {Gunnar Carlsson and Afra Zomorodian},
  title   = {The Theory of Multidimensional Persistence},
  journal = {Discrete \& Computational Geometry},
  year    = {2009},
  volume  = {42},
  number  = {1},
  pages   = {71--93},
  doi     = {10.1007/s00454-009-9176-0},
  url     = {https://doi.org/10.1007/s00454-009-9176-0}
}

@article{BotnanOppermannOudotScoccola2022Bottleneck,
  author  = {Botnan, Magnus Bakke and Oppermann, Steffen and
             Oudot, Steve and Scoccola, Luis},
  title   = {On the Bottleneck Stability of Rank Decompositions of
             Multi-Parameter Persistence Modules},
  journal = {Advances in Mathematics},
  volume  = {451},
  pages   = {109780},
  year    = {2024},
  doi     = {10.1016/j.aim.2024.109780},
  url     = {https://doi.org/10.1016/j.aim.2024.109780}
}

@article{lesnick2015interleaving,
  author  = {Lesnick, Michael},
  title   = {The theory of the interleaving distance on multidimensional persistence modules},
  journal = {Foundations of Computational Mathematics},
  volume  = {15},
  number  = {3},
  pages   = {613--650},
  year    = {2015}
}

@misc{ClauseKimMemoli2024GRI,
  title         = {The Generalized Rank Invariant: {M{\"o}bius} invertibility, Discriminating Power, and Connection to Other Invariants},
  author        = {Clause, Nathaniel and Kim, Woojin and M{\'e}moli, Facundo},
  year          = {2024},
  eprint        = {2207.11591},
  archivePrefix = {arXiv},
  primaryClass  = {math.AT},
  doi           = {10.48550/arXiv.2207.11591},
  url           = {https://arxiv.org/abs/2207.11591},
  note          = {arXiv:2207.11591 [math.AT], v5 (Jul 9, 2024), 62 pp.}
}

@inproceedings {lian15nonrigid,
booktitle = {Eurographics Workshop on 3D Object Retrieval},
editor = {I. Pratikakis and M. Spagnuolo and T. Theoharis and L. Van Gool and R. Veltkamp},
title = {{Non-rigid 3D Shape Retrieval}},
author = {Lian, Z. and Zhang, J. and Choi, S. and ElNaghy, H. and El-Sana, J. and Furuya, T. and Giachetti, A. and Guler, R. A. and Lai, L. and Li, C. and Li, H. and Limberger, F. A. and Martin, R. and Nakanishi, R. U. and Neto, A. P. and Nonato, L. G. and Ohbuchi, R. and Pevzner, K. and Pickup, D. and Rosin, P. and Sharf, A. and Sun, L. and Sun, X. and Tari, S. and Unal, G. and Wilson, R. C.},
year = {2015},
publisher = {The Eurographics Association},
DOI = {10.2312/3dor.20151064}
}

@article{medmnistv2,
    title={MedMNIST v2-A large-scale lightweight benchmark for 2D and 3D biomedical image classification},
    author={Yang, Jiancheng and Shi, Rui and Wei, Donglai and Liu, Zequan and Zhao, Lin and Ke, Bilian and Pfister, Hanspeter and Ni, Bingbing},
    journal={Scientific Data},
    volume={10},
    number={1},
    pages={41},
    year={2023},
    publisher={Nature Publishing Group UK London}
}

@inproceedings{medmnistv1,
    title={MedMNIST Classification Decathlon: A Lightweight AutoML Benchmark for Medical Image Analysis},
    author={Yang, Jiancheng and Shi, Rui and Ni, Bingbing},
    booktitle={IEEE 18th International Symposium on Biomedical Imaging (ISBI)},
    pages={191--195},
    year={2021}
}

@misc{wang2026glog,
  title         = {{G-LoG} Bi-filtration for Medical Image Classification},
  author        = {Wang, Qingsong and He, Jiaxing and Hou, Bingzhe and
                   Wu, Tieru and Cao, Yang and Yao, Cailing},
  year          = {2026},
  eprint        = {2602.18329},
  archivePrefix = {arXiv},
  primaryClass  = {cs.CV},
  doi           = {10.48550/arXiv.2602.18329},
  url           = {https://arxiv.org/abs/2602.18329},
  note          = {arXiv:2602.18329}
}

@article{sun2009concise,
  author  = {Sun, Jian and Ovsjanikov, Maks and Guibas, Leonidas},
  title   = {A Concise and Provably Informative Multi-Scale
             Signature Based on Heat Diffusion},
  journal = {Computer Graphics Forum},
  volume  = {28},
  number  = {5},
  pages   = {1383--1392},
  year    = {2009},
  doi     = {10.1111/j.1467-8659.2009.01515.x}
}

@inproceedings{aubry2011wave,
  author    = {Aubry, Mathieu and Schlickewei, Ulrich and Cremers, Daniel},
  title     = {The Wave Kernel Signature: A Quantum Mechanical
               Approach to Shape Analysis},
  booktitle = {2011 IEEE International Conference on Computer Vision Workshops},
  pages     = {1626--1633},
  year      = {2011},
  doi       = {10.1109/ICCVW.2011.6130444}
}

@article{CrawleyBoevey2015,
  author  = {Crawley-Boevey, William},
  title   = {Decomposition of Pointwise Finite-Dimensional Persistence Modules},
  journal = {Journal of Algebra and Its Applications},
  volume  = {14},
  number  = {5},
  pages   = {1550066},
  year    = {2015},
  doi     = {10.1142/S0219498815500668},
  url     = {https://doi.org/10.1142/S0219498815500668}
}

@misc{LesnickWright2015,
  author        = {Lesnick, Michael and Wright, Matthew},
  title         = {Interactive Visualization of 2-D Persistence Modules},
  year          = {2015},
  eprint        = {1512.00180},
  archivePrefix = {arXiv},
  primaryClass  = {math.AT},
  url           = {https://arxiv.org/abs/1512.00180},
  note          = {arXiv:1512.00180}
}

@inproceedings{XinMukherjeeSamagaDey2023GRIL,
  author    = {Xin, Cheng and Mukherjee, Soham and
               Samaga, Shreyas N. and Dey, Tamal K.},
  title     = {{GRIL}: A {$2$}-Parameter Persistence Based
               Vectorization for Machine Learning},
  booktitle = {Proceedings of the 2nd Annual Workshop on Topology,
               Algebra, and Geometry in Machine Learning},
  series    = {Proceedings of Machine Learning Research},
  volume    = {221},
  pages     = {313--333},
  year      = {2023},
  publisher = {PMLR},
  url       = {https://proceedings.mlr.press/v221/xin23a.html}
}

@article{EdelsbrunnerLetscherZomorodian2002,
    author  = {Edelsbrunner, Herbert and Letscher, David and
               Zomorodian, Afra},
    title   = {Topological Persistence and Simplification},
    journal = {Discrete \& Computational Geometry},
    volume  = {28},
    number  = {4},
    pages   = {511--533},
    year    = {2002},
    doi     = {10.1007/s00454-002-2885-2},
    url     = {https://doi.org/10.1007/s00454-002-2885-2}
}

@article{ZomorodianCarlsson2005,
    author  = {Zomorodian, Afra and Carlsson, Gunnar},
    title   = {Computing Persistent Homology},
    journal = {Discrete \& Computational Geometry},
    volume  = {33},
    number  = {2},
    pages   = {249--274},
    year    = {2005},
    doi     = {10.1007/s00454-004-1146-y},
    url     = {https://doi.org/10.1007/s00454-004-1146-y}
}

@article{OtterEtAl2017,
    author  = {Otter, Nina and Porter, Mason A. and Tillmann, Ulrike
               and Grindrod, Peter and Harrington, Heather A.},
    title   = {A Roadmap for the Computation of Persistent Homology},
    journal = {EPJ Data Science},
    volume  = {6},
    number  = {1},
    pages   = {17},
    year    = {2017},
    doi     = {10.1140/epjds/s13688-017-0109-5},
    url     = {https://doi.org/10.1140/epjds/s13688-017-0109-5}
}

@article{mcfaddin2026interleaving,
  title={Interleaving distances, monoidal actions and 2-categories},
  author={McFaddin, Patrick K and Needham, Tom},
  journal={Algebraic \& Geometric Topology},
  volume={26},
  number={1},
  pages={227--281},
  year={2026},
  publisher={Mathematical Sciences Publishers}
}

@article{ayhan2027equivalence,
  title={Equivalence of landscape and erosion distances for persistence diagrams},
  author={Ayhan, Cagatay and Needham, Tom},
  journal={Advances in Applied Mathematics},
  volume={182},
  pages={103161},
  year={2027},
  publisher={Elsevier}
}

@inproceedings{ChazalEtAl2009Proximity,
  author    = {Chazal, Fr{\'e}d{\'e}ric and Cohen-Steiner, David
               and Glisse, Marc and Guibas, Leonidas J.
               and Oudot, Steve Y.},
  title     = {Proximity of Persistence Modules and Their Diagrams},
  booktitle = {Proceedings of the Twenty-Fifth Annual Symposium
               on Computational Geometry},
  pages     = {237--246},
  year      = {2009},
  publisher = {Association for Computing Machinery},
  doi       = {10.1145/1542362.1542407},
  url       = {https://doi.org/10.1145/1542362.1542407}
}

@article{BubenikScott2014,
  author  = {Bubenik, Peter and Scott, Jonathan A.},
  title   = {Categorification of Persistent Homology},
  journal = {Discrete \& Computational Geometry},
  volume  = {51},
  number  = {3},
  pages   = {600--627},
  year    = {2014},
  doi     = {10.1007/s00454-014-9573-x},
  url     = {https://doi.org/10.1007/s00454-014-9573-x}
}

@article{BubenikDeSilvaScott2015,
  author  = {Bubenik, Peter and {de Silva}, Vin and Scott, Jonathan A.},
  title   = {Metrics for Generalized Persistence Modules},
  journal = {Foundations of Computational Mathematics},
  volume  = {15},
  number  = {6},
  pages   = {1501--1531},
  year    = {2015},
  doi     = {10.1007/s10208-014-9229-5},
  url     = {https://doi.org/10.1007/s10208-014-9229-5}
}

@article{BauerLesnick2015,
  author  = {Bauer, Ulrich and Lesnick, Michael},
  title   = {Induced Matchings and the Algebraic Stability of
             Persistence Barcodes},
  journal = {Journal of Computational Geometry},
  volume  = {6},
  number  = {2},
  pages   = {162--191},
  year    = {2015},
  doi     = {10.20382/jocg.v6i2a9},
  url     = {https://doi.org/10.20382/jocg.v6i2a9}
}

@article{BotnanLesnick2018,
  author  = {Botnan, Magnus Bakke and Lesnick, Michael},
  title   = {Algebraic Stability of Zigzag Persistence Modules},
  journal = {Algebraic \& Geometric Topology},
  volume  = {18},
  number  = {6},
  pages   = {3133--3204},
  year    = {2018},
  doi     = {10.2140/agt.2018.18.3133},
  url     = {https://doi.org/10.2140/agt.2018.18.3133}
}

@article{Bjerkevik2021Stability,
  author  = {Bjerkevik, H{\aa}vard Bakke},
  title   = {On the Stability of Interval Decomposable Persistence
             Modules},
  journal = {Discrete \& Computational Geometry},
  volume  = {66},
  number  = {1},
  pages   = {92--121},
  year    = {2021},
  doi     = {10.1007/s00454-021-00298-0},
  url     = {https://doi.org/10.1007/s00454-021-00298-0}
}

@article{BlumbergLesnick2024Stability,
  author  = {Blumberg, Andrew J. and Lesnick, Michael},
  title   = {Stability of 2-Parameter Persistent Homology},
  journal = {Foundations of Computational Mathematics},
  volume  = {24},
  number  = {2},
  pages   = {385--427},
  year    = {2024},
  doi     = {10.1007/s10208-022-09576-6},
  url     = {https://doi.org/10.1007/s10208-022-09576-6}
}

@article{BlumbergLesnick2023Universality,
  author  = {Blumberg, Andrew J. and Lesnick, Michael},
  title   = {Universality of the Homotopy Interleaving Distance},
  journal = {Transactions of the American Mathematical Society},
  volume  = {376},
  number  = {12},
  pages   = {8269--8307},
  year    = {2023},
  doi     = {10.1090/tran/8738},
  url     = {https://doi.org/10.1090/tran/8738}
}

@article{Frosini2013StableComparison,
  author  = {Frosini, Patrizio},
  title   = {Stable Comparison of Multidimensional Persistent
             Homology Groups with Torsion},
  journal = {Acta Applicandae Mathematicae},
  volume  = {124},
  number  = {1},
  pages   = {43--54},
  year    = {2013},
  doi     = {10.1007/s10440-012-9769-0},
  url     = {https://doi.org/10.1007/s10440-012-9769-0}
}

@article{Puuska2020Erosion,
  author  = {Puuska, Ville},
  title   = {Erosion Distance for Generalized Persistence Modules},
  journal = {Homology, Homotopy and Applications},
  volume  = {22},
  number  = {1},
  pages   = {233--254},
  year    = {2020},
  doi     = {10.4310/HHA.2020.v22.n1.a14},
  url     = {https://doi.org/10.4310/HHA.2020.v22.n1.a14}
}

@article{KimMemoli2021Spatiotemporal,
  author  = {Kim, Woojin and M{\'e}moli, Facundo},
  title   = {Spatiotemporal Persistent Homology for Dynamic
             Metric Spaces},
  journal = {Discrete \& Computational Geometry},
  volume  = {66},
  number  = {3},
  pages   = {831--875},
  year    = {2021},
  doi     = {10.1007/s00454-019-00168-w},
  url     = {https://doi.org/10.1007/s00454-019-00168-w}
}

@article{MemoliStefanouZhou2024CupProducts,
  author  = {M{\'e}moli, Facundo and Stefanou, Anastasios
             and Zhou, Ling},
  title   = {Persistent Cup Product Structures and Related Invariants},
  journal = {Journal of Applied and Computational Topology},
  volume  = {8},
  pages   = {93--148},
  year    = {2024},
  doi     = {10.1007/s41468-023-00138-5},
  url     = {https://doi.org/10.1007/s41468-023-00138-5}
}

@inproceedings{ClauseDeyMemoliWang2023MetaDiagrams,
  author    = {Clause, Nate and Dey, Tamal K. and
               M{\'e}moli, Facundo and Wang, Bei},
  title     = {Meta-Diagrams for 2-Parameter Persistence},
  booktitle = {39th International Symposium on Computational
               Geometry (SoCG 2023)},
  series    = {Leibniz International Proceedings in Informatics
               (LIPIcs)},
  volume    = {258},
  pages     = {25:1--25:16},
  year      = {2023},
  publisher = {Schloss Dagstuhl--Leibniz-Zentrum f{\"u}r Informatik},
  doi       = {10.4230/LIPIcs.SoCG.2023.25},
  url       = {https://doi.org/10.4230/LIPIcs.SoCG.2023.25}
}

@inproceedings{NuwagiraEtAl2025TopoMed,
title     = {Topological Machine Learning for Low Data
Medical Imaging},
author    = {Nuwagira, Brighton and Korkmaz, Caner and
Koung, Philmore and Coskunuzer, Baris},
booktitle = {Proceedings of the 4th Machine Learning for
Health Symposium},
series    = {Proceedings of Machine Learning Research},
volume    = {259},
pages     = {824--838},
year      = {2025},
publisher = {PMLR},
url       = {https://proceedings.mlr.press/v259/nuwagira25a.html}
}

@inproceedings{CarriereBlumberg2020,
  author    = {Carri{\`e}re, Mathieu and Blumberg, Andrew J.},
  title     = {Multiparameter Persistence Image for Topological
               Machine Learning},
  booktitle = {Advances in Neural Information Processing Systems},
  volume    = {33},
  pages     = {22432--22444},
  year      = {2020},
  url       = {https://proceedings.neurips.cc/paper/2020/hash/fdff71fcab656abfbefaabecab1a7f6d-Abstract.html}
}

@inproceedings{CignoniEtAl2008MeshLab,
  author    = {Cignoni, Paolo and Callieri, Marco and
               Corsini, Massimiliano and Dellepiane, Matteo and
               Ganovelli, Fabio and Ranzuglia, Guido},
  title     = {{MeshLab}: An Open-Source Mesh Processing Tool},
  booktitle = {Eurographics Italian Chapter Conference},
  editor    = {Scarano, Vittorio and De Chiara, Rosario and Erra, Ugo},
  publisher = {The Eurographics Association},
  year      = {2008},
  pages     = {129--136},
  doi       = {10.2312/LocalChapterEvents/ItalChap/ItalianChapConf2008/129-136}
}

@article{Miller2025Homological,
  author  = {Miller, Ezra},
  title   = {Homological Algebra of Modules over Posets},
  journal = {SIAM Journal on Applied Algebra and Geometry},
  volume  = {9},
  number  = {3},
  pages   = {483--524},
  year    = {2025},
  doi     = {10.1137/22M1516361},
  url     = {https://doi.org/10.1137/22M1516361}
}

@article{KimMemoliStefanou2024,
  author  = {Kim, Woojin and M{\'e}moli, Facundo and Stefanou, Anastasios},
  title   = {Interleaving by Parts: Join Decompositions of
             Interleavings and Join-Assemblage of Geodesics},
  journal = {Order},
  volume  = {41},
  number  = {2},
  pages   = {497--537},
  year    = {2024},
  doi     = {10.1007/s11083-023-09643-9},
  url     = {https://doi.org/10.1007/s11083-023-09643-9}
}

@article{CerriEtAl2013,
  author  = {Cerri, Andrea and Di Fabio, Barbara and Ferri, Massimo
             and Frosini, Patrizio and Landi, Claudia},
  title   = {Betti Numbers in Multidimensional Persistent Homology
             Are Stable Functions},
  journal = {Mathematical Methods in the Applied Sciences},
  volume  = {36},
  number  = {12},
  pages   = {1543--1557},
  year    = {2013},
  doi     = {10.1002/mma.2704},
  url     = {https://doi.org/10.1002/mma.2704}
}

@article{BetthauserBubenikEdwards2022,
  author  = {Betthauser, Leo and Bubenik, Peter and Edwards, Parker B.},
  title   = {Graded Persistence Diagrams and Persistence Landscapes},
  journal = {Discrete \& Computational Geometry},
  volume  = {67},
  number  = {1},
  pages   = {203--230},
  year    = {2022},
  doi     = {10.1007/s00454-021-00316-1},
  url     = {https://doi.org/10.1007/s00454-021-00316-1}
}

@article{DeyKimMemoli2024,
  author  = {Dey, Tamal K. and Kim, Woojin and M{\'e}moli, Facundo},
  title   = {Computing Generalized Rank Invariant for
             {2-Parameter} Persistence Modules via Zigzag
             Persistence and Its Applications},
  journal = {Discrete \& Computational Geometry},
  volume  = {71},
  number  = {1},
  pages   = {67--94},
  year    = {2024},
  doi     = {10.1007/s00454-023-00584-z},
  url     = {https://doi.org/10.1007/s00454-023-00584-z}
}

@article{CorbetEtAl2019Kernel,
  author  = {Corbet, Ren{\'e} and Fugacci, Ulderico and Kerber,
             Michael and Landi, Claudia and Wang, Bei},
  title   = {A Kernel for Multi-Parameter Persistent Homology},
  journal = {Computers \& Graphics: X},
  volume  = {2},
  pages   = {100005},
  year    = {2019},
  doi     = {10.1016/j.cagx.2019.100005},
  url     = {https://doi.org/10.1016/j.cagx.2019.100005}
}

@inproceedings{LoiseauxCarriereBlumberg2023,
  author    = {Loiseaux, David and Carri{\`e}re, Mathieu and
               Blumberg, Andrew J.},
  title     = {A Framework for Fast and Stable Representations of
               Multiparameter Persistent Homology Decompositions},
  booktitle = {Advances in Neural Information Processing Systems},
  volume    = {36},
  pages     = {35774--35798},
  year      = {2023},
  doi       = {10.52202/075280-1553},
  url       = {https://papers.neurips.cc/paper_files/paper/2023/hash/702b67152ec4435795f681865b67999c-Abstract-Conference.html}
}

@inproceedings{LoiseauxEtAl2023SignedMeasures,
  author    = {Loiseaux, David and Scoccola, Luis and
               Carri{\`e}re, Mathieu and Botnan, Magnus Bakke
               and Oudot, Steve},
  title     = {Stable Vectorization of Multiparameter Persistent
               Homology Using Signed Barcodes as Measures},
  booktitle = {Advances in Neural Information Processing Systems},
  volume    = {36},
  pages     = {68316--68342},
  year      = {2023},
  doi       = {10.52202/075280-2987},
  url       = {https://papers.neurips.cc/paper_files/paper/2023/hash/d75c474bc01735929a1fab5d0de3b189-Abstract-Conference.html}
}

@inproceedings{MukherjeeEtAl2026DGRIL,
  author    = {Mukherjee, Soham and Samaga, Shreyas N. and Xin,
               Cheng and Oudot, Steve and Dey, Tamal K.},
  title     = {{D-GRIL}: End-To-End Topological Learning with
               {$2$}-Parameter Persistence},
  booktitle = {42nd International Symposium on Computational
               Geometry (SoCG 2026)},
  series    = {Leibniz International Proceedings in Informatics
               (LIPIcs)},
  volume    = {367},
  pages     = {79:1--79:17},
  editor    = {Ahn, Hee-Kap and Hoffmann, Michael and Nayyeri, Amir},
  publisher = {Schloss Dagstuhl--Leibniz-Zentrum f{\"u}r Informatik},
  address   = {Dagstuhl, Germany},
  year      = {2026},
  doi       = {10.4230/LIPIcs.SoCG.2026.79},
  url       = {https://doi.org/10.4230/LIPIcs.SoCG.2026.79}
}

@article{VipondEtAl2021Tumors,
  author  = {Vipond, Oliver and Bull, Joshua A. and Macklin,
             Philip S. and Tillmann, Ulrike and Pugh,
             Christopher W. and Byrne, Helen M. and Harrington,
             Heather A.},
  title   = {Multiparameter Persistent Homology Landscapes Identify
             Immune Cell Spatial Patterns in Tumors},
  journal = {Proceedings of the National Academy of Sciences},
  volume  = {118},
  number  = {41},
  pages   = {e2102166118},
  year    = {2021},
  doi     = {10.1073/pnas.2102166118},
  url     = {https://doi.org/10.1073/pnas.2102166118}
}

@inproceedings{GarciaRedondoMonodWang2026,
  author    = {Garc{\'i}a-Redondo, In{\'e}s and Monod, Anthea
               and Wang, Qiquan},
  title     = {Confidence Bands for Multiparameter Persistence
               Landscapes},
  booktitle = {Geometric Science of Information},
  editor    = {Nielsen, Frank and Barbaresco, Fr{\'e}d{\'e}ric},
  series    = {Lecture Notes in Computer Science},
  volume    = {16034},
  pages     = {22--31},
  publisher = {Springer},
  address   = {Cham},
  year      = {2026},
  doi       = {10.1007/978-3-032-03921-7_3},
  url       = {https://doi.org/10.1007/978-3-032-03921-7_3}
}

@article{FlammerHuper2026,
  author  = {Flammer, Martina and H{\"u}per, Knut},
  title   = {Spatiotemporal Persistence Landscapes},
  journal = {Journal of Applied and Computational Topology},
  volume  = {10},
  number  = {4},
  pages   = {24},
  year    = {2026},
  doi     = {10.1007/s41468-026-00251-1},
  url     = {https://doi.org/10.1007/s41468-026-00251-1}
}

@misc{FuEtAl2025GeneralizedLandscapes,
  author        = {Fu, Yaoying and Lagoda, Evgeniya and Li, Shiying
                   and Needham, Tom and Ver Hoef, Lander and
                   Weiler, Morgan},
  title         = {Persistent Homology for Labeled Datasets:
                   {Gromov--Hausdorff} Stability and Generalized
                   Landscapes},
  year          = {2025},
  eprint        = {2512.08794},
  archivePrefix = {arXiv},
  primaryClass  = {math.AT},
  url           = {https://arxiv.org/abs/2512.08794},
  note          = {arXiv:2512.08794}
}

@inproceedings{MariaEtAl2014GUDHI,
  author    = {Maria, Cl{\'e}ment and Boissonnat, Jean-Daniel and
               Glisse, Marc and Yvinec, Mariette},
  editor    = {Hong, Hoon and Yap, Chee},
  title     = {The {GUDHI} Library: Simplicial Complexes and
               Persistent Homology},
  booktitle = {Mathematical Software -- ICMS 2014},
  series    = {Lecture Notes in Computer Science},
  volume    = {8592},
  pages     = {167--174},
  publisher = {Springer Berlin Heidelberg},
  year      = {2014},
  doi       = {10.1007/978-3-662-44199-2_28},
  url       = {https://doi.org/10.1007/978-3-662-44199-2_28}
}

@article{PedregosaEtAl2011ScikitLearn,
  author  = {Pedregosa, Fabian and Varoquaux, Ga{\"e}l and
             Gramfort, Alexandre and Michel, Vincent and Thirion,
             Bertrand and Grisel, Olivier and Blondel, Mathieu and
             Prettenhofer, Peter and Weiss, Ron and Dubourg, Vincent
             and Vanderplas, Jake and Passos, Alexandre and
             Cournapeau, David and Brucher, Matthieu and Perrot,
             Matthieu and Duchesnay, {\'E}douard},
  title   = {Scikit-learn: Machine Learning in Python},
  journal = {Journal of Machine Learning Research},
  volume  = {12},
  number  = {85},
  pages   = {2825--2830},
  year    = {2011},
  url     = {https://jmlr.org/papers/v12/pedregosa11a.html}
}

@article{CurryMukherjeeTurner2022,
  author  = {Curry, Justin and Mukherjee, Sayan and Turner, Katharine},
  title   = {How Many Directions Determine a Shape and Other Sufficiency Results for Two Topological Transforms},
  journal = {Transactions of the American Mathematical Society, Series B},
  volume  = {9},
  pages   = {1006--1043},
  year    = {2022},
  doi     = {10.1090/btran/122}
}
